\documentclass{amsart}
\usepackage{amssymb,amsfonts, color,epsf,mathrsfs}
\usepackage [cmtip,arrow]{xy}
\xyoption{all}
\usepackage {pb-diagram,pb-xy}
\usepackage{enumerate}
\usepackage{accents}
\usepackage[noautoscale]{youngtab}
\usepackage{subfigure}
\usepackage[utf8]{inputenc}
\usepackage{todonotes}

\ifx\pdfpageheight\undefined
\usepackage[dvips,colorlinks=true,linkcolor=blue,citecolor=red,%
urlcolor=green]{hyperref}
\usepackage[dvips]{graphicx}
\makeatletter
\edef\Gin@extensions{\Gin@extensions,.mps}
\makeatother
\else
\usepackage[bookmarksopen=false,pdftex=true,breaklinks=true,%
backref=page,pagebackref=true,plainpages=false,%
hyperindex=true,pdfstartview=FitH,colorlinks=true,%
pdfpagelabels=true,colorlinks=true,linkcolor=blue,%
citecolor=red,urlcolor=green,hypertexnames=false%
]%
{hyperref}
\fi
\usepackage{bm}

\usepackage[margin=1.5in]{geometry}
\usepackage{tikz-cd}
\makeatletter
\tikzset{
	column sep/.code=\def\pgfmatrixcolumnsep{\pgf@matrix@xscale*(#1)},
	row sep/.code   =\def\pgfmatrixrowsep{\pgf@matrix@yscale*(#1)},
	matrix xscale/.code=%
	\pgfmathsetmacro\pgf@matrix@xscale{\pgf@matrix@xscale*(#1)},
	matrix yscale/.code=%
	\pgfmathsetmacro\pgf@matrix@yscale{\pgf@matrix@yscale*(#1)},
	matrix scale/.style={/tikz/matrix xscale={#1},/tikz/matrix yscale={#1}}}
\def\pgf@matrix@xscale{1}
\def\pgf@matrix@yscale{1}
\makeatother

\usepackage{amsmath,amssymb,amsfonts}
\newtheorem{theorem}{Theorem}

\newtheorem{corollary}{Corollary}
\newtheorem{proposition}{Proposition}[section]

\newtheorem*{claim*}{Claim}

\newtheorem{problem}{Problem}%
\newtheorem{assumption}{Assumption}%

\newtheorem*{theorem*}{Theorem}
\newtheorem*{corollary*}{Corollary}
\theoremstyle{definition}
\newtheorem{definition}{Definition}[section]
\newtheorem{example}{Example}

\newtheorem{notation}{Notation}

\usepackage{algorithm}
\usepackage{algpseudocode}

\usepackage{tikz}

\algnewcommand\algorithmicinput{\textbf{Input:}}
\algnewcommand\INPUT{\item[\algorithmicinput]}
\algnewcommand\algorithmicoutput{\textbf{Output:}}
\algnewcommand\OUTPUT{\item[\algorithmicoutput]}

\algnewcommand\algorithmicproc{\textbf{Procedure:}}
\algnewcommand\PROCEDURE{\item[\algorithmicproc]}
\algnewcommand\algorithmiccomplexity{\textbf{Complexity:}}
\algnewcommand\COMPLEXITY{\item[\algorithmiccomplexity]}

\newlength{\continueindent}
\usepackage{etoolbox}
\makeatletter
\newcommand*{\ALG@customparshape}{\parshape 2 \leftmargin \linewidth \dimexpr\ALG@tlm+\continueindent\relax \dimexpr\linewidth+\leftmargin-\ALG@tlm-\continueindent\relax}
\apptocmd{\ALG@beginblock}{\ALG@customparshape}{}{\errmessage{failed to patch}}
\makeatother

\newcommand{\PO}{$\texttt{PO}$}
\newcommand{\SDO}{$\texttt{SDO}$}
\newcommand{\LO}{$\texttt{LO}$}
\newcommand{\NO}{$\texttt{NO}$}

\theoremstyle{remark}
\newtheorem{remark}{Remark}

\theoremstyle{observation}

\definecolor{DarkBlue}{rgb}{0,0.1,0.55}

\numberwithin{equation}{section}

\newcommand {\hide}[1]{}
\newcommand {\sign} {\mbox{\bf sign}}
\newcommand{\zero}{\mbox{\bf zero}}

\newcommand {\junk}[1]{}

\newcommand {\R} {\mathrm{R}}

\newcommand {\C}     {\mathrm{C}}

\newcommand {\grad}     {\mbox{\rm Grad}}

\newcommand {\hess}     {\mbox{\rm Hes}}

\newcommand {\Z}  {\mathbb{Z}}

\newcommand {\Q}         {\mathbb{Q}}

\newcommand {\Der} {{\rm Der}}

\newcommand {\la}   {{\langle}}
\newcommand {\ra}   {{\rangle}}
\newcommand {\eps} {{\varepsilon}}

\newcommand {\Sing}      {\mbox{\rm Sing}}
\newcommand {\Ireg}      {\mbox{\rm Ireg}}

\newcommand {\PP}     {\mathbb{P}} 

\newcommand{\card}{\mathrm{card}}
\newcommand{\rank}{\mathrm{rank}}

\def\addots{\mathinner{\mkern1mu
		\raise1pt\vbox{\kern7pt\hbox{.}}
		\mkern2mu\raise4pt\hbox{.}\mkern2mu
		\raise7pt\hbox{.}\mkern1mu}}

\DeclareMathOperator{\trace}{Tr}

\DeclareMathOperator{\order} {ord}

\newcommand{\Crit}{\mathrm{Crit}}
\newcommand{\interior}{\mathrm{int}}

	\title[On the convergence of critical points on real algebraic sets]
	{
		On the convergence of critical points on real algebraic sets and applications to optimization
	}
	
        \author{Saugata Basu}
	\address{Department of Mathematics, 
		Purdue University, West Lafayette, IN 47905, U.S.A.}
	\email{sbasu@purdue.edu}
	
	\author{Ali Mohammad-Nezhad}
	\address{Department of Statistics and Operations Research, University of North Carolina at Chapel Hill, Chapel Hill, NC 27599, U.S.A.}
	\email{alimn@unc.edu}

\begin{document}
\begin{abstract}
Let $F \in \R[X_1,\ldots,X_n]$ and the zero set $V=\zero(\mathcal{P},\R^n)$, where $\mathcal{P}:=\{P_1,\ldots,P_s\} \subset \R[X_1,\ldots,X_n]$ is a finite set of polynomials. We investigate existence of critical points of $F$ on an infinitesimal perturbation $V_{\xi} = \zero(\{P_1-\xi_1,\ldots,P_s-\xi_s\},\R\la \xi \ra^n)$. Our main motivation is to understand the limiting behavior of local minimizers of the log-barrier function (and central paths) in polynomial optimization, whose existence plays a fundamental role, in theory and practice, for modern interior point methods. We establish different sets of conditions that ensure existence, finiteness, boundedness, and non-degeneracy of critical points of $F$ on $V_{\xi}$, respectively. These lead to new conditions for the existence, convergence, and smoothness of central paths of polynomial optimization and its extension to non-linear optimization problems involving definable sets and functions in an o-minimal structure. In particular, for non-linear programs defined by real globally analytic functions, our extension provides a stronger form of the convergence result obtained by Drummond and Peterzil.   

\end{abstract}


\maketitle
\tableofcontents

\section{Introduction}\label{intro}
Let $\R$ be a real closed field, $F \in \R[X_1,\ldots,X_n]$ a polynomial, and $V$ a real 
algebraic set
in $\R^n$, defined by
\begin{align}\label{variety}
V:=\zero(\mathcal{P},\R^n):=\Bigg\{x \in \R^{n} \mid \bigwedge_{P \in \mathcal{P}} P(x)=0\Bigg\},
\end{align}
where $\mathcal{P}:=\{P_1,\ldots,P_s\} \subset \R[X_1,\ldots,X_n]$ is a finite family of polynomials. The study of critical points of a smooth function on a smooth manifold or a singular space is a well-established topic in differential topology~\cite{BR88,GM,M63,Morse,Morse2}. A basic question at the intersection of real algebraic geometry and \textcolor{black}{polynomial optimization (\PO)} (see e.g., Section~\ref{appl_to_polynom_optim}) is the existence and boundedness of critical points (to be defined in Section~\ref{KKT_Conditions_Lagrange}) of $F$ on an infinitesimal perturbation
\begin{align}\label{perturbed_variety}
   \textcolor{black}{ V_{\xi} := \zero(\{P_1-\xi_1,\ldots,P_s-\xi_s\},\R \la \xi \ra^n), \quad \xi := (\xi_1,\ldots,\xi_s)},
\end{align}
where $\xi_1,\ldots,\xi_s > 0$ are infinitesimally small (see Notation~\ref{not:Puiseux} for precise meaning of ``infinitesimal''). If the Jacobian of the polynomials is non-singular at every point of $V$, then critical points of $F$ on $V$ are defined in the sense of Morse Theory~\cite{M63}. \textcolor{black}{Otherwise, we equip $V$ with its canonical stratification
(to be defined in Section~\ref{stratified_Morse}), and define critical points and
non-degeneracy stratumwise; see Definition~\ref{Generalized_Morse_Func}.}

\vspace{5px}
\noindent
Infinitesimal deformation of real algebraic sets has been a key tool for derivation of effective bounds and algorithms in computational algebraic geometry~\cite{Basu1,BPR10,BRi18,Milnor2,BPR95,BPR06,HHS23,BPR99,BR2014,BRMS10,LSW20,RRS00,SS16}. This body of work focuses on applying Morse theoretic and infinitesimal perturbation techniques within algorithmic real algebraic geometry to count the number of critical points of a Morse polynomial on a bounded non-singular algebraic hypersurface. For instance, effective bounds on 
quantifier elimination in the theory of the reals~\cite{BPR95}, 
topological complexity (e.g., sum of Betti numbers of semi-algebraic sets)~\cite{Basu1,BPR10,BRi18,Milnor2}, and computation of roadmaps of a semi-algebraic set~\cite{BPR99,BRMS10,BR2014} leverage computation of critical points of a Morse polynomial on an infinitesimally perturbed algebraic hypersurface (which is non-singular and bounded over $\R$ (see also~\cite{B17} and~\cite[Chapters 12-16]{BPR06})). These techniques have been also applied for the computation of smooth points of a real/complex 
algebraic set~\cite{HHS23,LSW20,RRS00,SS16}.

\vspace{5px}
\noindent
\textcolor{black}{Beyond its role in computational real algebraic geometry, the infinitesimal formulation (e.g., the analysis of critical points of $F$ on the infinitesimally perturbed algebraic set $V_{\xi}$) may
also be interpreted as an exact symbolic perturbation scheme. It
separates distinct perturbation scales through the hierarchy of
infinitesimals, without requiring the numerical choice of extremely
small floating-point parameters, while providing a well-defined limit
map to the original problem.  This viewpoint is compatible with the local
nature of classical sensitivity and perturbation analysis in non-linear
optimization (\NO), where one studies the behavior of solutions and critical
points under sufficiently small changes in the problem data. Such
symbolic perturbations can also be useful for regularizing degenerate
polynomial multiplier systems and for analyzing degeneracies arising in
semi-definite optimization (\SDO)~\cite{BM22,BM24}. The infinitesimal framework
is especially natural for studying log-barrier critical points and
central paths in \PO; see
Section~\ref{appl_to_polynom_optim} and Remark~\ref{inf_approach}. We explain this connection in
Section~\ref{central_path_manifold_correspondence}.}

\subsection{Importance of the log-barrier function}
Although the log-barrier function of \NO~(which is generally non-convex) has received considerably less attention in the optimization literature compared to its convex counterpart (e.g.,~\cite{Kl02,GP02,Guler93,KMNY91,MZ98,NN94,Ren01,CRT06}), the existence of its local minimizers remains a problem of central importance for both the theoretical foundations and practical performance of modern interior point methods (IPMs) (see e.g.,~\cite{FM90,FGW02,GOST01,NW06,VS99,WO2002}). For instance, globally convergent IPMs (such as trust-region IPMs~\cite[Section~19.5]{NW06}) depend on the existence of local minimizers of the log-barrier function. In particular, in the absence of a central path (which is a special trajectory of local minimizers of the log-barrier function), the search directions in primal-dual IPMs may be ill-defined (see~\cite[Page~569]{NW06} and Remark~\ref{finiteness_condition}), and the super-linear convergence of these methods fail without convergence of the derivatives of the central path~\cite{GOST01}. This problem has been widely studied under convexity assumptions (see e.g.,~\cite{Kl02,Guler93,KMNY91,NN94,CRT06}) but not specifically for non-convex~\NO.  In convex optimization, the existence of a central path typically relies on Slater’s condition (see e.g.,~\cite[Section~5.4]{CRT06},~\cite[Section~3.1]{Kl02}, and~\cite[Chapter~3]{NN94}). However, this condition does not generally ensure the existence of a central path in the non-convex case of \NO~(see Example~\ref{non_existence}).  

\vspace{5px}
\noindent
It is worth noting that only very few papers have explored existential conditions on the central path of non-convex \NO. In~\cite[Section~7]{FGW02}, the authors state that

\vspace{5px}

``\textit{The theoretical results for general nonconvex problems are weaker and mostly
asymptotic ... In spite of this potential concern, interior methods successfully and efficiently solve large nonconvex non-linear
programming problems every day, but the possibility of strange or even pathological
behavior should not be ignored}".

\vspace{5px}
On the other hand, the symbolic computation of optimal solutions to \PO~ has long been a classical topic in algorithmic real algebraic geometry (see e.g.,~\cite[Algorithm~14.9]{BPR06}). However, this line of work is purely algebraic and does not take the central path into account. 

\subsection{Main problems}
Motivated by the importance of the log-barrier function and central paths to~\NO, we would like to address the following problems, which are of independent interest in algorithmic real algebraic geometry as well.

\begin{problem}[Existence]\label{existence_critical_path}
What are the conditions (necessary or sufficient) for the existence of critical points of $F$ on $V_{\xi}$? 
\end{problem}
\begin{problem}[Boundedness]\label{boundedness_critical_path}
What are the conditions (necessary or sufficient) for the boundedness of critical points of $F$ on $V_{\xi}$, if there exists any? 
\end{problem}
\noindent
When it comes to central paths of \PO, it is important that a central path is smooth and behaves smoothly in a neighborhood of its limit point. The smoothness of central paths at the limit point plays a central role in the convergence analysis of primal-dual IPMs. 
\begin{problem}[Existence, convergence, and smoothness]\label{smoothness_critical_path}
What are the conditions (necessary or sufficient) for the existence, convergence, and smoothness of a central path at the limit point? 
\end{problem}

To the best of our knowledge, our paper tackles Problems~\ref{existence_critical_path}-\ref{smoothness_critical_path} listed above,
in the context of \PO~without any added convexity assumptions 
for the first time.

\begin{remark}
One early usage of standard algebraic geometry techniques for the convergence of a central path can be found in the work of Kojima et al.~\cite{KMNY91}. They proved the convergence of a bounded central path of linear complementarity problems using a result of Hironaka~\cite{H75} on triangulation of real algebraic sets. Subsequent work by~\cite{HKR02} and~\cite{GP02} leveraged the Curve Selection Lemma~\cite{Milnor68} and the Monotonicity Theorem~\cite{Dries} in the context of o-minimal geometry (see Sections~\ref{central_path_background}-\ref{o-minimal_optimization}) to prove the convergence of a bounded central path of \SDO~and convex \SDO, respectively. However, compared to our paper, the scope of problems and algebro-geometric techniques considered in these studies is too limited to fully address Problems~\ref{existence_critical_path}-\ref{smoothness_critical_path}. For instance, the problems considered in~\cite{GP02,HKR02,KMNY91} are confined to convex optimization, and the techniques they employ do not appear applicable to establishing the existence or analyticity of a central path in non-convex settings.  
\end{remark}

\hide{
It is worth noting that the existential properties of the central path in non-convex \NO~have received little attention in the literature. In contrast, the symbolic computation of optimal solutions to \PO~ has been a classical topic in algorithmic real algebraic geometry (see, e.g.,~\cite[Algorithm 14.9]{BPR06}). However,
the notion of central path does not arise in the algebraic approach.
To the best of our knowledge, our work is the first to address 
these problems 
in the context of general \PO without requiring any convexity assumptions.
}

\section{Main results}
Our main results are the following. In Section~\ref{sec:critical_points_on_V_xi}, we establish conditions under which the set of critical points of $F$ on $V_{\xi}$ is non-empty, finite, and the critical points are bounded and non-degenerate (Theorems~~\ref{bounded_fibers}-\ref{convergence_of_critical_points}). We also characterize the limit of critical points. In Section~\ref{appl_to_polynom_optim}, we apply Theorems~~\ref{bounded_fibers}-\ref{convergence_of_critical_points} to the following special case: we formulate local minimizers of the log-barrier function as a special case of the problem in Section~\ref{sec:critical_points_on_V_xi} and derive sufficient conditions for the existence, convergence, and smoothness of a central path of \PO~(Theorems~\ref{critcal_path_existence}-\ref{analytic_reparam_semi-algebraic}). In Section~\ref{definable_central_path}, we extend our results for central paths to \NO~problems defined by definable functions in an (polynomially bounded) o-minimal structure (Theorems~\ref{sufficient_conditions_existence_definable}-\ref{thm:analyticity_definable_func}).

\vspace{5px}
\noindent
\textcolor{black}{
Theorem~\ref{convergence_of_critical_points} and Theorems~\ref{limiting_behavior_of_critical_path} and~\ref{analytic_reparam_semi-algebraic} constitute the principal results of the paper. Theorem~\ref{convergence_of_critical_points} is the main geometric result: it characterizes the limit
of a bounded critical point on an infinitesimally perturbed algebraic
hypersurface in terms of the canonical Whitney stratification of the limiting algebraic set. From the optimization perspective,
Theorems~\ref{limiting_behavior_of_critical_path}-\ref{analytic_reparam_semi-algebraic} are the principal consequences of this geometric
framework. Theorem~\ref{limiting_behavior_of_critical_path} characterizes the limit of
every bounded central path as a stratified critical point of the
algebraic boundary and provides an effective H\"older convergence rate. Theorem~\ref{analytic_reparam_semi-algebraic} shows that a power reparametrization of a central path restores $\mathcal{C}^{\infty}$-smoothness, and over \(\mathbb{R}\) analyticity, at the
limit, with an effective bound on the required ramification exponent. Theorems~\ref{limiting_behavior_of_definable_critical_path}-\ref{thm:analyticity_definable_func} provide corresponding extensions to definable \NO~in an o-minimal structure.}

\begin{remark}
Although the central path is a specific trajectory among local minimizers of the log-barrier function, we state existence and convergence conditions in terms of the central path for simplicity. The conclusions, however, remain valid for all local minimizers of the log-barrier function.
\end{remark}

\subsection{Critical points of $F$ on $V_{\xi}$}\label{sec:critical_points_on_V_xi}
\subsubsection{Existence of critical points}\label{sec:existence_critical}
First, we prove conditions that ensure boundedness of $V_{\xi}$ for sufficiently small $\xi > 0$. 
In order to make the notion of 
``sufficiently small'' more precise, we will make use of 
certain non-archimedean extensions of the given real closed field $\R$
containing the coefficients of the given polynomials. In this non-archimedean extension, the various $\xi_i$'s will be infinitesimal elements, which are positive but smaller than all positive elements of $\R$ (we write $0 < \eps \ll 1$ to denote that $\eps$ is a positive infinitesimal over the ground field $\R$). It follows from the Tarski-Seidenberg Transfer Principle (see for example~\cite[Chapter 2]{BPR06}) 
that properties that we state about $V_\xi$ (as long as they are expressible by a first-order formula in the language of ordered fields), for $\xi$ infinitesimal,
also holds if we substitute sufficiently small positive values for
$\xi$.

\begin{notation}[Algebraic Puiseux series and $\lim_\xi$ map]
\label{not:Puiseux}
For any real closed field $\R$, we denote by $\R\la\eps\ra$ the real closed field of algebraic Puiseux series in $\eps$ with coefficients in $\R$. In the unique ordering of the real closed field $\R\la\eps\ra$, $0 < \eps \ll 1$. In other words, $\eps$ is positive but smaller than all positive elements of $\R$ (precise definitions appear in Section~\ref{sec:Puiseux_series}).

We say that $x \in \R\la \eps\ra$ is bounded over $\R$, if there exists an element $a \in \R, a > 0$, such that $|x| < a$. We denote the set of 
elements of $\R\la\eps\ra$ which are bounded over $\R$ by $\R\la\eps\ra_b$.
It is easy to see that $\R\la\eps\ra_b$ is a sub-ring of $\R\la\eps\ra$,
called the valuation ring of $\R\la\eps\ra$, and there is a well-defined
ring homomorphism $\lim_\eps: \R\la\eps\ra_b \rightarrow \R$, defined by
setting $\eps$ to $0$ in the corresponding Puiseux series.

    More generally, we denote by $\R\la\xi_1,\ldots,\xi_s\ra$ the field
    $\R\la\xi_1\ra\cdots\la\xi_s\ra$. Notice that in the ordering of the real closed field $\R\la\xi_1,\ldots,\xi_s\ra$,
    $0 < \xi_s \ll \cdots \ll \xi_1 \ll 1$ and we say that
    $\xi_{i+1}$ is \emph{infinitesimal} with respect to the elements of
    $\R\la\xi_1,\ldots,\xi_i\ra$.

We denote by $\R\la\xi\ra_b$ the sub-ring of elements of 
$\R\la\xi\ra := \R\la\xi_1,\ldots,\xi_s\ra$ which are bounded over $\R$,
and by $\lim_\xi: \R\la\xi\ra_b \rightarrow \R$ the map defined 
as the composition of $\lim_{\xi_1}, \ldots, \lim_{\xi_s}$ 
restricted to $\R\la\xi\ra_b$.
\end{notation}

\begin{remark}[Infinitesimals]
\label{infinitesimal_interpretation}
\textcolor{black}{The extension $\R\langle \xi_1,\ldots,\xi_s\rangle$
provides an exact algebraic encoding of a hierarchy of sufficiently small
positive perturbations
\begin{align*}
0<\xi_s\ll\cdots\ll\xi_1\ll 1.
\end{align*}
Consequently, conclusions obtained using the Tarski--Seidenberg Transfer
Principle can be interpreted over $\R$ by first choosing $\xi_1>0$ sufficiently small, then choosing $\xi_2>0$ sufficiently
small relative to $\xi_1$, and continuing successively. This is more
specific than merely requiring all components of
\(\xi=(\xi_1,\ldots,\xi_s)\) to be simultaneously small.}  
\end{remark}

\begin{remark}[Working over a real closed field]
\textcolor{black}{Although optimization problems and algorithms are primarily concerned
with real numbers, we choose to work over an arbitrary real closed field $\R$. Readers interested only in optimization over $\mathbb{R}$ may
take $\R=\mathbb{R}$ throughout. The more general setting is useful
because the infinitesimal extensions $\R\langle \xi_1,\ldots,\xi_s\rangle$
remain real closed, allowing us to use the Tarski-Seidenberg Transfer
Principle and the standard algebraic and semi-algebraic tools over real
closed fields without introducing additional Archimedean assumptions.}    
\end{remark}
\noindent
In our theorems (e.g., Theorems~\ref{bounded_fibers}, ~\ref{constrained_real_isolated_critical_points}-\ref{constrained_real_non-degenerate_critical_points}), we will often need to impose conditions on projective zeros, ``non-singularity" of projective
algebraic sets (see Definition~\ref{non_singular_point}), or non-degeneracy of ``projective critical points" (see Definition~\ref{projective_KKT_point}) to establish boundedness or finiteness results. For this purpose, we need to introduce some more notation, namely that of projective space over $\R$ as well as
$\C:=\R[i]$, the algebraic closure of $\R$. 

\begin{notation}
We denote by $\mathbb{P}_{n}(\mathrm{K})$ the $n$-dimensional \textit{projective space} over a field $\mathrm{K}$, where $\mathrm{K}$ could be $\R$ and $\C$.  
\end{notation}
\begin{theorem}\label{bounded_fibers}
Let $V \subset \R^n$ and $V_{\xi} \subset \R\la\xi\ra^n$,
(where $\xi = (\xi_1,\ldots,\xi_s)$)
be defined by Eqns.~\eqref{variety} and~\eqref{perturbed_variety}.
 Let $V^H=\zero\big(\mathcal{P}^H,\mathbb{P}_n(\R)\big)$, where 
$\mathcal{P}^H = \{P_1^H,\ldots, P_s^H\}$, and 
$P_i^H \in \R[X_0,\ldots,X_n]$ is the homogenization of $P_i$.
Suppose that the $P_i$'s have no common real zero at infinity, i.e., 
\begin{align}\label{no_point_at_infinity}
V^H \cap \zero(X_0,\mathbb{P}_n(\R)) = \emptyset.
\end{align}
Then $V_{\xi}$ is bounded over $\R$. 
\end{theorem}
\noindent
\begin{remark}\label{no_point_infinity_condition}
In particular, the condition ~\eqref{no_point_at_infinity} implies that $V$ is bounded. However, we should indicate that~\eqref{no_point_at_infinity} is an integral part of Theorem~\ref{bounded_fibers} and cannot be replaced by the weaker condition ``$V$ being bounded". For instance, $V=\zero(X^2 + (XY-1)^2,\mathbb{R}^2)$ is obviously empty and thus bounded, but $V_{\xi}$ is unbounded over $\mathbb{R}$. It is easy to check $V^H \cap \zero(X_0,\mathbb{P}_n(\mathbb{R}))$ is non-empty for this example. 
\end{remark}
\noindent
Given a polynomial $F \in \R[X_1,\ldots,X_n]$, Theorem~\ref{bounded_fibers} gives rise to a sufficient condition for existence and boundedness of critical points of $F$ on $V_{\xi}$ (see Corollary~\ref{bounded_critical_points}). To that end, we need a ``non-singularity" condition, defined below, that ensures that $V_{\xi}$ remains non-empty. We also introduce its projective counterpart that will be needed later for Theorem~\ref{constrained_real_non-degenerate_critical_points}. 
\begin{definition}[Non-singular point]\label{non_singular_point}
Let $\mathcal{P}:=\{P_1,\ldots,P_s\} \subset \mathrm{K}[X_1,\ldots,X_n]$ where $\mathrm{K} = \R$ or $\mathrm{K} =\C$, and $V=\zero(\mathcal{P},\mathrm{K}^n)$. A point $x=(x_1,\ldots,x_n) \in V$ is called a \textit{non-singular} zero of $\mathcal{P}$ if the Jacobian matrix 
\begin{align*}
J(\{P_1,\ldots,P_s\})(x):=\Bigg[\frac{\partial P_i}{\partial X_j}(x)\Bigg]_{\substack{i=1,\ldots,s\\ j=1,\ldots,n}}
\end{align*}
at $x$ has the rank of $s$. Otherwise, $x$ is called a \textit{singular} zero of $\mathcal{P}$. The real algebraic set $V$ is called non-singular if all zeros of $\mathcal{P}$ are non-singular. The set of all singular zeros of $\mathcal{P}$ is denoted by $\Sing(\mathcal{P})$. 

\vspace{5px}
\noindent
Now, consider the projective algebraic set $V=\zero(\mathcal{P},\mathbb{P}_n(\mathrm{K}))$, where  $\mathcal{P}:=\{P_1,\ldots,P_s\} \subset \mathrm{K}[X_0,\ldots,X_n]$ is a set of homogeneous polynomials. A point $x=(x_0:\ldots:x_n) \in V$ is called a non-singular projective zero of $\mathcal{P}$ if 
\begin{align*}
\rank\Bigg(\Big[\cfrac{\partial P_i}{\partial X_j}(x)\Big]_{\substack{i=1,\ldots,s\\ j=0,\ldots,n}}\Bigg) = s.
\end{align*}
\end{definition} 
\noindent
The notion of a critical point will be frequently used in this paper (and also in the context of central paths of~\PO). Therefore, for the ease of exposition, we introduce the following notation. 

\begin{notation}
Given a canonically stratified real algebraic set $V \subset \R^n$ and a polynomial $F \in \R[X_1,\ldots,X_n]$, we will denote by $\Crit(V,F)$ the set of critical points of $F$ on $V$ \textcolor{black}{stratumwise with respect to the canonical stratification} (see Section~\ref{stratified_Morse}). 
\end{notation}
\noindent

\begin{corollary}\label{bounded_critical_points}
Let $F \in \R[X_1,\ldots,X_n]$
and suppose that $V$ satisfies the conditions of Theorem~\ref{bounded_fibers} and contains a non-singular zero of $\mathcal{P}$. Then $\Crit(V_\xi,F) \neq \emptyset$ and $\Crit(V_\xi,F)\subset \R\la\xi\ra_b^n$
(i.e. the set of critical points of $F$ on $V_\xi$ is non-empty and 
bounded over $\R$).
\end{corollary}

\hide{
\begin{notation}[Sign condition]
Let $\R$ be a real closed field. Given a finite family $\mathcal{P} \subset \R[X_1,\ldots,X_{n}]$, a \textit{sign condition} on $\mathcal{P}$ is an element of $\{-1,0,1\}^{\mathcal{P}}$, i.e., a mapping $\mathcal{P} \to \{-1,0,1\}$. A \textit{weak sign condition} on $\mathcal{P}$ is an element of $\{\{0\},\{0,-1\},\{0,1\}\}^{\mathcal{P}}$. 
\end{notation}}

\hide{
\begin{theorem}\label{bounded_connected_component}
Let $\mathcal{P} \subset \R[X_1,\ldots,X_n]$ be a finite set and
$S\subset \R^n$ be a $\mathcal{P}$-closed semi-algebraic set defined by 
a $\mathcal{P}$-closed formula $\Phi$.
Let $\xi = (\xi_1,\ldots,\xi_s)$ and 
$S_{\xi} \subset \R\langle \xi \rangle^n$ be the $\mathcal{P}_{\xi}$-closed semi-algebraic set, 
with $\mathcal{P}_{\xi}:=\{P_i \pm \xi \mid 1 \leq i \leq s\}$, 
defined by the $\mathcal{P}_\xi$-closed formula
$\Phi_\xi$ obtained from $\Phi$ by replacing
each occurrence of $P_i \leq 0$ by $P_i - \xi_i \leq 0$ and 
$P_i \geq 0$ by $P_i + \xi \geq 0$.
 Let $x \in S$, $D$ the semi-algebraically connected component of $S$ containing $x$, and suppose that $D$ is bounded. 
 Let $x_\xi  \in S_{\xi}$ be such that $\lim_{\xi} x_\xi  = x$, and $D_{\xi}$ be the semi-algebraically connected component of $S_{\xi}$ that contains $x_\xi$. Then $D_{\xi}$ is bounded over $\R$, and $\lim_{\xi} D_{\xi} \subset D$.
\end{theorem}}
\noindent
In many applications (e.g., central paths of~\PO), the conditions of Theorem~\ref{bounded_fibers} might seem too restrictive. We prove a variant of Theorem~\ref{bounded_fibers}, which obviates the need for the extra condition in Theorem~\ref{bounded_fibers}, but still guarantees the existence of a bounded semi-algebraically connected component of $V_{\xi}$. 
\begin{theorem}\label{bounded_connected_component}
Let $V \subset \R^n$ and $V_{\xi} \subset \R\la\xi\ra^n$,
(where $\xi = (\xi_1,\ldots,\xi_s)$)
be defined by Eqns.~\eqref{variety} and~\eqref{perturbed_variety}. Let $x \in V$, $D$ the semi-algebraically connected component of $V$ containing $x$, and suppose that $D$ is bounded. 
 Let $x_\xi  \in V_{\xi}$ be such that $\lim_{\xi} (x_\xi)  = x$, and $D_{\xi}$ be the semi-algebraically connected component of $V_{\xi}$ that contains $x_\xi$. Then $D_{\xi}$ is bounded over $\R$, and $\lim_{\xi} (D_{\xi}) \subset D$.
\end{theorem}
\noindent
Analogous to Theorem~\ref{bounded_fibers}, Theorem~\ref{bounded_connected_component} leads to conditions for existence and boundedness of critical points of $F$ on $V_{\xi}$.   

\begin{corollary}\label{existence_critical_points}
Let $F\in \R[X_1,\ldots,X_n]$,
and suppose that $V$ satisfies the conditions of Theorem~\ref{bounded_connected_component} and has a non-singular zero of $P$ 
belonging to
a bounded semi-algebraically connected component of $V$. Then $\Crit(V_\xi,F) \cap \R\la\xi\ra_b^n \neq \emptyset$
(i.e. $V_\xi$ has a critical point of $F$ which is bounded over $\R$).   
\end{corollary}

\hide{
\begin{remark}
We note that  Theorem~\ref{bounded_connected_component} no longer holds if the quantifier-free formula that defines $S$ contains strict inequalities or inequations. For example, consider the quantifier-free formula $\{X_1^2 + X_2^2 = 0\} \cup \{X^2 < 0\}$ whose realization over $\R^2$ is $\{(0,0)\}$. However, the realization of $\{X_1^2 + X_2^2 = 0\} \cup \{X^2 < \xi\}$ over $\mathbb{R}\langle \xi \rangle^2$ is unbounded, although it contains the real point $(\xi,\xi)$ such that $\lim_{\xi}(\xi,\xi)=(0,0)$. 
\end{remark}}
\begin{remark}
Although Theorems~\ref{bounded_fibers} and~\ref{bounded_connected_component} are valid only for infinitesimal $\xi$, they are closely related to a classic result on perturbation (not necessarily infinitesimal) of bounded convex sets~\cite[Theorem~24]{FM90}, which has important implications for convex optimization: if $S=\{x \mid g_i(x) \ge 0\}$ defined by concave functions $g_i$ is bounded, then $\{x \mid g_i(x) \ge -\varepsilon_i\}$ remains bounded for all $\varepsilon_i \ge 0$. The example in Remark~\ref{no_point_infinity_condition} already clarifies that a word-for-word extension of the above classic result
to possibly non-concave polynomial functions (even with infinitesimal perturbation) is not possible without additional conditions.   
\end{remark}

\subsubsection{Isolated critical points}\label{sec:isolated_critical}
Corollaries~\ref{bounded_critical_points}-\ref{existence_critical_points} only ensure the existence of bounded critical points on $V_{\xi}$. We establish conditions that guarantee that $\Crit(V_\xi,F)$ has isolated (non-degenerate) critical points. To that end, we leverage a projective version of the Lagrange multiplier equations (see Definition~\ref{projective_KKT_point}). Lagrange multiplier equations are well-known in the theory of constrained \NO~(see e.g.,~\cite{G10}). \textcolor{black}{For the purposes of this paper, the relevant consequence of Theorems~\ref{constrained_real_isolated_critical_points} and~\ref{constrained_real_non-degenerate_critical_points} is that they provide sufficient conditions for the isolation of critical points on the perturbed hypersurfaces.}

\hide{
\begin{notation}
Given a polynomial $P \in \R[X_1,\ldots,X_n]$, the gradient vector of $P$ is defined as
\begin{align*}
  \grad(P):=\Big(\cfrac{\partial P}{\partial X_1},\ldots,  \cfrac{\partial P}{\partial X_n}\Big).
\end{align*}
\end{notation}

\begin{theorem}\label{constrained_real_non-degenerate_critical_points}
Let $F, P \in \R[X_1,\ldots,X_n]$, $V=\zero(P,\R^n)$, and $V_\xi=\zero(P-\xi,\R\la \xi \ra^n)$. Suppose that $\grad(F)$ and $\grad(P)$ are linearly independent, and there exists $\xi \in \C$ such that $\xi F + (1-\xi) P$ has only complex isolated (non-degenerate) projective critical points (see Section~\ref{sec:projective_critical}). Then all critical points of $F$ on $V_{\xi}$ are isolated (non-degenerate).
\end{theorem}

\begin{remark}
Note that assuming linear independence of $F$ and $P$ in Theorem~\ref{constrained_real_non-degenerate_critical_points} is not a big loss of generality. This assumption only eliminates the trivial case that every point of $\zero(P-\xi,\R \la \xi \ra^n)$ is a critical point of $F$ on $V_{\xi}$.
\end{remark}
}
 
\begin{theorem}\label{constrained_real_isolated_critical_points}
Let $F, P \in \R[X_1,\ldots,X_n]$, $V=\zero(P,\R^n)$, and $V_\xi=\zero(P-\xi,\R\la \xi \ra^n)$. Suppose that for some $c \in \C$, the set of complex projective Lagrange critical pairs of $F$ in $\mathbb{P}_n(\C) \times \mathbb{P}_1(\C)$ on $V_{c} := \zero(P-c,\C^n)$ is finite. Then $\Crit(V_{\xi},F)$ is finite.
\end{theorem}
\noindent
We show that if we incorporate non-degeneracy of critical points of $F$ on $V_{c}$, then all critical point of $F$ on $V_{\xi}$ must be non-degenerate.

\begin{theorem}\label{constrained_real_non-degenerate_critical_points}
Let $F, P \in \R[X_1,\ldots,X_n]$, $V=\zero(P,\R^n)$, and $V_\xi=\zero(P-\xi,\R\la \xi \ra^n)$. Suppose that for some $c \in \C$, all complex projective Lagrange critical pairs of $F$ in $\mathbb{P}_n(\C) \times \mathbb{P}_1(\C)$ on $V_{c} = \zero(P-c,\C^n)$ are non-singular. Then all critical points of $F$ on $V_{\xi}$ are non-degenerate (and thus $\Crit(V_{\xi},F)$ is finite).
\end{theorem}

\begin{remark}
\textcolor{black}{Although the notion of a complex projective Lagrange critical pair relies
on non-singularity of $V_c$, we should clarify that the conditions of
Theorems~\ref{constrained_real_isolated_critical_points}-\ref{constrained_real_non-degenerate_critical_points}
are not vacuously true (see Examples~\ref{fintely_many_critical_points}-\ref{non-degenerate_critical_points}). In fact, as shown in the
proof of Theorem~\ref{constrained_real_isolated_critical_points}, the
Classical Sard Theorem over $\mathbb C$ (e.g.,~\cite[Theorem~6.10]{L13}) and Chevalley's Theorem~\cite[I.8, Corollary~2]{Mum99}, together with the Lefschetz principle~\cite[Theorem~1.26]{BPR06},
imply that for all but finitely many $c\in\C$, $V_c$ is non-singular.
In particular, $V_\xi$ is non-singular, which can be independently
obtained from the Semi-algebraic Sard Theorem~\cite[Theorem~5.56]{BPR06}.}
\end{remark}

\begin{example}\label{fintely_many_critical_points}
Let $F = X_1$ and $P = X_1^3 - X_2^2$. Then we can check that $F$ has only non-degenerate critical points on $V_\xi$, and all conditions of Theorem~\ref{constrained_real_isolated_critical_points} hold. More precisely, for any $c \in \C\setminus\{0\}$, $X_1^3 - X_2^2-c$ has no singular complex zero. The complex projective Lagrange critical pairs are the zeros of 
\begin{align*}
    \{U_0X_0^2-3U_1X^2_1, 2U_1X_2,X_1^3-X_2^2X_0 - c X_0^3\},
\end{align*}
which are $((0:0:1),(1:0))$, and $((1:t:0),(3t^2:1))$, where $t$ is a complex root of $t^3 - c = 0$. Thus,  
for each $c \in \C\setminus \{0\}$, $F$ has finitely many complex projective Lagrange critical pairs in $\mathbb{P}_2(\C) \times \mathbb{P}_1(\C)$ on $V$, although some of them are singular. Therefore, the conditions of Theorem~\ref{constrained_real_non-degenerate_critical_points} do not hold in this case. 
\end{example}
\begin{example}\label{non-degenerate_critical_points}
Let $F = X_1+X_2$ and $P = X_1X_2$. We can check that $F$ has only non-degenerate critical points on $V_\xi$, and all conditions of Theorem~\ref{constrained_real_non-degenerate_critical_points} hold. More precisely, for any $c \in \C\setminus\{0\}$, $X_1X_2-c$ has no singular complex zero. The complex projective Lagrange critical pairs are the zeros of 
\begin{align*}
    \{U_0X_0-U_1X_2,  U_0X_0 - U_1X_1,X_1X_2 - c X_0^2\},
\end{align*}
which are $((0:1:0),(1:0))$, $((0:0:1),(1:0))$, and $((1:t: t),(t:1))$, where $t$ is a complex root of $t^2 - c = 0$. All these complex projective Lagrange critical pairs are non-singular for $c \in \C \setminus\{0\}$.
\end{example}
\noindent
The conditions of Theorems~\ref{constrained_real_isolated_critical_points}-\ref{constrained_real_non-degenerate_critical_points} guarantee 
that $\Crit(V_{\xi},F)$ is a finite set (but still, it could be empty).
Additionally, if we require $F$ to have a non-degenerate critical point on a non-singular $V$ (see Section~\ref{sec:KKT_non-singular}), then we can also guarantee that $\Crit(V_{\xi},F) \cap \R\la\xi\ra_b^n$ is non-empty.

\begin{theorem}\label{Morse_on_smooth_hypersurface}
Let $V \subset \R^n$ and $V_{\xi} \subset \R\la\xi\ra^n$,
(where $\xi = (\xi_1,\ldots,\xi_s)$) be defined by Eqns.~\eqref{variety} and~\eqref{perturbed_variety}, and assume that $V$ is non-singular. Further, let $\bar{x}$ be a non-degenerate critical point of $F$ on $V$. Then there exists a non-degenerate critical point $x_\xi \in \Crit(V_\xi,F) \cap \R\la\xi\ra_b^n$, and $\lim_{\xi} (x_\xi) = \bar{x}$. 
\end{theorem}

\begin{remark}
Although Theorem~\ref{Morse_on_smooth_hypersurface} is stated under non-singularity of $V$, the result is still valid locally, where we only need a non-singular zero of $\mathcal{P}$ (in fact, the existence of a non-singular zero of $\mathcal{P}$ is enough for $V_{\xi}$ to be non-empty (see Proposition~\ref{convergence_of_varieties})) that is a non-degenerate critical point of $F$ on $V$, with respect to a ``local stratification" of $V$. However, this local approach would create an unwieldy statement, which we prefer to avoid. 
\end{remark}
\noindent
It is important to note that if any of the conditions of Theorem~\ref{Morse_on_smooth_hypersurface} are omitted, then the existence of an isolated critical point is no longer guaranteed, as shown by Examples~\ref{nonMorse_smooth_isolated_critical}-\ref{Morse_singular_co_critical}.
\begin{example}\label{nonMorse_smooth_isolated_critical}
Consider the polynomial $F = X_1^3 + X_1X_2^2$ on the non-singular hypersurface $V = \zero(X_2 - X_1,\mathbb{R}^2)$, which has an isolated degenerate critical point $(0,0)$. However, from the critical conditions~\eqref{critical_point_conditions_on_manifold} we get
\begin{align*}
\cfrac{\partial F}{\partial X_1}  =  6X_1^2 + 4\xi X_1 + \xi^2 = 0,
\end{align*}
which has no real root in $\mathbb{R}\langle \xi \rangle$.
\end{example}
\begin{example}\label{Morse_singular_co_critical}
Consider the polynomial $F=X_1^2 - X_2^2$ and the zero set of $P = X_1X_2$, which has a singular zero at $(0,0)$. It is easy to check that $F$ has no critical point on $V_{\xi}$.
\end{example}
\noindent
Theorem~\ref{Morse_on_smooth_hypersurface} ensures that a non-degenerate critical point exists which in some applications may not be always needed. For example, the existence of a central path in \PO~only requires the existence of an isolated, but not necessarily non-degenerate, critical point. For the purpose of this paper, we also establish a weaker non-singularity condition that guarantees that $\Crit(V_\xi,F) \cap \R\la\xi\ra_b^n$ has at least one critical point. This condition can be incorporated into the assumptions of Theorem~\ref{constrained_real_isolated_critical_points} to ensure the existence of an isolated bounded critical point.
\begin{definition}[General position]
\label{def:general-position}
We say that a finite set $\mathcal{P} \subset \R[X_1,\ldots,X_n]$ 
is in \emph{general position}
if for any subset $I \subset \{1,\ldots,s\}$, 
\begin{align*}
\zero(\{P_i\}_{i \in I},\R^n)
\end{align*}
is non-singular.     
\end{definition}
\noindent
\textcolor{black}{We recall that the set of strata of $V$ with respect to the canonical stratification is obtained by computing the set $\Ireg(V)$ of irregular points of $V$ (which is a real variety in $\R^n$) and then the set of irregular points of $\Ireg(V)$, recursively according to~\eqref{algebraic_set_canonical_stratification} (see Section~\ref{sec:Whitney_strat} for details).}

\begin{theorem}\label{Morse_on_transversal}
Let
\begin{align*}
V=\zero\left(\prod_{i=1}^s P_i,\R^n\right), \quad V_{\xi} = \zero\left( \prod_{i=1}^s P_{i}-\xi,\R\la \xi \ra^n\right),
\end{align*}
where $\mathcal P=\{P_1,\ldots,P_s\}\subset
\R[X_1,\ldots,X_n]$ is in general position. Further, let
$\bar{x}$ be a non-degenerate critical point of $F$ on $V$
with respect to its canonical stratification, and define
\begin{align*}
I:=\{i\in\{1,\ldots,s\}\mid P_i(\bar x)=0\},
\qquad
A:=\prod_{i\notin I}P_i,
\end{align*}
where the product over the empty set is equal to $1$. Assume that the Lagrange multipliers corresponding to the active
polynomials $\{P_i\}_{i\in I}$ are all positive and that
\begin{align*}
A(\bar x)>0.
\end{align*}
Then there exists a bounded critical point $x_\xi\in \Crit(V_\xi,F)\cap\R\langle\xi\rangle_b^n$ such that $\lim_\xi(x_\xi)=\bar{x}$.
\end{theorem}

\begin{remark}
\textcolor{black}{Under the general position condition, the canonical stratification of $V$ in Theorem~\ref{Morse_on_transversal} satisfies the Whitney regularity conditions A-B (see Section~\ref{stratified_Morse}), which is referred to as \textit{canonical Whitney stratification}. In this case, an explicit description of the strata is given by Proposition~\ref{Whitney_stratification_union_of_sets}.}
\end{remark}

\subsubsection{Limit of bounded critical points}\label{sec:convergence_critical}
Now, we assume that $F$ has a bounded critical point on $V_{\xi}$ (e.g., when the conditions of Corollary~\ref{bounded_critical_points} or~\ref{existence_critical_points} hold). We prove that when $\mathcal{P}$ is in general position, the $\lim_{\xi}$ of a bounded critical point of $F$ on $V_{\xi}$ is a critical point of $F$ on $V$ with respect to its canonical Whitney stratification. To facilitate later applications to central paths of \PO, we only state the theorem for a hypersurface $V$.

\begin{theorem}\label{convergence_of_critical_points}
Let $V = \zero\big( \prod_{i=1}^s P_{i},\R^n\big)$, 
where $\mathcal{P} = \{P_1,\ldots,P_s\} \subset \R[X_1,\ldots,X_n]$ is 
in general position.
Let $x_\xi \in V_{\xi} = \zero\big( \prod_{i=1}^s P_{i}-\xi,\R\la \xi \ra^n\big)$ be a bounded critical point of $F$ on $V_{\xi}$ and $\bar{x} = \lim_{\xi} (x_\xi)$. Then $\bar{x}$ is a critical point of $F$ on $V$ with respect to its canonical Whitney stratification. 
\end{theorem}

\begin{remark}
Analogous to Theorem~\ref{Morse_on_smooth_hypersurface}, a local condition would be enough for the result of Theorem~\ref{convergence_of_critical_points} to hold: $\mathcal{P}$ only needs to be ``locally in general position". However, this would unnecessarily complicate the stratification of $V$ and also the proof of Theorem~\ref{convergence_of_critical_points}. 
\end{remark}

\subsection{Critical and central paths of \PO}\label{appl_to_polynom_optim}
Our main motivation for addressing Problems~\ref{existence_critical_path}-\ref{boundedness_critical_path} is to understand the limiting behavior of local minimizers of the log-barrier function (and particularly central paths) for \PO. Roughly speaking, a \PO~problem is the minimization of a polynomial function over a basic closed semi-algebraic set, defined as follows. 
\begin{definition}
Let $\R$ be a real closed field, and $\mathcal{P} \subset \R[X_1,\ldots,X_{n}]$. A \textit{quantifier-free $\mathcal{P}$-formula} $\mathrm{\Phi}(X_1,\ldots,X_{n})$ with coefficients in $\R$ is a Boolean combination of atoms $P > 0$, $P=0$, or $P < 0$ where $P \in \mathcal{P}$, and $\{X_1,\ldots,X_{n}\}$ are the \emph{free variables} of $\mathrm{\Phi}$. 
A quantified $\mathcal{P}$-formula is given by 
\begin{align*}
\Psi= (\mathrm{Q}_1 Y_1) \cdots (\mathrm{Q}_{k} Y_{k}) \ \mathcal{B}(Y_1,\ldots, Y_{k},X_1,\ldots,X_{n}),
\end{align*}
in which $\mathrm{Q}_i \in\{\forall,\exists\}$ are quantifiers and $\mathcal{B}$ is a quantifier-free $\mathcal{P}$-formula with
\begin{align*}
\mathcal{P} \subset \R[Y_1,\ldots,Y_{k},X_1,\ldots,X_{n}].
\end{align*}
The set of all $(x_1,\ldots,x_{n}) \in \R^{n}$ satisfying $\Psi$ is called the $\R$-realization of $\Psi$. A $\mathcal{P}$-\textit{semi-algebraic} subset of $\R^{n}$ is defined as the $\R$-realization of a quantifier-free $\mathcal{P}$-formula. A function with a semi-algebraic graph is called a \textit{semi-algebraic function}. 
\end{definition}
\noindent
Formally, we define a \PO~problem as 
\begin{align}\label{poly_optim}
\inf_x\{f(x) \mid x \in S\},
\end{align}
where $f \in \R[X_1,\ldots,X_n]$ and $S$ is a basic closed $\mathcal{Q}$-semi-algebraic set, with $\mathcal{Q}:=\{g_1,\ldots,g_r\} \subset \R[X_1,\ldots,X_n]$, defined by
\begin{align*}
S:=\{x \in \R^n \mid g_i(x) \ge 0, \ i=1,\ldots,r\}.
\end{align*}

\begin{remark}[$\R$ versus $\mathbb{R}$]
Normally (and for practical purposes), the functions $f$ and $g_i$ in a \PO~problem are polynomials with coefficients in $\mathbb{R}$. However, for the sake of generality, and because all of our results in this section remain valid over any real closed field $\R$, we will work over $\R$, unless stated otherwise.
\end{remark}
\noindent
Notice that a \PO~problem becomes trivial if a local optimal solution satisfies $\prod_{i=1}^r g_i(x) >0$. In that case, the problem reduces to an unconstrained optimization problem, and we only need to compute zeros of $df$. Thus, to avoid trivialities, we assume that:
\begin{assumption}\label{non_trivial_problem}
The differential of $f$ does not vanish on $S$.
\end{assumption}
\noindent 
It is well-known that \PO~is provably hard from the point of view of both algebraic and bit-complexity~\cite{B17,BCSS98}. In the real number model of computation~\cite{BCSS98} (where the input is a finite set of real numbers, and complexity counts on the number of arithmetic operations, each of which requires one unit of time), checking the feasibility of \PO~when there is a polynomial of degree at least 4 is $\textbf{NP}_{\R}$-complete~\cite{BSS89}. More concretely, a global optimal solution of~\PO, when it exists, can be symbolically computed in singly exponential time~\cite[Algorithm~14.9]{BPR06}. However, in the bit model of computation, it is not even known if \PO~belongs to \textbf{NP}.
\begin{remark}\label{doubly_exponentially_small}
It follows from~\cite[Algorithm~14.9]{BPR06} that a \PO~problem defined by $f,g_i \in \mathbb{Z}[X_,\ldots,X_n]$ could have doubly exponentially small local minimizers in the worst-case scenario. See~\cite[Example~22]{R97} for a concrete example.
\end{remark}

\vspace{5px}
\noindent
Although there is no polynomial-time algorithm for an exact optimal solution of \PO, one can utilize a penalization technique in \NO~(known as Sequential Unconstrained Optimization Technique, as introduced in~\cite{FM90}) to ``numerically" solve \PO. The idea is to minimize the log-barrier function
\begin{align}\label{unconstrained_semi-algebraic_optim}
\inf_x\bigg\{f(x) - \mu \sum_{i=1}^r \log(g_i(x)) \mid g_i(x) > 0, \quad i=1,\ldots,r\bigg\},
\end{align}
while letting $\mu \downarrow 0$. The unconstrained minimization~\eqref{unconstrained_semi-algebraic_optim} assumes the existence of a strictly feasible point (known as \textit{Slater's condition}~\cite{NW06}), as follows.
\begin{assumption}\label{Slater_cond}
There exists an $x$ such that $g_i(x) > 0$ for all $i=1,\ldots,r$.
\end{assumption}
\noindent
If there exists a local optimal solution $x^*$ to~\eqref{poly_optim}, then a local optimal solution of~\eqref{unconstrained_semi-algebraic_optim}, denoted by $x(\mu)$, should be close to $x^*$ for sufficiently small positive $\mu$. Computing local minimizers of the log-barrier function is the basis of modern IPMs, that has been successfully applied to linear optimization (\LO)~\cite[Chapter~5]{CRT06}, \SDO~\cite[Chapter~5]{Kl02}, conic optimization~\cite[Chapter~3]{NN94}, and general \NO~\cite[Chapter~19]{NW06}. A special trajectory of local minimizers of the log-barrier function leads to the notion of a central path, as defined in~\cite[Page~72]{FM90}.
\begin{notation}
For the ease of exposition, we define the \textit{algebraic boundary} of $S$ by
\begin{align*}
S_{=}:=\Bigg\{x \in \R^n \mid \prod_{i=1}^r g_i(x) = 0\Bigg\},
\end{align*}
and the \textit{algebraic interior} of $S$ by
\begin{align*}
S_{>}:=S \setminus S_=.
\end{align*}
\end{notation}
\begin{definition}[Central path]\label{central_path_def}
Given $\mu_0 > 0$, a central path is a continuous function $\zeta:(0,\mu_0) \to \R^n$ such that $\zeta(\mu) = x(\mu)$ is an isolated (but not necessarily unique) local minimizer of~\eqref{unconstrained_semi-algebraic_optim} for each fixed $\mu \in (0,\mu_0)$.
\end{definition}

\begin{remark}[Isolated versus unique]
The notion of central path in Definition~\ref{central_path_def} generalizes the classical central path concepts from \LO~\cite[Section~5.6]{CRT06}, \SDO~\cite[Section~3.1]{Kl02}, and conic optimization~\cite[Section~3.2.2]{NN94}. Unlike the central paths for these classes of optimization problems, which are always unique, the central path in \PO~may fail to be unique. This distinction motivates the use of the term ``isolated" in Definition~\ref{central_path_def}, as opposed to ``unique" in~\cite{Kl02,NN94,CRT06}.   
\end{remark}

\begin{remark}[Global versus local optimality]
We do not assume the attainment of the optimal value or even the boundedness of the optimal value of~\eqref{poly_optim} or~\eqref{unconstrained_semi-algebraic_optim}. However, we will be always explicit about the status of optimality. More precisely, we will use the adjective ``optimal" (or ``minimizer") to refer to a ``global optimal" value or solution. Otherwise, we will spell out the local optimality. 
\end{remark}
\noindent
A central path lies at the heart of primal-dual IPMs, and its algebro-geometric features (e.g., the degree of the Zariski closure of the image of a central path) reflect the iteration complexity or super-linear convergence of IPMs~\cite{ABGJ18,BL89a,BL89b,DSV12,DMS05,GS98,H02,HLMZ21,HST21,LSZ98}. However, unlike convex optimization, the existence or convergence of a central path is not guaranteed for \NO~(see Examples~\ref{non_existence}-\ref{Morse_non_compact} and~\cite[Section~7]{FGW02}). Even if a central path exists and converges, it may converge to a non-local optimal solution, as the following example illustrates.
\begin{example}\label{convergence_to_algebraic_boundary}
 Consider the minimization problem over the solid \textit{figure eight}:
\begin{align*}
\inf_x \{x _1 \mid  x_1^2 - x_1^4 - x_2^4 -x_2^2  \ge 0\},
\end{align*}
which has a global minimum at $(-1,0)$ (see Fig.~\ref{central_path_fig8}). There are two central paths converging to $(-1,0)$ and $(0,0)$, but $(0,0)$ is not a local minimizer of $f$ on $S$.

 \begin{figure}
\includegraphics[height=2in]{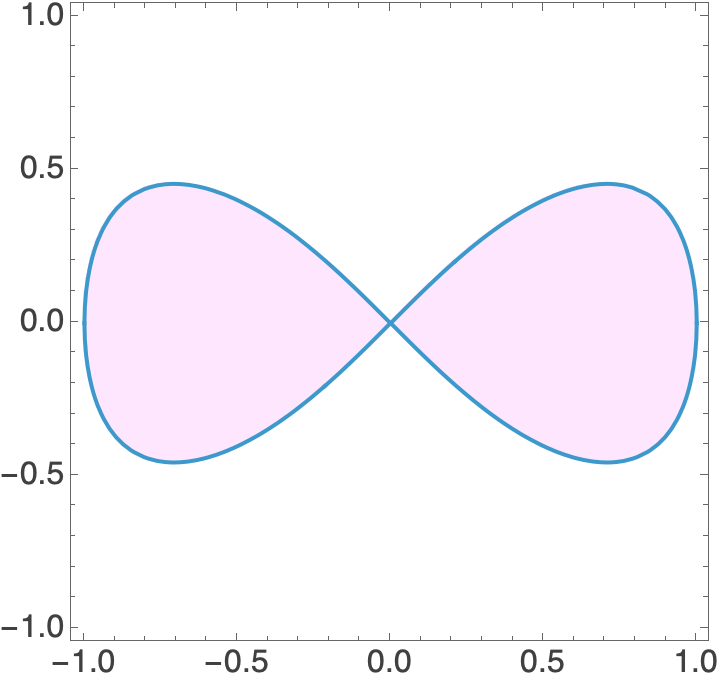} \ \
\includegraphics[height=2in]{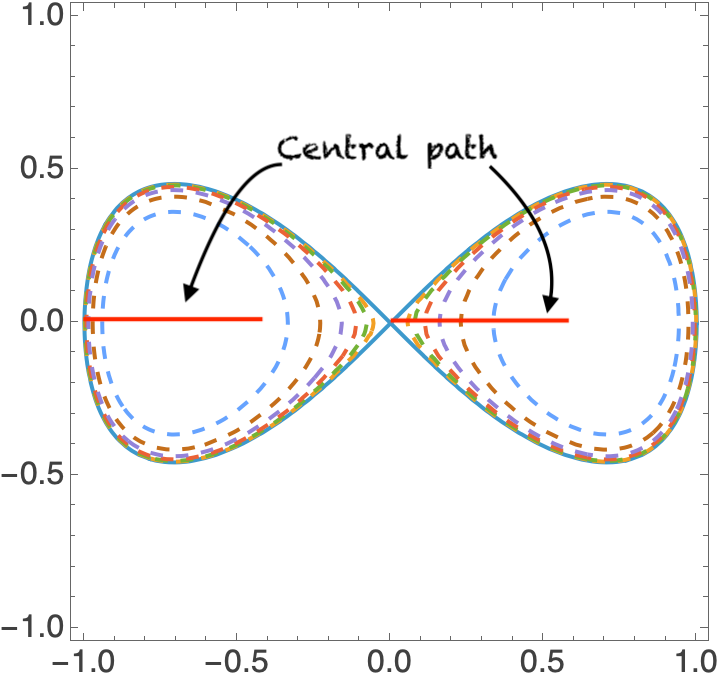}
\caption{A central path may converge to a non-local minimizer.}
\label{central_path_fig8}
\end{figure}
\end{example}
\noindent
The above issue arises from the presence of multiple local minimizers, maximizers or saddle points of the log-barrier function of~\PO. More precisely, the first-order optimality conditions for the log-barrier function~\eqref{unconstrained_semi-algebraic_optim} are given by
\begin{align}\label{central_path_original_form}
\frac{\partial f}{\partial x_j} - \mu \sum_{i=1}^r \frac{\partial g_i/\partial x_j}{g_i} = 0, \qquad j=1,\ldots,n. 
\end{align}
Given a fixed $\mu>0$ and the fact that $g_i(x) > 0$ for a central solution, the first-order conditions define an algebraic subset of $\R^n$ as follows
\begin{align}\label{critical_path_algebraic}
V_{\mu}:=\zero\Bigg(\Bigg\{\frac{\partial f}{\partial x_j} \prod_{i=1}^r g_i - \mu \sum_{k=1}^r \frac{\partial g_k}{\partial x_j} \prod_{i \neq k} g_i\Bigg\}_{j=1,\ldots,n},\R^n\Bigg). 
\end{align}
This would lead to possibly multiple central paths or non-central paths. For computational optimization purposes, it is important to understand the limiting behavior of all points in $V_{\mu} \cap S_{>}$, because a critical point of the log-barrier function might be simply mistaken (caused by numerical round-off errors (see Remark~\ref{doubly_exponentially_small})) with a local minimizer. Even if a local minimizer is correctly decided, the trajectory may converge to a non-local minimizer of \PO, as illustrated by Example~\ref{convergence_to_algebraic_boundary}.

\vspace{5px}
\noindent
Therefore, we need to expand the study of local minimizers of the log-barrier function to the set $V_{\mu}$. This introduces a new family of semi-algebraic paths associated with the log-barrier function (including central paths), so-called \textit{critical paths}, which we formally define below. 
\begin{definition}[Critical path]\label{critical_path_def}
Given $\mu_0 > 0$, a critical path is a continuous function $\nu:(0,\mu_0) \to \R^n$ such that $\nu(\mu) = x(\mu)$ is an isolated point of $V_{\mu} \cap S_{>}$ for each fixed $\mu \in (0,\mu_0)$.
\end{definition}
\begin{remark}
As a proper extension of a central path, Definition~\ref{critical_path_def} excludes the points in $V_{\mu} \setminus S_>$. This exclusion does not significantly limit generality, because the limit of bounded points in $V_{\mu} \setminus S_>$ might be infeasible with respect to~\PO, and are thus irrelevant from an optimization perspective.  
\end{remark}
\noindent
\textcolor{black}{In contrast to the central path of \SDO, which is intrinsically defined on the entire interval $(0,\infty)$ (see, e.g.,~\cite{Kl02}), a central path, and in general, a critical path of \NO\ may be defined only on an interval $(0,\bar{\mu})$, where $\bar{\mu}>0$ may be very small. Since the asymptotic behavior of a critical path is determined entirely by its behavior as $\mu\downarrow 0$, restricting attention to infinitesimal values of $\mu$ incurs no loss of information. In order to make this precise, we use the terminology and notation introduced in Section~\ref{sec:critical_points_on_V_xi}, occasionally referring to a critical path as a point of $V_{\mu} \subset \R \la \mu \ra^n$. For consistency with Section~\ref{sec:critical_points_on_V_xi}, we denote a critical path as a semi-algebraic function (see Proposition~\ref{critical_path_semi-algebraic_proof}) by $x(\mu)$, and as a vector of Puiseux series by $x_{\mu}$.}    

\textcolor{black}{
\begin{remark}\label{inf_approach}
When there is a single parameter, the interpretation of the infinitesimals (see Remark~\ref{infinitesimal_interpretation})
is especially concrete. In this case, the field $\R\langle\xi\rangle$ can be
identified with the field of germs at \(0^+\) of semi-algebraic
functions~\cite[Theorem~3.14]{BPR06}. Further, every algebraic Puiseux series has a convergent
expansion on some interval $(0,\xi_0)$. More specifically, in case of critical paths, a bounded point $x_\mu\in\mathbb{R}\langle\mu\rangle_b^n$
represents an ordinary semi-algebraic function $x(\mu)$, defined for all
sufficiently small $\mu>0$, and
\begin{align*}
\lim_{\mu}x_\mu=\lim_{\mu\downarrow 0}x(\mu).
\end{align*}
Thus, all of our results concerning critical paths have a
direct interpretation for the barrier parameter $\mu$. The
infinitesimal notation records the germ at $\mu=0$ and avoids repeatedly
writing ``for all sufficiently small $\mu>0$.''
\end{remark}}

\noindent
As a semi-algebraic function, a critical path of \PO~is piece-wise continuous~\cite[Proposition~5.20]{BPR06}. Further, it can be shown that for sufficiently small $\mu$, a critical path is a Nash mapping (see Definition~\ref{Nash_def} and Proposition~\ref{critical_path_semi-algebraic_proof}). Nash functions are of significant importance in real algebraic geometry. Nash functions preserve the good algebraic properties of polynomials and they flexibly appear in implicit function and preparations theorems~\cite{BCR98}, which are utilized in the proofs of Theorems~\ref{Morse_on_smooth_hypersurface},~\ref{Morse_on_transversal}, and~\ref{sufficient_conditions_existence}. 
\begin{definition}[Nash function]\label{Nash_def}
Given an open semi-algebraic set $U \subset \R^n$, a semi-algebraic function $f:U \to \R$ is called Nash if $f$ is $\mathcal{C}^{\infty}$-smooth. If $\R=\mathbb{R}$, then a Nash function is analytic and semi-algebraic~\cite[Proposition~8.1.8]{BCR98}.
\end{definition}

\vspace{5px}
\noindent
In what follows, we show that the problem of existence of a critical path reduces to problem of existence of critical points of $f$ on the infinitesimally deformed hypersurface   
\begin{align}\label{S_perturbed}
S_{\xi}:=\Bigg\{x \in \R^n \mid \prod_{i=1}^r g_i(x) = \xi\Bigg\},
\end{align}
and we apply our theoretical results in Sections~\ref{sec:existence_critical}-\ref{sec:convergence_critical} to establish conditions for the existence, convergence, and smoothness of a critical path.

\begin{remark}
Our results extend the classical analysis of the existence, convergence, and smoothness of central paths to the broader setting of critical paths. In contrast to traditional approaches in \NO, our framework does not rely on the existence of KKT points.
\end{remark}

\subsubsection{Existence of a critical path}\label{critical_path_existence}
Conditions on the existence of a critical path can be divided to (i) $V_{\mu} \cap S_> \subset \R \la \mu \ra^n$ being non-empty, and (ii) $V_{\mu} \cap S_>$ having isolated points. These conditions automatically hold for \LO~and~\SDO~~\cite{CRT06,Kl02} (and even more general classes of convex optimization), when the Slater's condition holds. We establish conditions that guarantee that $V_{\mu}$ is non-empty and finite.

\vspace{5px}
\noindent
First, we show that the existence of a special Lagrange critical pair $(x_\xi,u_\xi)$ with $x_{\xi} \in S_>$ ensures that $V_{\mu}$ is non-empty.

\begin{theorem}\label{critcal_path_existence}
Let $x_{\xi}$ be a critical point of $f$ on $S_\xi$ and let $u_\xi$ be the Lagrange multiplier associated with $x_{\xi}$. If $x_{\xi} \in S_>$ and $\xi u_\xi -\mu$ has a positive \textcolor{black}{infinitesimal} zero in $\R\la \mu \ra$, then $V_{\mu} \cap  S_> \neq\emptyset$. 
\end{theorem}
\begin{remark}\label{no_critical_path}
We should indicate that without the condition on the Lagrange multiplier, Theorem~\ref{critcal_path_existence} is not true. For instance, consider the minimization 
\begin{align}\label{Morse_central_path_problem}
\inf_x\{x_1^2-x_2^2 \mid x_2 \ge 0\},
\end{align}
which has an isolated critical point at $(0,0)$. Although $F=X_1^2-X_2^2$ has a critical point $x_\xi=(0,\xi)$ on $X_2 - \xi = 0$ with Lagrange multiplier $u = -2\xi$, there is no critical path for~\eqref{Morse_central_path_problem}. Notice that the equation $(\xi)(-2\xi) = \mu$ has no real root in $\R\la \mu \ra$. 
\end{remark}
\noindent
Theorem~\ref{critcal_path_existence} can be seen as an extension of~\cite[Theorem~3.10]{FGW02}. The key distinction is that, unlike a bounded set of local optima, a bounded set of critical points does not necessarily guarantee the existence of a critical path converging to it.

\vspace{5px}
\noindent
\textcolor{black}{Theorem~\ref{constrained_real_isolated_critical_points} provides finiteness of the Lagrange critical pairs on all but finitely many level sets $S_\xi$. Theorem~\ref{isolated_critical_points} shows that this fixed-level finiteness is sufficient to guarantee the isolation of a bounded point of $V_\mu \cap S_>$.}

\begin{remark}[Bounded critical paths]
For the purpose of establishing the existence of a critical path, we may often (e.g., in Theorem~\ref{isolated_critical_points}) assume that a critical point of the log-barrier function is bounded. This assumption does not significantly limit generality from an optimization perspective, as an unbounded critical path may indicate that the optimal value of the \PO~problem is either infinite or not attained. This assumption is unnecessary in convex optimization, as the central path exists if and only if it is bounded (see e.g.,~\cite[Proposition~6]{GP02}). 
\end{remark}
\begin{theorem}\label{isolated_critical_points}
\textcolor{black}{Let $x_{\mu} \in V_{\mu} \cap \R \la \mu \ra_b^n \cap S_>$ be a bounded solution}. Suppose that for some $c \in \C$, the set of all complex projective Lagrange critical pairs of $f$ in $\mathbb{P}_n(\C) \times \mathbb{P}_1(\C)$ on $S_{c} = \zero(\prod_{i=1}^r g_i - c,\C^n)$ is finite. Then $x_{\mu}$ is a critical path. 
\end{theorem}

\begin{remark}
The non-degeneracy condition in Theorem~\ref{constrained_real_non-degenerate_critical_points} not only ensures that $x_{\mu}$ is isolated, but also 
implies its
non-degeneracy, which in turn guarantees that $x(\mu)$ is a Nash mapping. However, this is somewhat trivial, as any critical path $x(\mu)$ is already a Nash function for sufficiently small $\mu$. Therefore, in Theorem~\ref{isolated_critical_points}, it suffices to assume only the finiteness of the set of complex projective Lagrange critical pairs.
\end{remark}

\hide{
\vspace{5px}
\noindent
Given a bounded critical path with limit point $\bar{x}$, we can state the finiteness conditions in terms of projective KKT points (see Section~\ref{KKT_Conditions_Lagrange}).
\begin{theorem}\label{isolated_KKT_points}
Suppose that $V_{\mu} \subset \R \la \mu \ra^n$ is non-empty, and there exists an accumulation point $\bar{x} \in S_{=} \cap S$ such that $g_i(\bar{x}) = 0$ for $i=1,\ldots,\bar{r}$, $g_i(\bar{x}) > 0$ for $i=\bar{r}+1,\ldots,r$, $\bar{r} > 0$ and $g_i = 0$ for $i=1,\ldots,\bar{s}$ intersect transversally at $\bar{x}$. If $f(x) + \varepsilon\prod_{i=\bar{s}+1}^{\bar{r}} g_{i}(x)$ has only non-degenerate projective critical points on  
\begin{align*}
 \{x \in \C^n \mid g_i(x) = 0, \ i=1,\ldots,\bar{s}\Big\}
\end{align*}
for some $\varepsilon \in \C$, then there exists a Nash critical path converging to $\bar{x} \in S_= \cap S$.  
\end{theorem}
}
\noindent
By utilizing Theorem~\ref{Morse_on_transversal}, we establish conditions that guarantee the existence of a bounded critical path. We prove results that involve the notion of critical points of $f$ on $S_=$. Recall that $S_=$ is canonically Whitney stratified according to Proposition~\ref{Whitney_stratification_union_of_sets}. 
\begin{theorem}\label{sufficient_conditions_existence}
Suppose that $\mathcal{Q}=\{g_1,\ldots,g_r\} \subset \R[X_1,\ldots,X_n]$ in the definition of~\PO~is in general position.
Further, let $\bar{x} \in S$ be a non-degenerate critical point of $f$ on $S_=$ with respect to its canonical Whitney stratification, and assume that the corresponding Lagrange multipliers of $\bar{x}$ are all positive. Then there exists a critical path $x_\mu \in \R \la \mu \ra_b^n$, and $\lim_{\mu} (x_\mu) = \bar{x}$. 

\vspace{5px}
\noindent
Furthermore, the critical path is Nash at $\mu = 0$. If $\R = \mathbb{R}$, the critical path is analytic at $\mu = 0$.
\end{theorem}
\begin{remark}
Theorem~\ref{sufficient_conditions_existence} extends the existence result of~\cite[Theorems~12]{FM90} which was only stated for a central path.
\end{remark}
\hide{
\begin{remark}
A simpler version of Theorem~\ref{Morse_on_smooth_hypersurface} has been proved in~\cite[Lemma~B]{M65} for a general Morse function on an open subset of $\mathbb{R}^n$. This result applies to our problem when $f$ is Morse on $\mathbb{R}^n$ and $\bar{x} \not \in S_{=}$. In that case, for sufficiently small $\mu$, $x(\mu)$ must be a non-degenerate critical point of the log-barrier function.  
\end{remark}
}
\subsubsection{Convergence of a critical path}
By utilizing Theorem~\ref{convergence_of_critical_points}, we can characterize the limit of a bounded critical path and quantify its convergence rate. The latter result extends the authors' worst-case convergence rate of central path from~\SDO~\cite{BM22} to~\PO.  

\begin{theorem}\label{limiting_behavior_of_critical_path}
Suppose that $\mathcal{Q}=\{g_1,\ldots,g_r\} \subset \R[X_1,\ldots,X_n]$ in the definition of~\PO~is in general position, and let $x_{\mu} \in \R \la \mu \ra_b^n$ be a bounded critical path. Then $\bar{x} = \lim_{\mu} (x_{\mu})$ is a critical point of $f$ on $S_{=}$ with respect to its canonical Whitney stratification. Further, there exists $\gamma \in \mathbb{Z}_+$ such that  
\begin{align*}
    \|x_\mu - \bar{x}\| = O(\mu^{1/\gamma}), 
\end{align*}
and $\gamma = (rd)^{O(n)}$.
\end{theorem}

\begin{example}\label{ex:No_central_path}
Consider the minimization problem
\begin{align*}
\inf_x\{x_1 \mid x_1^2+x_2^2 \ge 1, \ x_1 \ge 0\},
\end{align*}
where the canonical Whitney stratification of $S_{=}$ gives the strata 
\begin{align*}
Z_0&=\{(x_1,x_2) \mid x_1 = 0, \ x_2 > 1\}
\bigcup \{(x_1,x_2) \mid x_1 = 0, \ x_2 < -1\}\\
&\qquad\bigcup \{(x_1,x_2) \mid x_1 = 0, \ -1 < x_2 < 1\}
\bigcup \{(x_1,x_2) \mid x_1^2+x_2^2 = 1, x_1 > 0\}\\
&\qquad\bigcup \{(x_1,x_2) \mid x_1^2+x_2^2 = 1, x_1 < 0\},\\
Z_1&=\{(0,1)\} \bigcup \{(0,-1)\}.
\end{align*}
The first-order conditions~\eqref{central_path_original_form} yield  
\begin{align*}
    &x_1^3-3\mu x_1^2 - x_1 + \mu = 0,\\
    &x_2 = 0,
\end{align*}
which results in a unique critical (non-central) path $x_{\mu}$ with $\lim_{\mu} (x_{\mu}) = (1,0)$, which is a non-singular point of $S_=$. It is easy to see that $(1,0)$ is a critical point of $f$ on $Z_0$. Further, we can observe that no critical path converges to the unbounded critical set $\{(x_1,x_2) \mid x_1 = 0\}$.
\end{example}

\subsubsection{Smoothness of a critical path at $\mu=0$}
As a consequence of the Semi-algebraic Implicit Function Theorem~\cite[Corollary~2.9.8]{BCR98}, Theorem~\ref{sufficient_conditions_existence} ensures that the critical path is $\mathcal{C}^{\infty}$-smooth at $\mu = 0$. However, if any of the conditions in Theorem~\ref{sufficient_conditions_existence} are not satisfied, the derivatives of a critical path (when it exists) may fail to exist at $\mu =0$. 
\begin{example}\label{non_analytic_at_0}
Consider the \PO~problem
  \begin{align*}
 \inf_x\{x_1 \mid  x_1^3 - x_2^2  \ge 0, \ x_2 \ge 0\},  
 \end{align*} 
which has a unique local minimizer at $(0,0)$. In this example, there exists a unique central path $(x_1(\mu),x_2(\mu))=(O(\mu), O(\mu^{\frac 32}))$, which is clearly not analytic at $\mu = 0$ (see Fig.~\ref{central_path_cusp}).
 \begin{figure}
  \centering
\includegraphics[height=2.0in]{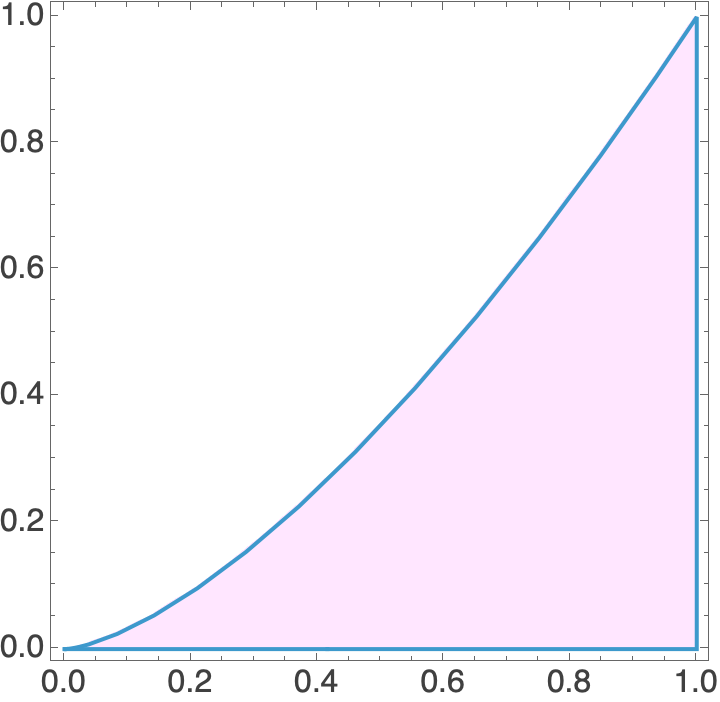} \ \ 
\includegraphics[height=2.0in]{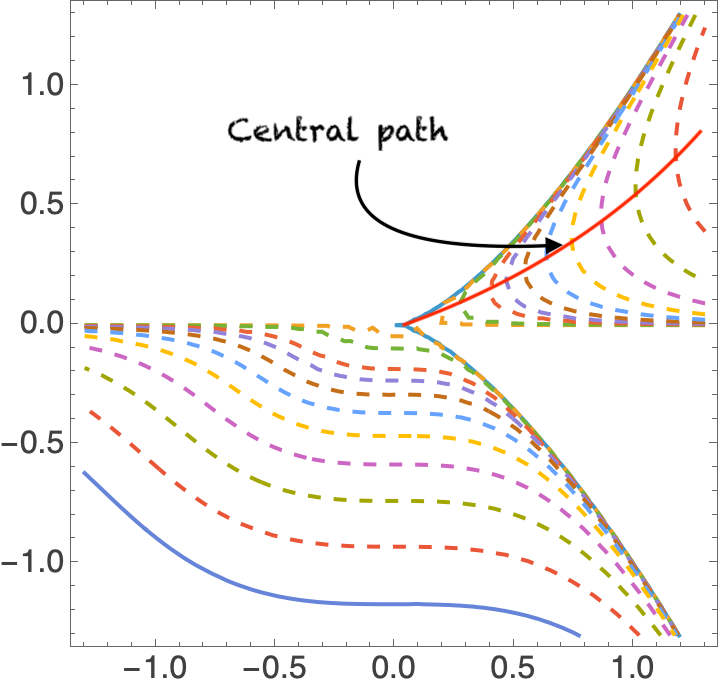}
\caption{Central path converges to an isolated singular solution.}
\label{central_path_cusp}
\end{figure}
\end{example}
\noindent
We prove the existence of a reparametrization that recovers the $\mathcal{C}^{\infty}$-smoothness of a critical path at $\mu = 0$. This is an extension of the authors' results on the analyticity of the central path for~\SDO~\cite{BM24}.

\begin{theorem}\label{analytic_reparam_semi-algebraic}
Let $x(\mu)$ be a bounded critical path. Then there exists a $\rho \in \mathbb{Z}_+$ such that $x(\mu^{\rho})$ is $\mathcal{C}^{\infty}$-smooth at $\mu = 0$. The minimal $\rho$ is bounded by $(rd)^{O(n^2)}$.

\vspace{5px}
\noindent
If $\R = \mathbb{R}$, then $x(\mu^{\rho})$ is analytic at $\mu = 0$.
\end{theorem}

\begin{remark}[Worst-case bounds]
\label{rem:asymptotic-degree-bounds}
\textcolor{black}{The notation $\gamma=(rd)^{O(n)}$ means that there exists an absolute constant $c>0$, independent of
$r$, $d$, the coefficients of the input polynomials, and the
particular critical path under consideration, such that $\gamma\leq (rd)^{cn}$. Similarly, the bound $\rho=(rd)^{O(n^{2})}$ means that there exists an absolute constant $c'>0$ such that $\rho\leq (rd)^{c'n^{2}}$. The constants $c$ and $c'$ are effective in principle, since the proofs rely on effective algorithms from real algebraic geometry (see~\cite{BPR06}), but
their numerical values are not optimized here.}
\end{remark}

\subsection{Extension to o-minimal structures}\label{definable_central_path}
\textcolor{black}{In this section, unless otherwise stated, we take $\R=\mathbb{R}$}. 

\vspace{5px}
\noindent
There are variants of \NO~problems in which the objective or the feasible set is 
not semi-algebraic but still 
definable in an o-minimal structure (e.g., an o-minimal structure
where the exponential function  or trigonometric functions on 
compact intervals are definable (see Section~\ref{sec:o-minimal})). Hence, it is important to understand the convergence behavior of central and critical paths of \NO~problems restricted to this category of sets.

\begin{example}
Consider the following \NO~problem with definable functions in $\mathbb{R}_{\mathrm{exp}}$ (see Section~\ref{sec:o-minimal}):
\begin{align*}
\inf_x\{x \mid xe^x - 1 \ge 0\},
\end{align*}
where the global minimizer is a zero of $xe^x - 1 = 0$. One can check (using the Implicit Function Theorem) that a central path exists, and its graph is the zero of $xe^x - \mu e^x(x+1) - 1$, i.e., it is a definable function in $\mathbb{R}_{\mathrm{exp}}$.
\end{example}

\subsubsection{Definable critical paths}
We extend the results on the existence, convergence, and smoothness of critical paths to \NO~problems where $f$ and $g_i$ are smooth definable functions in an o-minimal structure $\mathcal{S}(\mathbb{R})$. Whenever $\mathcal{Q}=\{g_1,\ldots,g_r\}$ is in general position, we assume that $S_=$ is stratified with respect to the \textcolor{black}{canonical Whitney stratification}~\cite[II.1.14]{Shiota97}.

\vspace{5px}
\noindent
First, we show that an analog of Theorem~\ref{Morse_on_transversal} can be proved for definable sets and functions and then can be applied to critical paths. 
\textcolor{black}{\begin{definition}[General position for definable functions]
Let $\mathcal{S}(\R)$ be an o-minimal structure, and let
$\mathcal{F}=\{f_1,\ldots,f_s\}$ be a finite set of $\mathcal{C}^1$-smooth
definable functions $f_i:\R^n\to\R$. We say that $\mathcal{F}$ is
in \emph{general position} if, for every non-empty subset $I\subset\{1,\ldots,s\}$
the matrix
\begin{align*}
\left[
\frac{\partial f_i}{\partial x_j}(x)
\right]_{\substack{i\in I\\ j=1,\ldots,n}}
\end{align*}
has full row rank at every
\begin{align*}
x\in\zero(\{f_i\}_{i\in I},\R^n).
\end{align*}
\end{definition}}
\begin{theorem}\label{sufficient_conditions_existence_definable}
Let $\mathcal{S}(\mathbb{R})$ be an o-minimal structure, suppose that $f,g_i \in \mathcal{S}(\mathbb{R})$ are $\mathcal{C}^2$-smooth definable functions, $\mathcal{Q}=\{g_1,\ldots,g_r\}$ is in general position, and let $\bar{x} \in S$ be a non-degenerate critical point of $f$ on $S_=$ with respect to its canonical Whitney stratification. Further, assume that the corresponding Lagrange multipliers of $\bar{x}$ are all positive. Then there exists a definable path $x(\mu)$ of critical points of the log-barrier function, and $\lim_{\mu \downarrow 0} (x(\mu)) = \bar{x}$.
\end{theorem}
\noindent
We prove the analog of Theorem~\ref{limiting_behavior_of_critical_path} to characterize the limit point of a critical path, as a definable function (see Proposition~\ref{critical_path_definable_proof}). 
\begin{theorem}\label{limiting_behavior_of_definable_critical_path}
Let $\mathcal{S}(\mathbb{R})$ be an o-minimal structure, and suppose that $f,g_i \in \mathcal{S}(\mathbb{R})$ are $\mathcal{C}^1$-smooth definable functions. Further, assume that $\mathcal{Q}=\{g_1,\ldots,g_r\}$ is in general position, and let $x(\mu)$ be a bounded critical path. Then $\bar{x} = \lim_{\mu \downarrow 0} (x(\mu))$ is a critical point of $f$ on $S_{=}$ with respect to its canonical Whitney stratification. 

\vspace{5px}
\noindent
If $\mathcal{S}(\mathbb{R})$ is a polynomially bounded o-minimal structure, then there exists $N \in (0,1]$ such that
\begin{align*}
\|x(\mu)-\bar{x}\|=O(\mu^N) 
\end{align*}
for all sufficiently small positive $\mu$. Further, if $f,g_i \in \mathbb{R}_{\mathrm{an}}$, then $N \in \mathbb{Q} \cap (0,1]$. 
\end{theorem}

\noindent
Finally, we establish an o-minimal version of Theorem~\ref{analytic_reparam_semi-algebraic} by showing that every definable bounded critical path in a polynomially bounded o-minimal structure can be made $\mathcal{C}^k$-smooth at $\mu = 0$. 
 
\begin{theorem}\label{thm:analyticity_definable_func}
Let $\mathcal{S}(\mathbb{R})$ be a polynomially bounded o-minimal structure, and suppose that $f,g_i \in \mathcal{S}(\mathbb{R})$ are $\mathcal{C}^1$-smooth definable functions. Let $x(\mu)$ be a bounded critical path. Then for every $k \in \mathbb{Z}_+$, there exists a $\rho \in \mathbb{Z}_+$ such that $x(\mu^{\rho})$ is $\mathcal{C}^k$-smooth at $\mu = 0$.

\vspace{5px}
\noindent
Further, if $f,g_i \in \mathbb{R}_{\mathrm{an}}$, then there exists a $\rho \in \mathbb{Z}_+$ such that $x(\mu^{\rho})$ is analytic at $\mu = 0$.
\end{theorem}

\begin{remark}
\textcolor{black}{We note that a real globally analytic function need not be globally definable in $\mathbb{R}_{\mathrm{an}}$. However, its restriction to any bounded box, after an affine rescaling, is a restricted analytic function and hence is definable in $\mathbb{R}_{\mathrm{an}}$. Consequently, for a bounded critical path, the distinction between globally analytic and
$\mathbb{R}_{\mathrm{an}}$-definable data disappears locally along the path.}

\vspace{5px}
\noindent
\textcolor{black}{It follows that Theorem~\ref{thm:analyticity_definable_func} also establishes the existence of an analytic reparametrization for a bounded
critical path of a \NO~problem defined by real globally analytic functions. Moreover, the critical path itself is analytic for all
sufficiently small $\mu>0$; see the proof and Remark~\ref{Extension_of_ Peterzil}. In this sense, our result extends and strengthens \cite[Proposition~8 and Remark~1]{GP02}, where the central path of a convex \SDO~problem with analytic data is shown to be definable in $\mathbb{R}_{\mathrm{an}}$ for all sufficiently small $\mu>0$.}

\vspace{5px}
\noindent
We return to this observation after the proof of Theorem~\ref{thm:analyticity_definable_func}.
\end{remark}

\begin{remark}
\textcolor{black}{Theorems~\ref{analytic_reparam_semi-algebraic} and~\ref{thm:analyticity_definable_func} 
establish the existence of an analytic reparametrization of a critical path
for the classes of semi-algebraic (when $\R=\mathbb{R}$), globally analytic,
and globally subanalytic functions. Such an analytic reparametrization,
however, cannot be expected for arbitrary definable functions in an
arbitrary o-minimal structure. For example, consider the o-minimal
structure $\mathbb{R}_{\exp}$ and let
\begin{align*}
f(x)=
\begin{cases}
e^{-1/x}, & x>0,\\
0, & x\leq 0.
\end{cases}
\end{align*}
Then, for every $\rho\in\mathbb{Z}_+$, the function $f(x^\rho)$ is
$C^\infty$-smooth but non-analytic at $x=0$.}
\end{remark}

\subsection{Implications for computational optimization}
\label{sec:computational-implications}
\textcolor{black}{This paper addresses several fundamental questions concerning the existence, convergence, and analyticity of critical and central paths in \PO~and~\NO. At the same time, these structural results also have potential implications for computational optimization. In particular, our convergence rate results (Theorems~\ref{limiting_behavior_of_critical_path} and~\ref{limiting_behavior_of_definable_critical_path}) provide an effective worst-case estimate relating the barrier parameter $\mu$ to the distance of a bounded critical path from its limiting point. Such estimates give quantitative information on how rapidly a path-following method can approach a limiting solution as the barrier parameter decreases and, consequently, provide insight into the possible worst-case behavior of IPMs near degenerate solutions. We emphasize that this path-convergence estimate is not, by itself, an iteration-complexity bound for a particular IPM.}

\vspace{5px}
\noindent
\textcolor{black}{Our smoothness and analyticity results provide a complementary perspective. A critical path may fail to be smooth with respect to the original perturbation or barrier parameter as it approaches a singular limiting point (see Example~\ref{non_analytic_at_0}). We show (Theorems~\ref{analytic_reparam_semi-algebraic} and~\ref{thm:analyticity_definable_func}) that, after a suitable power reparametrization, the path becomes smooth, and, in the algebraic or globally subanalytic setting, analytic, at its limit. This observation provides a theoretical basis for developing higher-order perturbation and interior-point schemes for \NO~in the presence of singularity. In particular, the resulting higher-order regularity can in principle be exploited to construct more accurate Taylor predictors near the limiting point, even when derivatives with respect to the original parameter become singular or unbounded. We therefore expect the proposed reparametrization to be useful in the
design of higher-order predictor-corrector and IPMs aimed at achieving faster local convergence in singular problems.
A systematic development and complexity analysis of such algorithms is
beyond the scope of the present paper.} 

\textcolor{black}{We note that, in ongoing work, the second author and collaborators develop higher-order IPMs for \SDO~and second-order conic optimization that exploit related reparametrization ideas to obtain super-linear local convergence without assuming \textit{strict complementarity}.}

\subsection{Outline of the proofs}
We will provide a brief outline of the key ideas behind the proofs of our main results. 

\vspace{5px}
For the proofs of Theorem~\ref{bounded_fibers} and Corollary~\ref{bounded_critical_points}, we consider projective zeros of $P_i^H$, and leverage the fact that the parts at infinity of $V^H$ and $V^H_{\xi}$ are identical. The proofs of Theorem~\ref{bounded_connected_component} and Corollary~\ref{existence_critical_points} rely on semi-algebraically connected components of a semi-algebraic set, the Semi-algebraic Intermediate Value Theorem~\cite[Proposition~3.4]{BPR06}, and infinitesimal properties of a bounded, over $\R$, semi-algebraically connected semi-algebraic subset of $\R\la \xi\ra^n$~\cite[Proposition~12.43]{BPR06}. 

The proofs of Theorems~\ref{constrained_real_isolated_critical_points} and~\ref{constrained_real_non-degenerate_critical_points} invoke generic properties of complex algebraic sets. More specifically, for Theorem~\ref{constrained_real_isolated_critical_points} we describe the set of complex $\xi$ with isolated complex projective Lagrange critical pairs on $V_{c}$ as a constructible subset of $\mathbb{P}_n(\mathrm{C}) \times \mathrm{C}$ and apply the fact that the projection of a constructible set is constructible (using Chevalley's theorem~\cite[I.8, Corollary~2]{Mum99}). We then leverage the upper semi-continuity of the dimension of fibers of the projection map~\cite[I.7, Corollary~3]{Mum99}. For Theorem~\ref{constrained_real_non-degenerate_critical_points} we restrict to non-degenerate critical points and describe this set as a complex algebraic set. We then utilize the fact that the projection of a complex algebraic set is Zariski closed~\cite[Theorems~1.9 and~1.10]{S13}.  

\vspace{5px}
In the proof of Theorem~\ref{Morse_on_smooth_hypersurface}, non-singularity of $V$, the Semi-algebraic Sard Theorem~\cite[Theorem~9.6.2]{BCR98}, and the Semi-algebraic Implicit Function Theorem~\cite[Corollary~2.9.8]{BCR98} together yield the existence of bounded non-degenerate critical points of $F$ on $V_{\xi}$. In the proof of Theorem~\ref{Morse_on_transversal}, we exploit the correspondence between critical points of $F$ on $\zero(\prod_{i=1}^s P_i,\R^n)$ and those on the non-singular algebraic set $\zero(\{P_1,\ldots,P_s\},\R^n)$, when $\mathcal{P}$ is in general position. We apply the Implicit Function Theorem to prove the existence of Lagrange multipliers. 

\vspace{5px}
In the proof of Theorem~\ref{convergence_of_critical_points}, we exploit the order of vanishing terms $\prod_{\ell=1, \ell \neq i}^s P_{\ell}(x_\xi)$ for each $i=1,\ldots,s$, and then leverage the canonical Whitney stratification of $V$ and the Semi-algebraic Sard Theorem to show that the limit of the tangent spaces of $V_{\xi}$, as a bounded subset of $\R \la \xi\ra^n$, exists and contains the limit of tangent space of the stratum that contains $\bar{x}$ (see Proposition~\ref{convergence_of_tangent_spaces_extended}).

\vspace{5px}
\noindent
In the second part of the paper, we apply Theorems~\ref{constrained_real_isolated_critical_points},~\ref{Morse_on_transversal}, and~\ref{convergence_of_critical_points} to the special case $F=f$ and $V=S_=$. \textcolor{black}{More specifically, Theorem~\ref{critcal_path_existence}
establishes the correspondence between critical points of $f$ on
$S_\xi$ and critical points of the log-barrier function.} For Theorems~\ref{isolated_critical_points},~\ref{sufficient_conditions_existence}, and~\ref{limiting_behavior_of_critical_path} we apply Theorems~\ref{constrained_real_isolated_critical_points},~\ref{Morse_on_transversal}, and~\ref{convergence_of_critical_points} to $F=f$ and $V=S_=$, respectively. In the second part of Theorem~\ref{limiting_behavior_of_critical_path}, we quantify the worst-case convergence rate of a critical path, as a semi-algebraic function (Proposition~\ref{critical_path_semi-algebraic_proof}). Although the proof of Proposition~\ref{critical_path_semi-algebraic_proof} (based on cell decomposition of semi-algebraic sets (see Definition~\ref{def:cad})) remains applicable, we instead apply the Parameterized Bounded Algebraic Sampling~\cite[Algorithm~12.18]{BPR06} to $V_{\mu}$ (see~\eqref{critical_path_algebraic}), and then we utilize the Quantifier Elimination Theorem (Theorem~\ref{14:the:tqe}) and the Newton-Puiseux Theorem (Proposition~\ref{Newton_Puiseux_Thm}). This approach is standard in algorithmic real algebraic geometry~\cite{BPR06} and has been previously employed in~\cite{BM22,BM24}
for describing the central path of \SDO. A similar technique will be used in the proof of Theorem~\ref{analytic_reparam_semi-algebraic} to show the existence of a reparametrization for a critical path that recovers smoothness at the limit point.

\vspace{5px}
\noindent
In the final part of the paper, we extend the existence, convergence, and smoothness results for critical paths to a class of \NO~problems involving definable functions in an o-minimal structure $\mathcal{S}(\mathbb{R})$. Theorem~\ref{sufficient_conditions_existence_definable} extends Theorem~\ref{sufficient_conditions_existence} by incorporating o-minimal analogs of Theorem~\ref{Morse_on_transversal}, Theorem~\ref{critcal_path_existence}, and the Definable Implicit Function Theorem~\cite[Page~113]{Dries2}. To prove Theorem~\ref{limiting_behavior_of_definable_critical_path}, an extension of Theorem~\ref{limiting_behavior_of_critical_path}, we first establish the o-minimal analogs of Proposition~\ref{convergence_of_tangent_spaces_extended} (presented as Proposition~\ref{convergence_of_tangent_spaces_definable}). These results are then combined with the Definable Sard Theorem~\cite[Theorem~2.7]{Wilkie2} in the proof of Theorem~\ref{limiting_behavior_of_definable_critical_path}. The quantitative aspect of Theorem~\ref{limiting_behavior_of_definable_critical_path} relies on the H\"older inequality~\cite{Dries2} and a Puiseux type expansion for globally sub-analytic functions~\cite[Lemma~2.6]{Kur88}. Finally, Theorem~\ref{thm:analyticity_definable_func} extends the smoothness result of Theorem~\ref{analytic_reparam_semi-algebraic} to polynomially bounded o-minimal structures. The proof leverages the definability of a critical path and applies a growth dichotomy result for definable functions in $\mathcal{S}(\mathbb{R})$~\cite[Page~258]{Miller94(b)}. Specifically, we employ the Puiseux type expansion from~\cite[Lemma~2.6]{Kur88} to establish the existence of an analytic reparametrization when $f,g_i$ are globally sub-analytic functions. 

\section{Prior and related work}\label{classical_results}
\textcolor{black}{This section reviews the literature most closely related to the present
work. We first discuss results on central paths and the role of
o-minimal geometry in optimization, and then compare our framework with
perturbation results concerning constraint qualifications in
semi-algebraic optimization.}
 
\subsection{Central path}\label{central_path_background}
Existence, convergence, and analyticity of central paths have been extensively studied for variants of convex and non-convex optimization problems (see e.g.,~\cite{AM91,FM90,GP02,Guler93,Gu94,HKR02,KMNY91,MP98,MZ98}). In this section, we survey some classical results on the theory of central paths in \NO~and \SDO~which are related to the problems discussed in this paper. 

\subsubsection{Central path for \NO}
Unlike the central path of \LO~and~\SDO, there are more complications with the existence of a central path for~\PO, and a \NO~problem in general (see~\cite{FM90,FGW02,WO2002}). The main challenge stems from the non-convexity of~\eqref{poly_optim} which introduces the possibility of local optimal solutions or saddle points of the log-barrier function. More importantly, a local minimizer of the log-barrier function may not even exist. These scenarios do not arise in~\LO~and~\SDO~due to their inherent convexity.   

\begin{example}\label{non_existence}
Consider the minimization of $x_1x_2^2$ over the non-negative orthant as follows
\begin{align}\label{aux_problem}
\inf_x \{x_1x_2^2 \mid x_1,x_2 \ge 0\}.  
\end{align}
Then the first-order optimality for the log-barrier function leads to
\begin{align*}
\begin{cases} x_1x_2^2 = \mu,\\
        x_1x_2^2 = \mu / 2,
        \end{cases}
\end{align*}
which has no solution for positive $\mu$.  
\end{example}
\noindent
Notice that the optimal set of~\eqref{aux_problem} is non-empty and unbounded. However, if we include some boundedness assumptions on a local optimal set of~\eqref{aux_problem}, then there are sufficient conditions for the existence of local minimizers of the log-barrier function~\eqref{unconstrained_semi-algebraic_optim}, when $\mu$ is sufficiently small. See~\cite{FM90,FGW02} for the proofs.
\begin{proposition}[Theorem~3.10 in~\cite{FGW02}] \label{existence_of_central_path}
Suppose that (a) the algebraic interior $S_>$ is non-empty, (b) corresponding to a local optimal value $v^*$ the set
\begin{align*}
 A:=\{x \in S \mid f(x) = v^*\}   
\end{align*}
has a non-empty isolated compact subset $A^*$, and (c) $A^* \cap \overline{S_>} \neq \emptyset$. Then given a sequence $\{\mu_k\} \downarrow 0$, there exist a compact set $U$ and a sequence $\{y_k\}$ such that $A^* \subset \interior(U)$, $y_k$ is the minimizer of the log-barrier function~\eqref{unconstrained_semi-algebraic_optim} on $S_> \cap \interior(U)$ when $k$ is sufficiently large, and the sequence $\{y_k\}$ has an accumulation point in $A^*$.    
\end{proposition}
\noindent
In practice, verifying conditions of Proposition~\ref{existence_of_central_path} 
 might be challenging. More concrete criteria for the existence of local minimizers are provided in~\cite{WO2002}, which involve KKT points (see Section~\ref{sec:KKT_PO}) that satisfy second-order sufficiency conditions and Mangasarian-Fromovitz constraint qualification~\cite[Theorem~15]{WO2002}. However, we should clarify that neither these conditions nor the conditions in Proposition~\ref{existence_of_central_path} ensure the existence of a central path, as demonstrated in Example~\ref{Morse_non_compact}. Furthermore, as pointed out in~\cite{FGW02}, even if a central path exists (which could be any sequence ${y_k}$ of minimizers within $\interior(U)$), Proposition~\ref{existence_of_central_path} does not claim the limit point to be a local minimum of~\eqref{poly_optim} (see Example~\ref{convergence_to_algebraic_boundary}).
\begin{example}\label{Morse_non_compact}
Consider the \PO~problem
\begin{align*}
\inf_x \{x_1^2 + x_2^2 \mid x_1^2 + x_2^2 \ge 1\}
\end{align*}
whose optimal set $\{(x_1,x_2) \in \mathbb{R}^2 \mid x_1^2 + x_2^2 = 1\}$ satisfies the conditions of Proposition~\ref{existence_of_central_path}. Then the first-order optimality conditions for the log-barrier function lead to $x_1^2+x_2^2 = \mu + 1$ which has infinitely many solutions for each $\mu > 0$. Therefore, there is no central path converging to an optimal solution (see Fig.~\ref{Infinitely_Many_Paths}). 
 \begin{figure}
 \centering
  \begin{minipage}[c]{0.5\textwidth}
\includegraphics[height=2.0in]{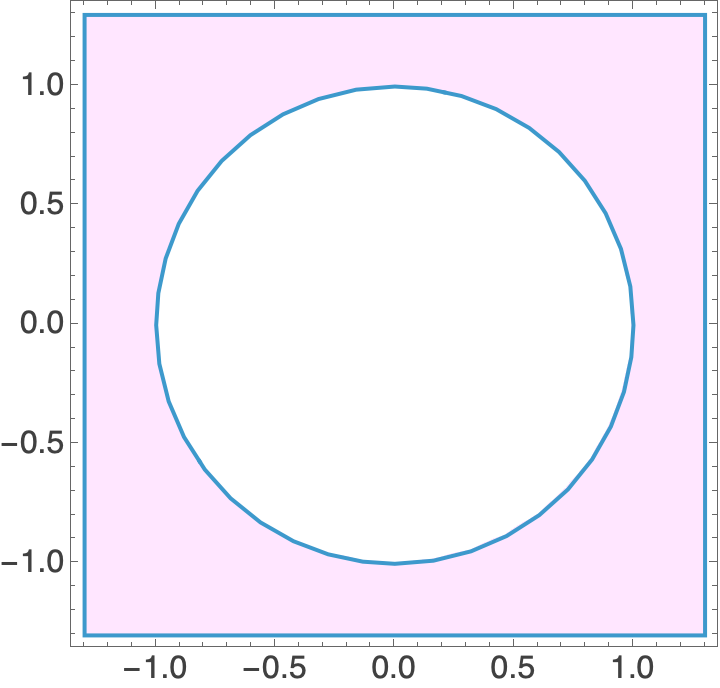}
\end{minipage}\hfill
 \begin{minipage}[c]{0.5\textwidth}
 \includegraphics[height=2.0in]{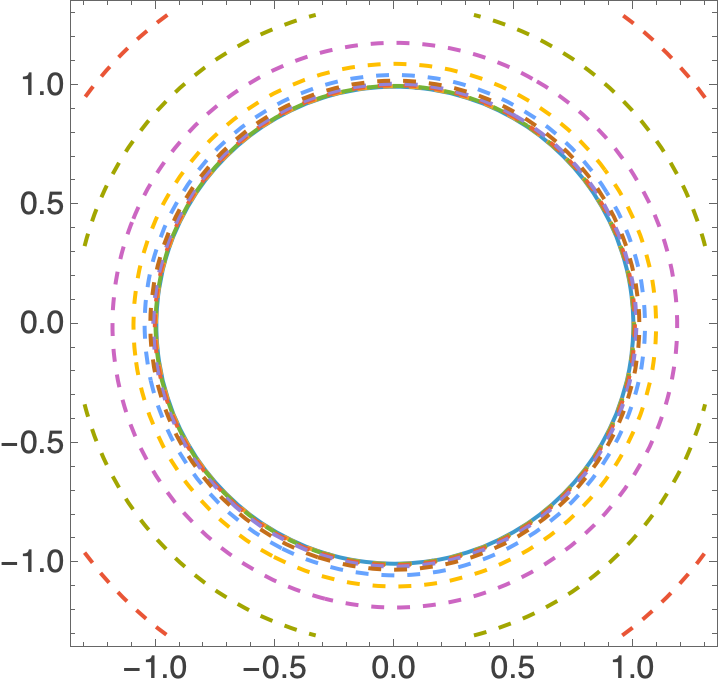}
 \end{minipage}
\caption{There are infinitely many non-isolated paths converging to the optimal set $\{(x_1,x_2) \in \mathbb{R}^2 \mid x_1^2 + x_2^2 = 1\}$.}
  \label{Infinitely_Many_Paths}
\end{figure}
\end{example}
\begin{remark}[Finiteness condition of the central path]\label{finiteness_condition}
It is important to emphasize the role of the log-barrier function and the central path in the theory of IPMs. The search directions in primal IPMs are obtained by approximately solving the log-barrier function%
\footnote{The idea of using log-barrier functions originated with Frisch~\cite{Frisch} in the context of~\LO~problems and was later developed by Fiacco and McCormick~\cite{FM90} for \NO.}(see~\cite[Section~19.5]{NW06}). The primal-dual IPMs operate within a neighborhood of the central path, where the search directions are obtained from a first-order Taylor approximation of~\eqref{primal_dual_central_path_form} (or equivalently, first-order derivatives of the central path (see~\cite[Page~569]{NW06})). Therefore, in both cases, well-defined search directions require the existence of local minimizers for sufficiently small $\mu>0$. In the latter case, the absence of the central path implies non-isolated local minimizers of the log-barrier function and a singular Jacobian of~\eqref{primal_dual_central_path_form}, which may render an ill-defined search direction for a primal-dual IPM.  
\end{remark}
\noindent
In~\cite[Theorems~12-15]{FM90}, the existence of a central path and its smoothness at $\mu = 0$ are guaranteed under linear independence constraint qualification, second-order sufficiency conditions, and the strict complementarity condition (see Section~\ref{sec:KKT_PO}). Alternative conditions are presented in~\cite[Theorems~8 and~12]{WO2002} which replace linear independence constraint qualification by Mangasarian-Fromovitz constraint qualification~\cite{NW06}. All these results involve the existence of a KKT point and ensure the non-singularity of the Jacobian of the KKT conditions or the Hessian of the log-barrier function, allowing the direct application of the Implicit Function Theorem. However, it is important to note that the existence of a KKT point is not required for the existence of a central path. For instance, consider the \PO~problem
 \begin{align}\label{unconstrained_cusp}
 \inf_x\{x_1 \mid x_1^3 - x_2^2  \ge 0\}.    
 \end{align}
One can easily check that the realization of the KKT system
 \begin{align*}
\{((x_1,x_2),u) \in \mathbb{R}^2 \times \mathbb{R} \mid 1-3x_1^2u = 0, \ 2x_2u = 0,\
     (x_2^2-x_1^3)u = 0,\
     x_2^2-x_1^3 \le 0\}
 \end{align*}
is empty, while there exists a unique central path $(x_1(\mu),x_2(\mu))=(3\mu,0)$. This is in sharp contrast with a conic optimization problem with finite optimal value, where the existence of the ``primal central path" implies convergence to an optimal KKT point~\cite[Page~74]{Ren01}.

\subsubsection{Central path for~\SDO}\label{central_path_SDO}
A special case of~\eqref{poly_optim} which also serves as a key computational tool for \PO~is \SDO~(see e.g.,~\cite{L15,L09}), given by
\begin{align*}
v^*_p:=\inf_{X} \big \{\langle C, X\rangle  \mid \langle A_{i}, X\rangle =b_i, \ \  i=1,\ldots, m, \ X \succeq 0 \big \},
\end{align*}
where
$A_1,\ldots,A_m,C$ are symmetric $n \times n$ matrices, $b \in \mathbb{R}^m$, $X \succeq 0$ means positive semi-definite, and $\langle X, Y\rangle:=\trace(XY)$ is the trace of $XY$. We recall that the graph of the central path of SDO~\cite{BM22} is defined as the semi-algebraic set  
\begin{align*}
\bigg\{(\mu,X,y,S)\mid \langle A_i,X \rangle = b_i, \ \ i=1,\ldots,m, \sum_{i=1}^m y_i A_i + S = C,\ XS = \mu I,\ X,S \succ 0\bigg\}.
\end{align*}
The existence of the central path for \SDO~is guaranteed under the linear independence of $A_i$ and the existence of a strictly feasible point, i.e., an $(X,y,S)$ with $X,S \succ 0$ such that $\langle A_i,X \rangle  = b_i$ and $ A^T y + S = C$~\cite{Kl02}. These conditions together with semi-algebraicity of the central path ensure the convergence of the central path to a primal-dual optimal solution $\big(X^{**},y^{**},S^{**}\big)$~\cite{BM22,GS98,HKR02}.  

\vspace{5px}
\noindent
It is known that the central path of \SDO~is unique, and it is analytic~\cite{BM22,Kl02}. Further, the analyticity can be extended to $\mu = 0$ when $X^{**}+S^{**} \succ 0$ (known as the strict complementarity condition for~\SDO), as shown in~\cite[Theorem~1]{H02}. In such cases, the central path converges to the limit point $\big(X^{**},y^{**},S^{**}\big)$ at the rate of 1~\cite[Theorem~3.5]{LSZ98}: 
\begin{align}\label{Lips_Bounds}
 \|X(\mu)-X^{**}\| = O(\mu) \ \ \text{and} \  \ \|S(\mu)-S^{**}\| = O(\mu).
\end{align}
However, both the analyticity at $\mu = 0$ and the Lipschitzian bounds~\eqref{Lips_Bounds} fail to hold without the strict complementarity condition~\cite{GS98}. In~\cite{BM22}, the authors proved a more general bound 
\begin{align*}
\|X({\mu})-X^{**}\|= O(\mu^{1/\gamma}) \ \ \text{and} \ \  \|S({\mu})-S^{**}\|=O(\mu^{1/\gamma}), \quad \gamma=2^{O(m+n^2)}
\end{align*}
that is independent of the strict complementarity condition. The authors showed~\cite{BM24} that a reparametrization $\mu \mapsto \mu^{\rho}$ with optimal $\rho=2^{O(m^2+n^2m + n^4)}$ recovers the analyticity of the central path at $\mu = 0$.

\vspace{5px}
\noindent
A slightly more general version of the central path was defined in~\cite{GP02} for a convex \SDO~problem as follows 
\begin{align}\label{CP_Convex_SDO}
\begin{cases}
    \cfrac{\partial f}{\partial y_i}(y) = \trace\bigg(\cfrac{\partial G}{\partial y_i}(y)S\bigg), \ \ i=1,\ldots,n,\\ 
    G(y)S = \mu I_n,\\
    G(y) \succ 0, S \succ 0,
\end{cases}
\end{align}
where $f:\mathbb{R}^n \to \mathbb{R}$ is an analytic convex function, and $G:\mathbb{R}^n \to \mathbb{S}^n$ is an analytic concave mapping. Assuming the existence of a strictly feasible point and the linear independence of the partial derivatives of $G$, the authors leveraged the underlying o-minimal structure (see Section~\ref{o-minimal_optimization}) to prove the convergence of the central path~\cite[Theorem~9]{GP02}. 

\vspace{5px}
\noindent
Before concluding this section, it is worth highlighting one variant of central path that specializes the central path of~\eqref{CP_Convex_SDO}. In~\cite{MZ98}, the convergence and analyticity of ``weighted" central paths were studied for a convex optimization problem with the same data as in~\cite{GP02}, where $G:\mathbb{R}^n \to \mathbb{R}^m$ is an analytic concave mapping. 

\subsection{O-minimal geometry and optimization}\label{o-minimal_optimization}
Significant advancements have been made at the intersection of optimization, variational analysis, and o-minimal structures (see e.g.,~\cite{ABS13,DDKL20,GP02,Io08}). The definable sets and functions (the analogs of semi-algebraic sets and functions) in an o-minimal structure share many geometric and topological properties with 
semi-algebraic sets~\cite{Michel2,Shiota97,Dries,Dries2}, and they 
arise in practice (e.g., in neural networks applications). One early application of o-minimality and definable sets and maps in optimization dates back to the work of Drummond and Peterzil~\cite{GP02} on the convergence of the central path of a convex \SDO~with analytic objective function and constraints (which in this case, the central path is definable in $\mathbb{R}_{\mathrm{an}}$ (see Section~\ref{central_path_SDO})). Their proof relies on the following result from the o-minimal geometry.
\begin{proposition}[Section~4.1 in~\cite{Dries2}]\label{Monoton_Thm}
Given a definable function $f:(a,b) \to \mathbb{R}$, the interval $(a,b)$ can be partitioned using the points $a_1 < a_2 < \ldots< a_n$ ($a_0:=a$ and $a_{n+1}:=b$) such that $f$ is either constant or $\mathcal{C}^{k}$-smooth, for some integer $k > 0$, and strictly monotone on each sub-interval $(a_i,a_{i+1})$.   
\end{proposition}
\noindent
Very recently, following the first author's work on combinatorial and topological complexity of definable sets~\cite{Basu9}, the authors proved~\cite[Theorem~2.20]{BM24b} a version of {\L}ojasiewicz inequality in polynomially bounded o-minimal structures (see Section~\ref{sec:o-minimal}), as an abstraction of the notion of independence of the {\L}ojasiewicz exponent from the combinatorial parameters. They proved the existence of a common {\L}ojasiewicz exponent for certain combinatorially defined infinite (but not necessarily definable) families of pairs of functions.

\subsection{Perturbations and constraint qualifications}
\label{sec:related-constraint-qualification}
\textcolor{black}{In classical perturbation analysis of \NO, perturbations of the problem data are typically described by continuous, or more regular, functions of one or several real parameters, rather than by passing to a non-Archimedean field extension (see, e.g.,~\cite{FM90,S97}). In the polynomial setting, however, the two viewpoints are closely related. In the one-parameter case, bounded elements of $\R\la \varepsilon \ra$ may be identified with germs at $\varepsilon=0^+$ of bounded semi-algebraic functions. Thus, considering a $\mathcal{P}$-semi-algebraic set with
\begin{align*}
\mathcal{P} \subset \R\la \varepsilon \ra_b[X_1,\ldots,X_n]
\end{align*}
is equivalent, at the level of germs, to considering a family of polynomially defined sets
\begin{align*}
V_\varepsilon
=
\left\{
x \in \mathbb{R}^n
\;\middle|\;
f_i(x,\varepsilon)=0,\quad i=1,\ldots,s
\right\},
\end{align*}
where, for each fixed $\varepsilon$, $f_i(\cdot,\varepsilon)$ is polynomial in $x$, and its coefficients are bounded semi-algebraic functions of $\varepsilon$ for all sufficiently small $\varepsilon>0$. The iterated extension $\R\la \xi_1,\ldots,\xi_s \ra$ used in this paper provides the corresponding multiscale version of this viewpoint, encoding successively smaller perturbation scales.}

\vspace{5px}
\noindent
\textcolor{black}{Our work is also related to perturbation analysis of optimization problems. On a related work, Bolte et al.~\cite{BHP18}
considered perturbations of semi-algebraic and definable optimization
problems. Their objective, however, is different from ours. They consider
perturbed inequality constraint sets
\begin{align*}
    C_{\alpha}
    =
    \{x\in\mathbb{R}^n
      \mid
      g_i(x)\leq \alpha,\ i=1,\ldots,m\}
\end{align*}
and prove that, for all sufficiently small positive perturbations, the
Mangasarian-Fromovitz constraint qualification holds at every point of
\(C_{\alpha}\). More generally, along a definable diagonal family, only
finitely many perturbation parameters are singular. In the polynomial
case, they give an effective bound on the number of such singular
parameters and apply the resulting regularization scheme to sequential
quadratic programming and semi-definite relaxations.}

\vspace{5px}
\noindent
\textcolor{black}{The two approaches share the use of semi-algebraic and o-minimal
perturbation techniques, but address complementary questions. The
regularity studied in~\cite{BHP18} is a property of the
constraint representation and is essentially independent of the
objective function. In contrast, our results concern
objective-dependent critical points on perturbed algebraic sets. We study
their existence, boundedness, finiteness, non-degeneracy, and limiting
behavior, and then apply these results to the existence, convergence,
rate of convergence, and smoothness of log-barrier critical paths.
Mangasarian-Fromovitz regularity alone does not imply the existence or
isolation of such a critical path, nor does it characterize its limit.
Conversely, our results do not establish that the perturbed inequality
system satisfies a constraint qualification throughout its feasible set.}

\section{Proof of the main results}
Before proving the main results, we briefly review concepts from algorithmic real algebraic geometry, Morse theory, o-minimal geometry, and the theory of real analytic functions. The reader is referred to~\cite{BPR06,GM,M63,Michel2,Dries,KP02,W78} for details.

\subsection{Algorithmic real algebraic geometry}
\label{background_RAG}
In this section, we recall the notions of Thom encoding, univariate representations, and quantifier elimination from algorithmic real algebraic geometry. See~\cite[Chapters~2,12, and~14]{BPR06} for more details.

\vspace{5px}
\noindent
\subsubsection{Real closed fields} Although in the optimization literature, the primary focus of \PO~is on the field of real numbers, here we will need to consider $\R$, as a real closed field, for the proofs of the first part and also the non-Archimedean real closed extensions of $\R$ -- namely, the field of algebraic Puiseux series with coefficients in $\R$ -- and their applications to the existence, convergence, and smoothness of critical paths of~\PO. Recall from~\cite[Chapter 2]{BPR06} that a real closed field $\R$ is an ordered field, where every positive element is a square, and every polynomial of odd degree has a root in $\R$.

\subsubsection{Univariate representations and Thom encodings}\label{sign_cond_Thom_encod}
Let $\R$ be a real closed field. An \textit{$\ell$-univariate representation} is $(\ell+2)$-tuple of polynomials $u=\big(f, g_{0},\ldots,g_{\ell}\big) \in \R[T]^{\ell+2}$, where $f$ and $g_0$ are coprime. A \textit{real $\ell$-univariate representation} of an $x \in \R^{\ell}$ is a pair $(u,\sigma)$ of an $\ell$-univariate representation $u$ and a \textit{Thom encoding} $\sigma$ of a real root $t_{\sigma}$ of $f$ such that
\begin{align*}
x=\bigg(\frac{g_1(t_{\sigma})}{g_0(t_{\sigma})},\ldots,\frac{g_{\ell}(t_{\sigma})}{g_0(t_{\sigma})}\bigg) \in \R^{\ell}.
\end{align*}
Let $\Der(f):= \big\{f, f^{(1)},f^{(2)},\ldots,f^{(\deg(f))}\big\}$ denote a list of polynomials in which $f^{(i)}$ for $i > 0$ is the formal $i^{\mathrm{th}}$-order derivative of $f$ and $\deg(f)$ stands for the degree of $f$. The Thom encoding $\sigma$ of $t_{\sigma}$ is a sign condition on $\Der(f)$ such that $\sigma(f)=0$. By Thom's Lemma~\cite[Proposition~2.27]{BPR06}, every root of a polynomial
$P \in \R[X]$ is uniquely characterized by a sign condition on $\Der(P)$. 

\begin{proposition} [Thom's Lemma] 
\label{prop:Thom}
Let $P \subset \R[X]$ be a univariate polynomial and $\sigma \in \{-1,0,1\}^{\Der(P)}$. Then the realization of the sign condition $\sigma$ is either empty, a point, or an open interval.
\end{proposition}

\subsubsection{Quantifier elimination}
A classical result due to Tarski \cite{Tarski51} states that every quantified formula is equivalent modulo the theory of real closed fields to a quantifier-free formula. We will use a quantitative version of this theorem~\cite[Theorem~14.16]{BPR06} in the proofs of Theorems~\ref{limiting_behavior_of_critical_path}-\ref{analytic_reparam_semi-algebraic}.

\begin{theorem}
[Quantifier Elimination]\label{14:the:tqe}
Let $\mathcal{P} \subset \R[X_{[1]},\ldots, X_{[\omega]}, Y]_{\leq d}$ be a
finite set of $s$ polynomials, where $X_{[i]}$ is a block of $k_i$ variables, and
$Y$ is a block of $\ell$ variables.
Consider the quantified formula 
\[
\Phi (Y) = (Q_1 X_{[1]})\cdots (Q_\omega X_{[\omega]}) \Psi(X_{[1]},\ldots, X_{[\omega]}, Y)
\]
and $\Psi$ a $\mathcal{P}$-formula.
 Then there exists a quantifier-free formula 
  \[ \Psi (Y) = \bigvee_{i=1}^{I} \bigwedge_{j=1}^{J_{i}} \Big(
     \bigvee_{n=1}^{N_{ij}} \sign (P_{ijn} (Y))= \sigma_{ijn} \Big)  \]
    equivalent to $\Phi$, where $P_{ijn} (Y)$ are polynomials in the variables $Y$, $\sigma_{ijn} \in
  \{0,1, - 1\}$, and
\begin{align*}
\sign(P_{ijn}(Y))\!&:=\begin{cases} \ \ 0 \ \ &P_{ijn}(Y) = 0,\\ \ \ 1 \ \ &P_{ijn}(Y) > 0,\\-1  &P_{ijn}(Y) < 0.  \end{cases}
\end{align*}
Furthermore, we have  
  \begin{eqnarray*}
    I & \leq & s^{(k_{\omega} +1) \cdots (k_{1} +1) ( \ell +1)} d^{O
    (k_{\omega} ) \cdots O (k_{1} ) O ( \ell )} ,\\
    J_{i} & \leq & 
      s^{(k_{\omega} +1) \cdots (k_{1} +1)} d^{O (k_{\omega} ) \cdots O (k_{1}
      )}
    ,\\
    N_{ij} & \leq &
      d^{O (k_{\omega} ) \cdots O (k_{1} )}
  ,
  \end{eqnarray*}
  and the degrees of the polynomials $P_{ijk} (y)$ are bounded by $d^{O
  (k_{\omega} ) \cdots O (k_{1} )}$.
\end{theorem}

\subsubsection{Newton-Puiseux theorem}\label{sec:Puiseux_series}
Puiseux series appear naturally in algorithmic real algebraic geometry for the description of roots
of a branch of a  real algebraic curve $F(X,Y) = 0$. In this paper, they serve a major role in the extension of solutions of $V$ to $\R\langle \xi \rangle$ in Theorems~\ref{bounded_fibers}-\ref{bounded_connected_component}, and also in the smoothness properties of a critical path, which is a semi-algebraic function.

\vspace{5px}
\noindent
A \textit{Puiseux series} with coefficients in $\R$ (resp.  $\C$)
is an infinite series of the form $\sum_{i = r}^{\infty} c_i \varepsilon^{i/q}$ where $c_i \in \R$ (resp. $c_i \in \C)$, $i,r \in \mathbb{Z}$, and $q$ is a positive integer, so-called the \textit{ramification index} of the Puiseux series. 

\vspace{5px}
\noindent
The field of Puiseux series in $\varepsilon$ with coefficients in $\R$ (resp. $\C$) is  denoted by 
$\R\langle \langle \varepsilon \rangle\rangle$ (resp. $\C\langle \langle \varepsilon \rangle\rangle$). It is a classical fact (see~\cite[Theorems~2.91 and~2.92]{BPR06}) that the field 
$\R\langle \langle \varepsilon \rangle\rangle$ 
(resp. $\C\langle \langle \varepsilon \rangle\rangle$)
is real closed (resp. algebraically closed). The subfield of $\R\langle \langle \varepsilon \rangle\rangle$ of elements which are algebraic over $\R(\varepsilon)$ is called the field of algebraic Puiseux series with coefficients in $\R$, 
and is denoted by $\R \langle \varepsilon \rangle$. 
It is the real closure of the ordered field $\R(\varepsilon)$ in which $\varepsilon$ is positive but smaller than every positive element of $\R$.
Alternatively, $\R \langle \varepsilon \rangle$ is the field of germs of semi-algebraic functions to the right of the origin, i.e., a continuous semi-algebraic
function $(0,t_0) \rightarrow \R$ can be represented by a Puiseux series in
$\R \langle \varepsilon \rangle$~\cite[Theorem~3.14]{BPR06}.

\vspace{5px}
\noindent
We let $\order(\cdot)$ denote the order of a Puiseux series, and it is defined as $\order(\sum_{i = r}^{\infty} c_i \varepsilon^{i/q})=r/q$ if $c_{r} \neq 0$ ( see~\cite[Section~2.6]{BPR06}). We denote by $\R\langle\varepsilon\rangle_b$ the subring of $\R\langle\varepsilon\rangle$ of elements with are bounded over $\R$
(i.e. all Puiseux series in $\R\langle\varepsilon \rangle$ whose orders are
non-negative). We denote by $\lim_{\varepsilon}:\R\langle\varepsilon\rangle_b \rightarrow \R$ 
which maps a bounded Puiseux series $\sum_{i = 0}^{\infty} c_i\varepsilon^{i/q}$ to $c_0$
(i.e. to the value at $0$ of the continuous extension of the corresponding curve). In terms of germs, the elements of $\R\langle\varepsilon\rangle_b$
are represented by semi-algebraic functions $(0,t_0) \rightarrow \R$ which can be extended continuously to $0$, and $\lim_\varepsilon$ maps such an element to the value
at $0$ of the continuous extension.

\vspace{5px}
\noindent
Finally, we state the Newton-Puiseux theorem~\cite[Theorem~3.1 of Chapter IV]{W78}, which describes the roots of $F(X,Y) = 0$ near $x = 0$ as a Puiseux series in $X$.  

\begin{proposition}[Theorems 3.2 and~4.1 and Section~4.2 of Chapter IV in~\cite{W78}]\label{Newton_Puiseux_Thm}
Let 
\begin{align*}
F(X,Y)=a_0 + a_1Y + \cdots + a_d Y^d \in \mathbb{C}\la \la X \ra\ra[Y],
\end{align*}
where $a_d \neq 0$. There exist $d$ (not necessarily distinct) Puiseux series $\psi_i(X) \in \mathbb{C} \la \la X \ra \ra$ for $i=1,\ldots,d$ such that 
\begin{align*}
F(X,Y)=a_d \prod_{i=1}^d (Y - \psi_i).
\end{align*}
\end{proposition}

\subsection{O-minimal geometry}\label{sec:o-minimal}
Let $\R$ be a real closed field. An \textit{o-minimal structure}~\cite{Michel2} over the field $\R$ is a sequence $\mathcal{S}(\R):=(\mathcal{S}_n)_{n \in \mathbb{N}}$, where $\mathcal{S}_n$ is a collection of subsets of $\R^n$, such that the following axioms hold:
\begin{itemize}
\item All algebraic subsets of $\R^n$ belong to $\mathcal{S}_n$.
\item If $A \in \mathcal{S}_m$ and $B \in \mathcal{S}_n$, then $A \times B \in \mathcal{S}_{m + n}$.
\item The class $\mathcal{S}_n$ is the Boolean algebra of subsets of $\R^n$, i.e., the complement and finite union and intersection of elements of $\mathcal{S}_n$ belong to $\mathcal{S}_n$.
\item If $A \in \mathcal{S}_{n+1}$, then $\pi(A) \in \mathcal{S}_n$, where $\pi: \R^{n+1} \to \R^n$ is the projection map to the first $n$ coordinates. 
\item The elements of $\mathcal{S}_1$ are precisely the finite union of points and intervals. 
\end{itemize}
\noindent
The elements of $\mathcal{S}_n$ are called \textit{definable sets} in the structure $\mathcal{S}(\R)$. Given two definable subsets $X \subset \R^n$ and $Y \subset \R^m$, a map $f : X \to Y$ is called definable if $\mathrm{graph}(f) \in \mathcal{S}_{m+n}$. An o-minimal expansion of $\mathbb{R}$ is called \textit{polynomially bounded} if for every definable function $f:\mathbb{R} \rightarrow \mathbb{R}$, there exist
$N \in \mathbb{N}$ and $c \in \mathbb{R}$, such that $|f(x)| < x^N$ for all $x > c$.

\vspace{5px}
\noindent
A classic example is the o-minimal structure over $\mathbb{R}$, where $\mathcal{S}_n$ is the class of semi-algebraic subsets of $\mathbb{R}^n$, which we denote by $\mathbb{R}_{\mathrm{sa}}$. We borrow the following examples from~\cite{Basu9,Dries2} (see also~\cite{Michel2,Dries,Dries3,Dries4,Wilkie2} for more instances of o-minimal structures).

\hide{
\begin{example}[Restricted analytic functions]
The o-minimal structure over $\mathbb{R}$ of restricted analytic functions, denoted by $\mathbb{R}_{\mathrm{an}}$, is defined by $\mathcal{S}_n$ being the image under the projection map $\mathbb{R}^{n+k} \to \mathbb{R}^n$ of subsets $\{(x,y) \in \mathbb{R}^{n} \times \mathbb{R}^k \mid P(x,y,e^x,e^y) = 0\}$, where $x=(x_1,\ldots,x_n)$, $y=(y_1,\ldots,y_k)$, and $P$ is an analytic function restricted to $[0,1]^{n+k}$. In this structure, $\mathcal{S}_n$ is called the class of \textit{globally sub-analytic} sets. 
\end{example}
}
\begin{example}[Restricted analytic functions]
The expansion of the real field by all restricted real analytic functions,
denoted by $\mathbb{R}_{\mathrm{an}}$, is an o-minimal structure.
Its definable sets are the globally subanalytic sets. 
\end{example}

\begin{example}[Restricted analytic functions with power functions]
The expansion of  $\mathbb{R}_{\mathrm{an}}$ by all power functions $x^r$ with $r \in \mathbb{R}$ defined as
\begin{align*}
    x^r = \begin{cases} x^r & x >0\\0 & x \le 0, \end{cases}
\end{align*}
is an o-minimal structure, which we denote by $\mathbb{R}^{\mathbb{R}}_{\mathrm{an}}$.
\end{example}

The o-minimal structures defined above are examples of polynomially bounded
o-minimal structures.
\begin{example}[Exponential functions]
The expansion of $\mathbb{R}$
by the globally defined  exponential function, denoted by $\mathbb{R}_{\mathrm{exp}}$, is an example of an o-minimal structure that is not polynomially bounded. 
\end{example}

\begin{example}[Exponential and restricted analytic functions]
The expansion of $\mathbb{R}$
by the restricted analytic functions as well as the globally defined  exponential function, denoted by $\mathbb{R}_{\mathrm{an},\mathrm{exp}}$, is another example of an o-minimal structure that is not polynomially bounded. 
\end{example}

\subsubsection{Cylindrical definable decomposition}
The notion of cylindrical definable decomposition~\cite{Loj2,Loj1} serves a key role in semi-algebraic and o-minimal geometry, and will be needed later in Sections~\ref{proof_of_critical_path_semi-algebraic} and~\ref{proof_of_o-minimal}. In what follows, we include the definition for definable sets and refer the reader to~\cite[Definition~5.1]{BPR06} for its semi-algebraic version.

\begin{definition}[Cell decomposition]
\label{def:cad}
Fixing the standard basis of $\R^n$, we identify for each $i,1 \leq i \leq n$, $\R^i$ with the span of the first
$i$ basis vectors.
Fixing an o-minimal expansion of $\R$,
a cylindrical definable decomposition (or simply a cell decomposition) of $\R$ is an $1$-tuple $(\mathcal{D}_1)$, where
$\mathcal{D}_1$ is a finite set of subsets of $\R$, each element being a point or an open interval,
which together gives a partition of $\R$. 
A cell decomposition of $\R^n$ is an $n$-tuple $(\mathcal{D}_1,\ldots,\mathcal{D}_n)$,
where each $\mathcal{D}_i$ is a decomposition of $\R^i$,
$(\mathcal{D}_1,\ldots, \mathcal{D}_{n-1})$ is a cell decomposition of $\R^{n-1}$,
and $\mathcal{D}_n$ is a finite set of definable subsets of $\R^n$ (called the cells of $\mathcal{D}_n$) giving a partition of $\R^n$ consisting 
of the following:
for each $C \in \mathcal{D}_{n-1}$, there is a finite set of definable continuous functions
$f_{C,1}, \ldots, f_{C,N_C}: C \rightarrow \R$ such that
$f_{C_1} < \cdots < f_{C,N_C}$,
and each element of $\mathcal{D}_n$ is either the graph of a function $f_{C,i}$ 
or of the form 
\begin{enumerate}[(a)]
    \item $\{ (x,t) \;\mid\; x\in C, \ t < f_{C,1}(x) \}$, 
    
    \item  $\{ (x,t) \;\mid\; x\in C, \ f_{C,i}(x) < t < f_{C,i+1}(x) \}$, 
    
    \item  $\{ (x,t) \;\mid\; x\in C, \ f_{C,N_C}(x) < t  \}$,     
    \item  $\{ (x,t) \;\mid\; x\in C \}$
\end{enumerate}
(the last case arising is if the set of functions $\{f_{C,i} | 1 \leq i \leq N_C \}$ is empty). We will say that the cell decomposition $(\mathcal{D}_1,\ldots,\mathcal{D}_n)$ is adapted to 
a definable subset $D$ of $\R^n$, if for each $C \in \mathcal{D}_n$, $C \cap D$ is either equal to 
$C$ or empty. 

\vspace{5px}
\noindent
If the definable functions $f_{C_i}$ are of $\mathcal{C}^k$ type for some positive integer $k$, then $(\mathcal{D}_1,\ldots,\mathcal{D}_n)$ is called a $\mathcal{C}^k$-cell decomposition. It is well-known that for every definable subset $D$ of $\R^n$ there exists a $\mathcal{C}^k$-cell decomposition adapted to $D$~\cite[Theorem~6.6]{Michel2}. 
\end{definition}

\subsection{Morse theory}\label{KKT_Conditions_Lagrange}
We will need the notions of critical points and projective critical points on a non-singular algebraic set in Sections~\ref{proof_on_V} and~\ref{proof_of_critical_path_semi-algebraic}. \textcolor{black}{These definitions are motivated by the corresponding notions for smooth
real manifolds and, when $\C=\mathbb C$, smooth complex manifolds
(see e.g.,~\cite{GH94,M63}).}

\subsubsection{Critical points}
Critical points of a complex polynomial can be described algebraically. The process for deriving these points in the real case follows a similar approach, with analogous techniques applied to the real algebraic set.

\vspace{5px}
\noindent
Let $F \in \C[X_1,\ldots,X_n]$. We call $x=(x_1,\ldots,x_n) \in \C^n$ a \textit{critical point} of $F$ if $x$ is a zero of $\Big\{\cfrac{\partial F}{\partial X_1},\ldots,\cfrac{\partial F}{\partial X_n}\Big\}$. The value of $F$ at $x$ is called a \textit{critical value}. A critical point $x$ is called \textit{non-degenerate} if the Hessian matrix
\begin{align*}
\Bigg[\frac{\partial^2 F}{\partial X_i \partial X_j}(x)\Bigg]_{\substack{i=1,\ldots,n\\ j=1,\ldots,n}}
\end{align*}
is non-singular. Further, the polynomial $F$ is called \textit{Morse} if its critical points are all non-degenerate.
\subsubsection{Projective critical points}\label{sec:projective_critical}
Let $F \in \C[X_1,\ldots,X_n]$. We call $x=(x_0:\cdots:x_n) \in \mathbb{P}_n(\C)$ a \textit{projective critical point of $F$} if $x$ is a projective zero of $\Big\{\Big(\cfrac{\partial F}{\partial X_1}\Big)^H,\ldots,\Big(\cfrac{\partial F}{\partial X_n}\Big)^H\Big\}$, where the homogeneous polynomial $\Big(\cfrac{\partial F}{\partial X_j}\Big)^H \in \C[X_0,\ldots,X_n]$ is obtained from homogenization of $\cfrac{\partial F}{\partial X_j}$. A projective critical point $x$ is called \textit{non-degenerate} if $x$ is a non-singular projective zero of $\Big\{\Big(\cfrac{\partial F}{\partial X_1}\Big)^H,\ldots,\Big(\cfrac{\partial F}{\partial X_n}\Big)^H\Big\}$ (see Definition~\ref{non_singular_point}).

\subsubsection{Critical points on a non-singular algebraic set}\label{sec:KKT_non-singular}
\textcolor{black}{For the complex-analytic discussion in this section, we first take
$\C=\mathbb C$. Let $F\in\mathbb C[X_1,\ldots,X_n]$ and $V_{\mathbb C}:=\zero(\mathcal P,\mathbb C^n)$, where
\begin{align*}
\mathcal P:=\{P_1,\ldots,P_s\}
\subset\mathbb C[X_1,\ldots,X_n],
\end{align*}
and suppose that $V_{\mathbb C}$ is non-empty and non-singular.
Thus, $V_{\mathbb C}$ is a complex sub-manifold of $\mathbb C^n$.} 

\vspace{5px}
\noindent
Let $\bar{x} \in V_{\mathbb C}$, and assume without loss of generality, that the leading $s$-principal submatrix of the Jacobian of $\mathcal{P}$, denoted by $[J(\{P_1,\ldots,P_s\})(\bar{x})]_{s \times s}$ is non-singular. Then, by the Implicit Function Theorem~\cite[Page~19]{GH94}, $\bar{x}$ has a local coordinate system $(X_{s+1},\ldots,X_n)$ in a sufficiently small neighborhood, i.e., there exist an open subset $U \subset \mathbb C^{n-s}$ and holomorphic functions $\phi_i:U \to \mathbb C$ for $i=1,\ldots,s$ such that 
\begin{align*}
\Phi(\bar{x}_{s+1},\ldots,\bar{x}_n):=(\phi_1(\bar{x}_{s+1},\ldots,\bar{x}_n),\ldots,\phi_{s}(\bar{x}_{s+1},\ldots,\bar{x}_n),\bar{x}_{s+1},\ldots,\bar{x}_n) \in V_{\mathbb C}.
\end{align*}
Then $\bar{x}$ is called a \textit{critical point} of $F$ on $V_{\mathbb C}$ if it satisfies the equations
\begin{align}\label{critical_point_conditions_on_manifold}
\begin{cases}
   \displaystyle \sum_{i=1}^{s} \cfrac{\partial F}{\partial X_{i}} \cfrac{\partial \phi_i}{\partial X_j} + \cfrac{\partial F}{\partial X_j} = 0, \ \ j=s+1,\ldots,n,\\
    P_1=\cdots =P_s= 0,
\end{cases}
\end{align}
and $\bar{x}$ is called \textit{non-degenerate} if the Jacobian of~\eqref{critical_point_conditions_on_manifold} at $\bar{x}$ is non-singular. Further, $F$ is called \textit{Morse} on $V_{\mathbb C}$ if its critical points are all non-degenerate.

\vspace{5px}
\noindent
For the purpose of our derivations in Theorems~\ref{Morse_on_smooth_hypersurface}-\ref{Morse_on_transversal}, we provide an equivalent definition of critical points. By taking the derivatives of $P_i= 0$, we get 
\begin{align*}
\sum_{j=1}^{s} \frac{\partial P_i}{\partial X_j} \frac{\partial \phi_j}{\partial X_k} +  \frac{\partial P_i}{\partial X_k}  = 0, \quad i=1,\ldots,s, \ \  k =s+1,\ldots,n.
\end{align*}
Letting
\begin{align}\label{Lagrange_multiplier}
\bar{u} :=[\bar{u}_1,\ldots,\bar{u}_{s}]^T= \big([J(\{P_1,\ldots,P_s\})(\bar{x})]_{s \times s}\big)^{-T} \Big[\cfrac{\partial F}{\partial X_{1}}(\bar{x}),\ldots,\cfrac{\partial F}{\partial X_{s}}(\bar{x})\Big]
\end{align}
then $(\bar{x},\bar{u})$ satisfies  
\begin{align}\label{KKT_system_complex}
\begin{cases}
\displaystyle \cfrac{\partial F}{\partial X_j} - \sum_{i=1}^{s} U_i \cfrac{\partial P_i}{\partial X_{j}} = 0, \ \ j=1,\ldots,n,\\
    P_1=\cdots =P_s= 0,
\end{cases}
\end{align}
where $u_i$ are called \textit{Lagrange multipliers} and~\eqref{KKT_system_complex} is called the \textit{Lagrange multiplier equations} for critical points of $F$ on $V_{\mathbb C}$. Accordingly, $(\bar{x},\bar{u})$ is called a Lagrange critical pair. Since $V_{\mathbb C}$ is non-singular, then $\bar{u}$ is the unique Lagrange multiplier for $\bar{x}$ that satisfies~\eqref{KKT_system_complex}. 

\begin{remark}
Our proof of the above equivalence property is not entirely new; it was previously established in~\cite[Theorem~A]{F82} (which has not received adequate attention in the optimization literature) in the context of \NO~and KKT points. For the sake of completeness, we provide a proof of the statement for critical points of $F$ on $V$. 
\end{remark}
\noindent
\textcolor{black}{The characterization of critical points by the Lagrange multiplier
equations~\eqref{KKT_system_complex} is algebraic.
Consequently, by the Lefschetz principle~\cite[Theorem~1.26]{BPR06}, this characterization remains
valid over any algebraically closed field of characteristic zero, and
in particular over $\C=\R[i]$ for an arbitrary real closed
field $\R$. We will use this algebraic characterization in what follows. In particular, we will show in Theorem~\ref{Morse_on_smooth_hypersurface} that $\bar{x}$ is a non-degenerate critical point of $F$ on $V_{\C}$ if and only if $(\bar{x},\bar{u})$ is a non-singular zero of~\eqref{KKT_system_complex}.}

\subsubsection{Complex projective critical points on a non-singular algebraic set}\label{sec:projective_KKT}
Our conditions in Theorems~\ref{constrained_real_isolated_critical_points}-\ref{constrained_real_non-degenerate_critical_points} rely on the notions of complex projective critical points and complex projective Lagrange critical pairs, which are defined as follows. 

\begin{definition}\label{projective_KKT_point}
Let $F$ and $V_{\C}=\zero(\mathcal{P},\C^n)$, where
\begin{align*}
\mathcal{P}:=\{P_1,\ldots,P_s\} \subset \C[X_1,\ldots,X_n],
\end{align*}
and $V_{\C}$ is non-empty and non-singular. Then $(x,u)=((x_0:\cdots:x_n),(u_0:\cdots:u_s))$ is called a complex \textit{projective Lagrange critical pair} of $F$ in $\mathbb{P}_n(\C) \times \mathbb{P}_s(\C)$ on $V_{\C}$ if it is a zero of 
\begin{align}\label{KKT_system_complex_homogeneous}
\begin{cases}
\displaystyle \Bigg(\cfrac{\partial F}{\partial X_j} - \sum_{i=1}^{s} U_i \cfrac{\partial P_i}{\partial X_{j}}\Bigg)^H = 0, \ \ j=1,\ldots,n,\\
    P^H_1=\cdots =P^H_s= 0,
\end{cases}
\end{align}
where $(\cdot)^H$ in~\eqref{KKT_system_complex_homogeneous} means the bi-homogenization of polynomials with respect to $X$ and $U$, and the equations in~\eqref{KKT_system_complex_homogeneous} are bi-homogeneous polynomials in $\C[X_0,\ldots,X_n;U_0,\ldots,U_s]$. Further, $x$ is called a complex \textit{projective critical point} of $F$ in $\mathbb{P}_n(\C) \times \mathbb{P}_s(\C)$ on $V_{\C}$. A complex projective critical point $x$ is called \textit{non-degenerate} if $(x,u)$ is a non-singular zero of~\eqref{KKT_system_complex_homogeneous} for some $u \in \mathbb{P}_s(\C)$. 
\end{definition}

\begin{remark}\label{infinitely_many_Lagrange_multipliers}
Note that the non-singularity of $V_{\mathrm{C}}$ implies that a complex projective critical point $x$ with $x_0 = 1$ has a unique vector of Lagrange multipliers $u = (u_0:\cdots:u_s)$, and $u$ satisfies $u_0 =1$, i.e., it is impossible to have a complex projective Lagrange critical pair $(x,u)$ with $x_0 = 1$ and $u_0 = 0$. However, it is possible to have a complex projective critical point $x$ with $x_0 = 0$ that admits infinitely many Lagrange multipliers. For example, consider $F = X_1$ and $V=\zero(X_1^2X_2^2-1,\mathbb{R}^2)$, for which the complex projective Lagrange multiplier equations are
\begin{align*}
\begin{cases}
U_0X_0^3-2U_1 X_1X_2^2 &= 0,\\
-2U_1X_1^2X_2 &= 0,\\
    X_1^2 X_2^2-X_0^4  &= 0.
\end{cases}
\end{align*}
It is easy to verify that $(0:0:1)$ and $(0:1:0)$ are degenerate projective critical points that admit infinitely many Lagrange multipliers. 
\end{remark}
\subsubsection{Projective KKT points for~\PO}\label{sec:KKT_PO}
The concept of a KKT point is a standard tool in \NO~and does not require  any smoothness assumptions on the feasible set~\cite{G10}. More specifically, a KKT point of~\PO~is a solution of 
\begin{equation}\label{KKT_conditions_PO}
\begin{aligned}
    \frac{\partial f}{\partial x_j} -  \sum_{i=1}^r u_i \frac{\partial g_i}{\partial x_j} &= 0, \qquad j=1,\ldots,n,\\
    u_i g_i &= 0, \quad i=1,\ldots,r,\\
    u_i,g_i(x) &\ge 0, \quad i=1,\ldots,r,
\end{aligned}
\end{equation}
where $u_i$ is a Lagrange multiplier associated to $g_i$. A KKT point $(x,u)$ of~\PO~is called \textit{strictly complementary} if $g_i(x) + u_i > 0$ for all $i=1,\ldots,r$. Under certain regularity conditions, KKT conditions are necessary for a local optimal solution of~\PO~(see e.g.,~\cite{G10}). In general, however,~\eqref{KKT_conditions_PO} may have no solution (see Example~\eqref{unconstrained_cusp}). 

\vspace{5px}
\noindent
In order to describe limit points of critical paths in Section~\ref{proof_of_critical_path_semi-algebraic}, we extend the notion of KKT points of~\PO~ (which may not always exist) to projective KKT points. Let us ignore the sign conditions $u_i \ge 0$ and bi-homogenize the polynomials in terms of $x$ and $u$:
\begin{equation}\label{Projective_KKT_conditions_PO}
\begin{aligned}
    u_0 F_j-  \sum_{i=1}^r u_i  G_{ij} &= 0, \qquad j=1,\ldots,n,\\
    u_i g^H_i &= 0, \quad i=1,\ldots,r,
\end{aligned}
\end{equation}
where for each $j$, $F_j,G_{1j},\ldots,G_{rj} \in \R[X_0,\ldots,X_n]$ and for each $i$, $g_i^H \in \R[X_0,\ldots,X_n]$ are homogeneous polynomials. A zero $(x,u)$ of~\eqref{Projective_KKT_conditions_PO} is called a \textit{projective KKT point} of~\eqref{poly_optim}. \textcolor{black}{Motivated by the classical notion of strict complementarity, we introduce the following projective analog. A projective KKT point $(x,u)$ is called \emph{strictly complementary} if
\begin{align*}
    \bigl(u_i,g_i^H(x)\bigr)\neq(0,0),
    \qquad i=1,\ldots,r.
\end{align*}
In view of the complementarity equations
$u_i g_i^H(x)=0$, this means that exactly one of $u_i$ and
$g_i^H(x)$ is nonzero for every $i=1,\ldots,r$. For a projective KKT point arising from an ordinary KKT point, this
definition reduces to the classical strict complementarity condition.}

\vspace{5px}
\noindent
The following proposition shows that as long as $f$ has a critical point on $S_=$, \PO~ always admits a projective KKT point.
\begin{proposition}\label{existence_projective_KKT_points}
Suppose that $f$ has a critical point on $S_=$. Then \PO~ has a projective KKT point.    
\end{proposition}
\begin{proof}
Let $\bar{x} \in S_=$ be a critical point of $f$ on $S_=$. \textcolor{black}{For simplicity of notation, we assume that all constraints are active at $\bar{x}$, i.e., $g_i(\bar{x}) =0$ for $i=1,\ldots,r$. In the general case, the following argument is applied to the active constraints, while the Lagrange multipliers corresponding to the inactive constraints are set equal to zero}. 

\vspace{5px}
\noindent
If $\mathcal{Q}=\{g_1,\ldots,g_r\} \subset \R[X_1,\ldots,X_n]$ is in general position, then by the canonical Whitney stratification of $S_=$ (see Proposition~\ref{Whitney_stratification_union_of_sets}) there exist unique Lagrange multipliers $\bar{u}_i$ such that $(\bar{x},\bar{u})$ is a KKT point of $f$ on $\zero(\mathcal{Q},\R^n)$ (see Section~\ref{sec:KKT_non-singular}). Then it is easy to see that $\big((1:\bar{x}_1:\cdots:\bar{x}_n),(1,\bar{u}_1:\cdots:\bar{u}_r)\big)$ is a projective KKT point of \PO. \textcolor{black}{Otherwise, if $\mathcal{Q}$ is not in general position, then there exist
a non-empty subset $I\subset\{1,\ldots,r\}$ and a point
$x'\in \zero(\{g_i\}_{i\in I},\R^n)$ such that
$\{dg_i(x')\}_{i\in I}$ are linearly dependent. Hence, there exist $\lambda_i\in\R$, $i\in I$, not all zero, such that
\begin{align*}
    \sum_{i\in I}\lambda_i\,dg_i(x')=0.
\end{align*}
Setting
\begin{align*}
    u_0=0,\qquad
    u_i=\lambda_i\ \ (i\in I),\qquad
    u_j=0\ \ (j\notin I),
\end{align*}
we obtain a projective KKT point of \PO.}
\end{proof}

\begin{remark}[Projective KKT versus Fritz-John conditions]
When $\mathcal{Q}$ is in general position, every critical point of $f$ on
$S_=$ gives rise to a projective KKT point, as shown in the proof
of Proposition~\ref{existence_projective_KKT_points}. When $\mathcal{Q}$ is not in general position, projective
KKT points may also arise from linear dependence among the
differentials of the active constraints, independently of the
criticality of $f$. In this sense, projective KKT points can be
regarded as an extension of the classical \textit{Fritz-John} (FJ)
conditions for \PO~(see e.g.,~\cite{G10}). FJ conditions are
necessary, though weaker than KKT conditions, for the existence
of a local optimum~\cite[Theorem~9.4]{G10}. Unlike a KKT point, an FJ point always exists
as long as $f$ has a local optimum on $S$.
\end{remark}

\subsection{Critical points on stratified algebraic sets}\label{stratified_Morse}
\textcolor{black}{In order to define critical points on singular algebraic sets, we
consider decompositions into smooth strata. We first recall Whitney
stratifications and then introduce the canonical stratification of a
real algebraic set that will be used throughout the paper.} We adopt the definition of a Whitney stratified space in~\cite{GM}. Let $\mathscr{S}$ be a partially ordered set. We define an $\mathscr{S}$-decomposition of a topological space $X$ as a locally finite collection of locally closed subsets $Z_i$ of $X$ such that 
\begin{itemize}
  \item $X = \bigcup_{i \in \mathscr{S}} Z_i$.
  \item $Z_i \cap \bar{Z}_j \neq 
  \emptyset \ \Leftrightarrow \ Z_i \subset \bar{Z}_j \ \Leftrightarrow i \le j$.
\end{itemize}
\hide{
The latter condition is denoted by $Z_i < Z_j$.
}

 \begin{definition}[Whitney stratification]
  Let $X$ be a closed subset of a smooth manifold $M$ with an $\mathscr{S}$-decomposition $X = \cup_{i \in \mathscr{S}} Z_i$. Then $X$ is called a \emph{Whitney stratified} space if
  \begin{itemize}
      \item Each $Z_i$ is a locally closed smooth sub-manifold of $M$.
      \item (Whitney's Condition A): Let $Z_i \subset \bar{Z}_j$ be two strata of $X$ and let $\{x_i\} \in Z_j$ be a sequence of points converging to $x \in Z_i$. If the tangent space $T_{x_i} Z_j$ converge to a subspace $V$ of $T_x M$, then $T_x Z_i \subseteq V$. 
      \item (Whitney's Condition B): Let the hypotheses of Whitney's Condition A hold and let $\{y_i\} \in Z_i$ be a sequence of points converging to $x$. If the sequence of one-dimensional subspaces $\mathbb{R}(y_i - x_i)$ (by choosing a local coordinate system around $x$) converges to a line $\ell$, then $\ell \subset V$.
  \end{itemize}
 The last two conditions are called Whitney regularity conditions.
 \end{definition}
\noindent
In this paper, the notions of criticality and non-degeneracy are
defined stratumwise and do not require the Whitney regularity conditions.
\textcolor{black}{
\begin{definition}[Critical points on a stratified set]\label{Generalized_Morse_Func}
Let $X\subset M$ be equipped with a stratification
\begin{align*}
    X = \bigcup_{i \in \mathscr{S}} Z_i
\end{align*}
into smooth locally closed submanifolds of a smooth manifold $M$. A function $f:X\to\mathbb{R}$ is called smooth if $f=g|_X$ for some
smooth function $g:M\to\mathbb{R}$. A point $x\in Z_\alpha$ is called a
\emph{critical point of $f$ on $X$ with respect to the
stratification} if
\begin{align*}
    d(g|_{Z_\alpha})(x)=0,
\end{align*}
or equivalently,
\begin{align*}
    dg(x)|_{T_xZ_\alpha}=0.
\end{align*}
Such a critical point is called \emph{non-degenerate} if it is a
non-degenerate critical point of $g|_{Z_\alpha}$ in the usual
sense, i.e., if the Hessian of $g|_{Z_\alpha}$ at $x$ is
non-singular.
\end{definition}}
\textcolor{black}{
\begin{remark}
Our notion of non-degeneracy is stratumwise and does not require
the stratification to satisfy the Whitney regularity conditions.
When the stratification is Whitney, a stratified Morse function in
the usual sense imposes additional conditions, such as distinct
critical values and a normal non-degeneracy condition
(see~\cite{GM}). These additional conditions are not needed in
this paper.
\end{remark}}
\subsubsection{Canonical stratification}\label{sec:Whitney_strat}
\begin{definition}[Regular and irregular points]
\label{regular_points}
Let $\mathcal{P} \subset \R[X_1,\ldots,X_n]$ be a finite set of polynomials and $V = \zero(\mathcal{P},\R^n)$. We define a regular point of $V$ according to~\cite[Definition~3.3.4]{BCR98} as follows. Let $\mathcal{I}(V) \subset \R[X_1,\ldots,X_n]$ be the vanishing ideal of $V$ and assume that $\mathcal{I}(V) = (R_1,\ldots,R_k)$, where $R_1,\ldots,R_k \in \R[X_1,\ldots,X_n]$. Then $\bar{x} \in V$ is called a \textit{regular point} of $V$ if the rank of $J(\{R_1,\ldots,R_k\}) = n - \dim(V)$. Otherwise, $\bar{x}$ is called an \textit{irregular point}. A real variety with no irregular point is called \textit{smooth}. 
\end{definition}
\begin{remark}
By definition~\ref{regular_points}, a non-singular zero of $\mathcal{P}$ (see Definition~\ref{non_singular_point}) is a regular point of $V$. 
\end{remark}
\hide{
\begin{remark}
Our notion of non-singularity defined in Definition~\ref{non_singular_point} is a special case of~\cite[Definition~3.3.4]{BCR98}, and unlike~\cite[Definition~3.3.4]{BCR98}, our notion is dependent on the given description of the algebraic set. To avoid confusion, non-singular and singular points in the sense of ~\cite[Definition~3.3.4]{BCR98} are referred to as regular and irregular points, and they are only utilized for stratification of varieties (see Section~\ref{stratified_Morse}).  
\end{remark}}

\begin{definition}[Canonical stratification of $V$]
\label{def:canonical-Whitney}
A canonical stratification of an algebraic set $V$ is the tuple
$(Z_i \subset V)_{i \geq 0}$, defined inductively by: 
\begin{eqnarray*}
V^{(0)} &=& V, \\
V^{(i+1)} &=& \Ireg(V^{(i)}), \ \ i \geq 0, \\
Z_i &=& V^{(i)} \setminus V^{(i+1)}, \ \ i \geq 0.
\end{eqnarray*}
\hide{
as follows: Let $V^{(0)} := V$ and $V^{(1)} = \Ireg(V)$, where $\Ireg(V^{(1)})$ is the set of irregular points of $V$.
\todo{I think we do not need to insist on the $k$ where it terminates}
Repeatedly, define $V^{(i)} = \Ireg(V^{(i-1)})$ for $i=2,\ldots,k$ and continue until $V^{(k)} = \emptyset$. Then the canonical Whitney stratification of $V$ consists of $k$ strata defined as follows
\begin{equation}\label{algebraic_set_canonical_stratification}
\begin{aligned}
Z_i&:=V^{(i-1)} \setminus V^{(i)}, \quad i=1,\ldots,k.
\end{aligned}
\end{equation}
}
\end{definition}
\begin{example}
\textcolor{black}{The polynomial $F=X_1X_2^2$ is not Morse, because $(0,0)$ is a degenerate critical point (see Fig.~\ref{Morse_Examples}). If we consider the canonical stratification of the cusp $V = \zero (X_1^3-X_2^2, \mathbb{R}^2)$ as 
\begin{align*}
     Z_0 = \{X_1^3 - X_2^2= 0, \ X_2 > 0\} \bigcup \{X_1^3 - X_2^2 = 0, \ X_2 < 0\}, \ \ Z_1 = \{(0,0)\},
\end{align*}
then we can observe that $F$ has no critical point on $Z_0$, and thus all critical points of $F$ on $V$ with respect to its canonical stratification are non-degenerate.}

\begin{figure}
\centering
 \begin{minipage}[c]{0.4\textwidth}
\includegraphics[height=1.7in]{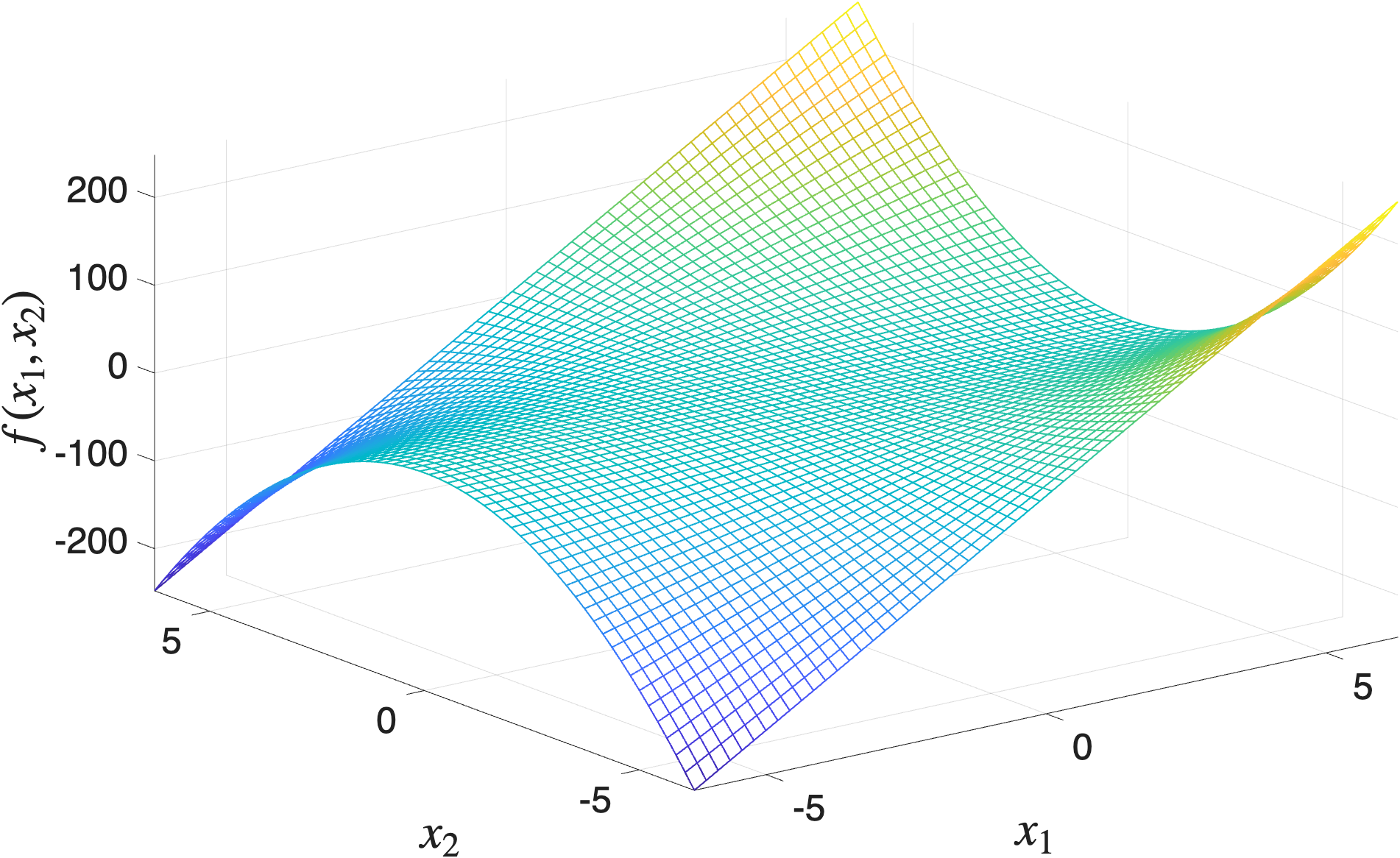}  
\end{minipage}\hfill
 \begin{minipage}[c]{0.5\textwidth}
 \includegraphics[height=2.8in]{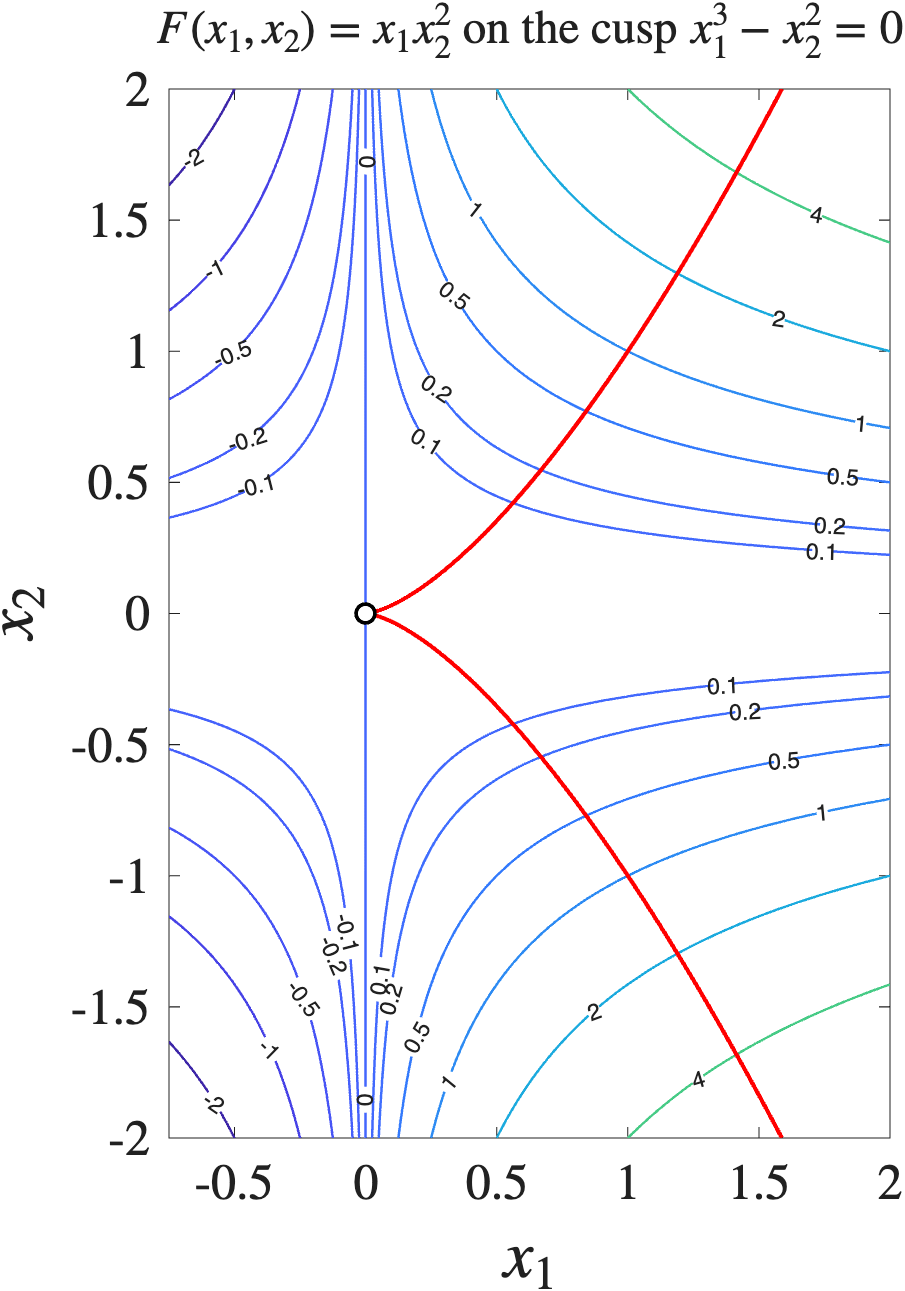}
\end{minipage}
\caption{$F=X_1X_2^2$ has no degenerate critical point on $V = \zero (X_1^3-X_2^2, \mathbb{R}^2)$ with respect to its canonical stratification, but it has
a degenerate critical point on $\mathbb{R}^2$.}
\label{Morse_Examples}
\end{figure}
\end{example}

\hide{
\begin{proposition}
If $V=\zero(\prod_i P_i,\R^n)$, where $\mathcal{P}=\{P_1,\ldots,P_s\} \subset \R[X_1,\ldots,X_n]$ is in general position. Then the canonical Whitney stratification of $V$ consists of $k+1$ strata defined as follows
\begin{equation}\label{algebraic_set_canonical_stratification}
\begin{aligned}
Z_1&:=V\setminus \bigcup_{I \subset \{1,\ldots,s\}}\zero(\mathcal{P}_{I},\R^n),\\
Z_2&=\quad i=1,\ldots,k,\\
Z_{k+1}&:=\Ireg(V^{(k)}).
\end{aligned}
\end{equation}
\end{proposition}
}
\noindent
If $\mathcal{P}$ is in general position, then the canonical  stratification can be explicitly characterized.
\begin{proposition}[Canonical Whitney stratification]\label{Whitney_stratification_union_of_sets}
Let 
\[
V=\zero\Big(\prod_{i=1}^s P_i,\R^n\Big) = \bigcup_{i=1}^s \zero(P_i,\R^n),
\]
where $\mathcal{P}=\{P_1,\ldots,P_s\} \subset \R[X_1,\ldots,X_n]$ is in general position. Then the canonical stratification of $V$ 
(see Definition~\ref{def:canonical-Whitney}) $(Z_i)_{i \geq 0}$ is defined by:
\begin{equation}\label{algebraic_set_canonical_stratification}
Z_i = \bigcup_{I \subset \{1,\ldots,s\}, \card(I) = i+1}
\left(\bigcap_{j \in I} \zero(P_j,\R^n)\right) \setminus 
\bigcup_{I \subset \{1,\ldots,s\}, \card(I) = i+2}
\left(\bigcap_{j \in I}  \zero(P_j,\R^n)\right).
\end{equation}
\textcolor{black}{Further, the canonical stratification satisfies the Whitney regularity conditions.}
\end{proposition}

\begin{proof}
The proof follows immediately from Definitions~\ref{def:general-position} and~\ref{def:canonical-Whitney}.
\end{proof}
\noindent
\textcolor{black}{We note that without the general position condition, the canonical stratification of Definition~\ref{def:canonical-Whitney} may fail to satisfy the Whitney regularity conditions (see e.g.,~\cite[Page~206]{N11} or~\cite[Page~237]{BCR98}).}
 \begin{example}\label{Whitney_umbrella}
\textcolor{black}{Consider the \textit{Whitney umbrella} defined by $V=\zero(X_1^2-X_3X_2^2,\mathbb{R}^3)$ which contains the singular stratum $Z:=\zero(\{X_1,X_2\},\mathbb{R}^3)$. Let $Z'=V \setminus Z$. It is easy to see that $Z \subset \bar{Z'}$. Let $\{(0,\xi_k,0)\} \subset Z'$ where $\xi_k \to 0$. One can check that 
\begin{align*}
T_{\textbf{0}} Z = \zero(\{X_1,X_2\},\mathbb{R}^3) \not \subseteq \lim_{k \to \infty} T_{(0,\xi_k,0)} Z'=\zero(X_3,\mathbb{R}^3),
\end{align*}
violating the Whitney's condition A.}
 \end{example}
\hide{
Given a semi-algebraic set $S$ defined by a Boolean combinations of polynomials $\{P_1,\ldots,P_s\}$ we define the strata as follows: let $V = \zero(\prod_{i=1}^s P_i,\R^n)$ and stratify $V$ according to~\eqref{algebraic_set_canonical_stratification}. Now, the canonical Whitney stratification of $S$ consists of $k+2$ stratum defined as follows:
\begin{equation}\label{semi-algebraic_set_canonical_stratification}
\begin{aligned}
Z_1&:=S \setminus V,\\
Z_{i+1}&:=V^{(i)} \cap S, \quad i=1,\ldots,k,\\
Z_{k+2}&:=\Ireg(V^{(k)}) \cap S.
\end{aligned}
\end{equation}
}

\subsection{Proofs of existence and convergence of critical points on $V_{\xi}$}\label{proof_on_V}
\subsubsection{Proofs of Theorems~\ref{bounded_fibers}-\ref{bounded_connected_component}, and 
Corollaries~\ref{bounded_critical_points}-\ref{existence_critical_points}}

\begin{proof}[Proof of Theorem~\ref{bounded_fibers}]
Suppose that $V_\xi$ is unbounded over $\R$. Then, there exists 
$x = (x_1,\ldots,x_n)  \in V_\xi \subset \R\la\xi\ra^n$, with 
$\lim_\xi \left(\frac{1}{||x||}\right) = 0$,
and satisfying
$P_{i}(x) = \xi_i, 1 \leq i \leq s$.
We denote  
\[
y = (y_1,\ldots,y_n) = \lim_\xi (x_1/||x||,\ldots,x_n/||x||).
\]

\hide{
Since, $\lim_\xi \left(\frac{1}{||x||}\right) = 0$, we have that
\begin{eqnarray*}
  \lim_\xi (1:x_1:\cdots:x_n) &=&  \lim_\xi(1/||x||: x_1/||x||: \cdots: x_n/||x||) \\
  &=& 
(0:y_1:\ldots:y_n) \in \PP_n(\R).  
\end{eqnarray*}
}
\noindent
For $1 \le i \leq s$ and $d \geq 0$,
let $P_{i,d}$ denote the homogeneous part of $P_i$ of degree $d$, and let 
$d_i = \deg(P_i)$. We have
\begin{eqnarray*}
 \xi_i &=& P_i(x) \\
 &=& ||x||^{d_i} P_{i,d_i}(x/||x||) + 
 ||x||^{d_i-1}P_{i,d_i-1}(x/||x||) + \cdots + 
P_{i,0}(x/||x||).
 \end{eqnarray*}
\noindent
 Dividing by $||x||^{d_i}$ we get that,
 \[
 P_{i,d_i}(x/||x||) + (1/||x||)P_{i,d_i-1}(x/||x||) + \cdots + (1/||x||^{d_i}) P_{i,0}(x/||x||) = \xi_i/||x||^{d_i}.
 \]
\noindent
Applying $\lim_\xi$ to both sides, and using the fact that the coefficients of $P_i$ belong to $\R$, we obtain  
\[
\lim_\xi (P_{i,d_i}(x_1/||x||,\ldots,x_n/||x||)) = 
P_{i,d_i}(\lim_\xi(x_1/||x||,\ldots,x_n/||x||)) = P_{i,d_i}(y_1,\ldots,y_n) = 0,
\]
which in turn 
implies that $(0:y_1:\cdots:y_n) \in \zero(P_i^H, \PP_n(\R))$ for each $i, 1 \leq i \leq s$, which is a contradiction.
\end{proof}

\begin{proof}[Proof of Corollary~\ref{bounded_critical_points}]
Since $V$ has at least one non-singular point, then $V_{\xi}$ is non-empty by Proposition~\ref{convergence_of_varieties}. Furthermore, by Theorem~\ref{bounded_fibers}, $V_{\xi}$ is bounded and its image under $F$, as a continuous semi-algebraic function, is closed and bounded (see~\cite[Theorem~3.20]{BPR06} or~\cite[Theorem~2.5.8]{BCR98}). \textcolor{black}{Thus, $F$ attains an extremum on $V_\xi$, which is a critical point of $F$ on $V_\xi$, since $V_\xi$ is non-singular. The boundedness of $\Crit(V_\xi,F)$ follows from that of $V_\xi$}.
\end{proof}

\vspace{5px}
\noindent
From now on, we use the following notation for open and closed balls.
\begin{notation}
We denote by $B(x,r)$ and $\bar{B}(x,r)$ open and closed balls in $\R^n$ with center $x$ and radius $r > 0$.  
\end{notation}

\begin{proof}[Proof of Theorem~\ref{bounded_connected_component}]
\hide{
By definition, $S_{\xi}$ is described by 
\begin{align*}
    P_i - \xi \ge 0, \ \ P_j - \xi = 0, \quad \forall P_i,P_j \in \mathcal{P}
\end{align*}
using the same quantifier-free formula that defines $S$.
Let $B^n_{r}(0)$ be an open ball such that $D \subset B^n_{r}(0)$ for some $r > 0$. Suppose that $D_{\xi}$ is unbounded and let $x \in D_{\xi} \setminus B^n_{2r}(0)$. Since $D_{\xi}$ is semi-algebraically connected, there exists a semi-algebraic function $\gamma: [0,1] \to D_{\xi}$ such that $\gamma(0) = x(\xi)$ and $\gamma(1) = x$. Let $x' \in \partial B^n_{2r}(0)$ be the point at which $\gamma$ intersects $\partial B^n_{2r}(0)$ for the first time, say $\gamma(t') = x'$, and consider the restriction of $\gamma$ to $[0,t']$. We note that $\gamma([0,t']) \subset D_{\xi} \cap B^n_{2r}$ and thus $\lim_{\xi} (\gamma([0,t'])) \in S \cap B^n_{2r}(0)$. Further, by~\cite[Proposition~12.43]{BPR06}, $\lim_{\xi} (\gamma([0,t']))$ is semi-algebraically connected, which by $\lim_{\xi} \gamma(0) = \bar{x} \in D$ implies that $\lim_{\xi} (\gamma([0,t'])) \subset D$. However, this is a contradiction, because $\lim_{\xi} (\gamma(t')) = \lim_{\xi}(x') \in \partial B^n_{2r}(0)$ while $D \subset B^n_{r}(0)$. Therefore $D_{\xi}$ must be bounded and using ~\cite[Proposition~12.43]{BPR06} again, we conclude that $\lim_{\xi} (D_{\xi}) \subset D$.
}
Since $D$ is bounded, there exists $r > 0$ such that $D \subset \bar{B}(\mathbf{0},r)$. 
Suppose $D_\xi$ is not bounded over $\R$. 
Then $D_\xi$ has a non-empty intersection with $\R\la\xi\ra^n \setminus B(\mathbf{0},2r)$.
Let 
$y_\xi \in D_\xi$, with $||y_\xi|| > 2r$. Then, there exists a semi-algebraic path, $\gamma_\xi: [0,1] \rightarrow D_\xi$, with $\gamma_\xi(0) = x_\xi, \gamma_\xi(1) = y_\xi$.

\vspace{5px}
\noindent
Since $\lim_\xi (x_\xi) = x$ and $||x|| \leq r, r \in \R, r>0$,
we have that
$||x_\xi||  <  2r$,  while $||y_\xi|| > 2r$. 
By the Semi-algebraic Intermediate Value Theorem~\cite[Proposition~3.4]{BPR06}, there exists
$t_\xi \in \R\la\xi\ra, 0 < t_\xi < 1$, such that
$||\gamma_\xi(t_\xi)|| =  2r$, and for all $t$ satisfying
$0 \leq t < t_\xi$, $||\gamma_\xi(t)|| < 2r$. Then the semi-algebraic set 
$\Gamma_\xi  := \gamma_\xi([0,t_\xi])$ is semi-algebraically connected 
and contained in $D_\xi \cap \bar{B}(\mathbf{0},2r)$. Using \cite[Proposition~12.43]{BPR06}, 
$\Gamma := \lim_\xi (\Gamma_\xi)$ is semi-algebraically connected.
Moreover, $\Gamma \subset V$, and $x = \lim_\xi (x_\xi) \in \Gamma$.
Therefore, $\Gamma \subset D$. 
However, 
$\lim_\xi \gamma_\xi(t_\xi) \in \Gamma \subset D$,
while $||\lim_\xi \gamma_\xi(t_\xi)||= ||\gamma_\xi(t_\xi)|| = 2r$,
which contradicts the fact that $D \subset \bar{B}(\mathbf{0},r)$. 

\vspace{5px}
\noindent
Thus, $D_\xi$ is bounded over $\R$. Again using \cite[Proposition~12.43]{BPR06}, $\lim_\xi (D_\xi)$ is semi-algebraically connected 
and $x \in \lim_\xi (D_\xi)$. This implies that $\lim_\xi (D_\xi) \subset D$.
 \end{proof}

\begin{proof}[Proof of Corollary~\ref{existence_critical_points}]
Analogous to the proof of Corollary~\ref{bounded_critical_points}, $V_{\xi}$ is non-empty by Proposition~\ref{convergence_of_varieties}. Furthermore, by Theorem~\ref{bounded_connected_component}, $V_{\xi}$ has a closed and bounded semi-algebraically connected component $D_{\xi}$ whose image under $F$ is closed and bounded. \textcolor{black}{Thus, $F$ attains an extremum on $
D_\xi$, which is a local extremum on $V_\xi$, and thus a critical point of $F$ on $V_\xi$, because $V_\xi$ is non-singular. The boundedness of the critical point follows from that of $D_\xi$}.   
\end{proof}

\subsubsection{Proofs of Theorems~\ref{constrained_real_isolated_critical_points}-\ref{constrained_real_non-degenerate_critical_points}}
Recall the notions of complex projective critical points and complex projective Lagrange critical pairs introduced in Section~\ref{KKT_Conditions_Lagrange}. In what follows, we establish conditions, formulated in terms of complex projective Lagrange critical pairs, that guarantee that $F$ has finitely many critical points on $V_\xi$. 

\begin{proof}[Proof of Theorem~\ref{constrained_real_isolated_critical_points}]
\textcolor{black}{First, we demonstrate that the non-singularity condition of
Theorem~\ref{constrained_real_isolated_critical_points} is not vacuous.
We first take $\R=\mathbb R$, and hence $\C=\mathbb C$. The set of critical points of $P$ forms an algebraic subset of $\mathbb{C}^n$, and by the Chevalley's Theorem~\cite[I.8, Corollary~2]{Mum99}, its image under the polynomial map $P:\mathbb{C}^n \to \mathbb{C}$, i.e., the set of critical values of $P$, is a constructible subset of $\mathbb{C}$. Therefore, the set of critical values is either finite or the complement of a finite subset of $\mathbb{C}$. However, by the classical Sard Theorem (see e.g.,~\cite[Theorem~6.10]{L13}), the set of critical values of $P$ has Lebesgue measure zero in $\mathbb{C}$. This implies that the set of critical values of $P$ must be finite, i.e., for all but finitely many $c \in \mathbb{C}$, $V_{c}$ is non-singular.}

Since the latter conclusion is algebraic, the Lefschetz principle~\cite[Theorem~1.26]{BPR06}
implies the same statement over any algebraically closed field of
characteristic zero. In particular, it holds over
$\C=\R[i]$ for an arbitrary real closed field $\R$. Independently, the Semi-algebraic Sard Theorem~\cite[Theorem~9.6.2]{BCR98} implies that the set of critical values of $P:\R^n \to \R$ is a semi-algebraic subset of $\R$ of dimension zero, implying that $V_{\xi}$ is non-singular.

\vspace{5px}
\noindent
Now, the complex projective Lagrange critical pairs of $F$ in $\mathbb{P}_n(\C) \times \mathbb{P}_1(\C)$ on $V_{c}$ (see Definition~\ref{projective_KKT_point}) are the zeros of
\begin{align}\label{projective_KKT_system}
\Bigg\{
\bigg\{\bigg(\displaystyle \cfrac{\partial F}{\partial X_j} - U \cfrac{\partial P}{\partial X_{j}}\bigg)^H\bigg\}_{j=1,\ldots,n}, P^H - c X_0^\alpha = 0\Bigg\},
\end{align}
where $\alpha = \deg(P)$. By the assumption,~\eqref{projective_KKT_system} has finitely many zeros in $\mathbb{P}_{n}(\C) \times \mathbb{P}_{1}(\C)$ (see Remark~\ref{infinitely_many_Lagrange_multipliers}). Further, by the upper semi-continuity of the dimension of the fibers of the projection map $\pi: \mathbb{P}_{n}(\C) \times \mathbb{P}_{1}(\C) \times \C \to \C$~\cite[I.8, Corollary~3]{Mum99}, the dimension of $\pi^{-1}(\cdot)$ \textcolor{black}{does not increase near the given $c$, where the fiber has dimension zero. All this implies that}~\eqref{projective_KKT_system} has isolated zeros in a small neighborhood of $c$.

\vspace{5px}
\noindent
\textcolor{black}{Finally, by applying the Chevalley's Theorem again, we conclude that
\begin{align*}
\Gamma_c:=\big\{c \in \C \mid \dim (\pi^{-1}(c)) > 0\big\}
\end{align*}
is constructible, and thus it is either finite or the complement of a finite set. Since the interior of $\C \setminus \Gamma_c$ is non-empty, it follows that $\Gamma_c$ must be finite. By non-singularity of $V_{\xi}$, all this implies that $F$ has finitely many critical points on $V_{\xi}$.}   
\end{proof}

\begin{proof}[Proof of Theorem~\ref{constrained_real_non-degenerate_critical_points}]
Consider the polynomial system~\eqref{projective_KKT_system}. We define the complex algebraic set

\begin{align*}
\Gamma:=\Bigg\{(x,u,c) &\in \mathbb{P}_{n}(\C) \times \mathbb{P}_{1}(\C) \times \C\mid\\
&\bigg(\bigg(\displaystyle \cfrac{\partial F}{\partial X_j} - U \cfrac{\partial P}{\partial X_{j}}\bigg)^H(x,u)= 0, \ j=1,\ldots,n,\\ 
&P^H(x) - c X_0^\alpha = 0,\\
&\rank\bigg(J\bigg(\bigg\{\bigg\{\displaystyle \bigg(\cfrac{\partial F}{\partial X_j} - U \cfrac{\partial P}{\partial X_{j}}\bigg)^H\bigg\}_{j=1,\ldots,n},P^H - c X_0^\alpha\bigg\}\bigg)(x,u)\bigg) \le n\Bigg\},
\end{align*}
\textcolor{black}{where the Jacobian is with respect to $x$ and $u$}. Since $\Gamma$ is a closed algebraic subset of $\mathbb{P}_{n}(\C) \times \mathbb{P}_{1}(\C) \times \C$ and the projection map $\pi:\mathbb{P}_{n}(\C) \times \mathbb{P}_{1}(\C) \times \C \to \C$ is regular, the image $\pi(\Gamma)$ is Zariski closed~\cite[Theorems~1.9 and~1.10]{S13}. By the assumption, there exists a $c \in \C \setminus \pi(\Gamma)$, which implies that $\pi(\Gamma)$ must be finite. Hence, the zero set of~\eqref{projective_KKT_system} is non-singular for all but finitely many $c \in \C$, which means that all critical points of $F$ on $V_{\xi}$ are non-degenerate.

\end{proof}

\noindent
Now, we show that non-degeneracy of critical points of $F$ on a non-singular $V$ implies the existence of non-degenerate critical points of $F$ on $V_{\xi}$. 
\begin{proof}[Proof of Theorem~\ref{Morse_on_smooth_hypersurface}]
The Hessian of $F$ at $\bar{x} \in V$ is given by
\begin{align*}
\hess_F(\bar{x}):= J(\Phi)(\bar{x})^T \Big[\cfrac{\partial^2 F}{\partial X_{k} \partial X_{\ell}}(\bar{x})\Big] J(\Phi)(\bar{x}) + \sum_{i=1}^{s} \frac{\partial F}{\partial X_i}(\bar{x}) \Big[\cfrac{\partial^2 \phi_i}{\partial X_{k} \partial X_{\ell}}(\bar{x})\Big]^2. 
\end{align*}
By taking the derivatives of $P_i= 0$, we get 
\begin{align*}
J(\Phi)^T \Big[\cfrac{\partial^2 P_i}{\partial X_{k} \partial X_{\ell}}\Big] J(\Phi) + \sum_{j=1}^{s} \frac{\partial P_i}{\partial X_j} \Big[\cfrac{\partial^2 \phi_j}{\partial X_{k} \partial X_{\ell}}\Big]^2 = 0, \quad i=1,\ldots,s,
\end{align*}
which, by~\eqref{Lagrange_multiplier}, yields 
\begin{align*}
\hess_F(\bar{x})= J(\Phi)(\bar{x})^T\bigg( \Big[\cfrac{\partial^2 F}{\partial X_{k} \partial X_{\ell}}(\bar{x})\Big] - \sum_{i=1}^s \bar{u}_i \Big[\cfrac{\partial^2 P_i}{\partial X_{k} \partial X_{\ell}}(\bar{x})\Big]\bigg)  J(\Phi)(\bar{x}). 
\end{align*}
\noindent
All this implies that at the KKT point $(\bar{x},\bar{u})$ the Jacobian matrix
\begin{align*}
\begin{bmatrix}
\Big[\cfrac{\partial^2 F}{\partial X_{k} \partial X_{\ell}}(\bar{x})\Big] - \sum_{i=1}^s \bar{u}_i \Big[\cfrac{\partial^2 P_i}{\partial X_{k} \partial X_{\ell}}(\bar{x})\Big] & J(\{P_1,\ldots,P_s\})(\bar{x})^T\\
        J(\{P_1,\ldots,P_s\})(\bar{x}) & 0 
\end{bmatrix}
\end{align*}
of~\eqref{KKT_system_complex} is non-singular.

\vspace{5px}
\noindent
By the Semi-algebraic Sard Theorem~\cite[Theorem~9.6.2]{BCR98}, $\{P_1 - \xi_1,\ldots,P_s-\xi_s\}$ is in general position. Thus, $V_{\xi}$ is non-singular and non-empty (see Proposition~\ref{convergence_of_varieties}). Now, by the application of the Semi-algebraic Implicit Function Theorem~\cite[Corollary~2.9.8]{BCR98}  to~\eqref{KKT_system_complex}, there exist an open set $U \subset \R^s$ and a semi-algebraic mapping $(x(\xi),u(\xi)):U \to \R^{n} \times \R^s$ such that for all $\xi \in U$, $(x(\xi),u(\xi))$ satisfies
\begin{align}\label{perturbed_KKT}
\begin{cases}
   \displaystyle \cfrac{\partial F}{\partial X_j} - \sum_{i=1}^{s} U_i \cfrac{\partial P_i}{\partial X_{j}} = 0, \ \ j=1,\ldots,n,\\
    P_i=\xi_i, \ \ i=1,\ldots,s,
\end{cases}
\end{align}
and $(x(0),u(0)) = (\bar{x},\bar{u})$. Furthermore, by the continuity of $(x(\xi),u(\xi))$, the Jacobian of~\eqref{perturbed_KKT} at $(x(\xi),u(\xi))$ remains non-singular for sufficiently small $\xi > 0$, and thus $x(\xi)$ is non-degenerate. 
\end{proof}
\hide{
\begin{remark}
We should clarify that Theorem~\ref{Morse_on_smooth_hypersurface} only assures the existence of non-degenerate critical points. However, not every critical point of $F$ on $V_\xi$ is necessarily non-degenerate. For example, consider $F=X_1X_2$ and $G = X_1^2 X_2^2$, where $F$ has a unique non-degenerate critical point $(0,0)$. It is easy to see that $\xi F+( 1-\xi) G$ has unbounded degenerate critical points for every non-zero $\xi$, although $(0,0)$ is still an isolated non-degenerate critical point.
\end{remark}}
\noindent
Now, we prove Theorem~\ref{Morse_on_transversal}.

\begin{proof}[Proof of Theorem~\ref{Morse_on_transversal}]
After relabeling the polynomials, we may assume that $I=\{1,\ldots,m\}$ for some $1\leq m\leq s$. Thus,
\begin{align*}
A=\prod_{i=m+1}^sP_i,
\qquad A(\bar x)>0.
\end{align*}
We also define
\[
\widetilde P_1:=AP_1,
\qquad
\widetilde P_i:=P_i,\quad i=2,\ldots,m.
\]
Notice that
\begin{align}\label{polynomial_product}
\prod_{i=1}^m\widetilde P_i
=
\prod_{i=1}^sP_i.
\end{align}
\noindent
By Proposition~\ref{Whitney_stratification_union_of_sets}, the stratum of $V$ containing $\bar x$ is
locally given by
\begin{align*}
P_1=\cdots=P_m=0.
\end{align*}
Since $A(\bar x)\neq0$, this stratum is also locally given by
\begin{align*}
\widetilde P_1=\cdots=\widetilde P_m=0.
\end{align*}
Moreover,
\begin{align*}
d\widetilde P_1(\bar x)
=
A(\bar x)dP_1(\bar x),
\qquad
d\widetilde P_i(\bar x)
=
dP_i(\bar x),\quad i=2,\ldots,m.
\end{align*}
Thus, $\bar x$ is a non-degenerate critical point of $F$ on
\begin{align*}
\zero(\{\widetilde P_1,\ldots,\widetilde P_m\},\R^n).
\end{align*}
\noindent
Let $\bar u_1,\ldots,\bar u_m$ be the Lagrange multipliers
corresponding to $P_1,\ldots,P_m$. The corresponding Lagrange
multipliers for $\widetilde P_1,\ldots,\widetilde P_m$ are
\begin{align*}
\bar v_1=\frac{\bar u_1}{A(\bar x)},
\qquad
\bar v_i=\bar u_i,\quad i=2,\ldots,m.
\end{align*}
Since $A(\bar x)>0$, we have
\begin{align*}
\bar v_i>0,\qquad i=1,\ldots,m.
\end{align*}
\noindent
Consider the system
\begin{align}\label{alternative_system}
\begin{cases}
\displaystyle \cfrac{\partial F}{\partial X_j} -  \sum_{i=1}^{m} U_i \cfrac{\partial \tilde{P}_i}{\partial X_{j}} = 0, \ \ j=1,\ldots,n,\\
     \tilde{P}_i= \varepsilon_i, \ \ i=1,\ldots,m.
\end{cases}
\end{align}
As in the proof of Theorem~\ref{Morse_on_smooth_hypersurface}, non-degeneracy of $\bar{x}$ implies
that the Jacobian of this system with respect to $(X,U)$ is
non-singular at $(\bar x,\bar v,0)$. Hence, by the
Semi-algebraic Implicit Function Theorem
\cite[Corollary~2.9.8]{BCR98},
there exist Nash mappings
\begin{align*}
x(\varepsilon),\ 
u(\varepsilon)=(u_1(\varepsilon),\ldots,u_m(\varepsilon))
\end{align*}
for $\varepsilon=(\varepsilon_1,\ldots,\varepsilon_m)$
sufficiently close to $0$, satisfying \eqref{alternative_system}, with
\begin{align*}
x(0)=\bar x, \ u(0)=\bar v.
\end{align*}
\noindent
We next consider the system
\begin{align}\label{epsilon_xi_system}
u_i(\varepsilon)\varepsilon_i=t,
\qquad i=1,\ldots,m.
\end{align}
At $(\varepsilon,t)=(0,0)$, the Jacobian of \eqref{epsilon_xi_system} with respect to $\varepsilon$ is $\operatorname{diag}(\bar v_1,\ldots,\bar v_m)$, and is therefore non-singular. Applying again the Semi-algebraic
Implicit Function Theorem, there exist Nash functions
$h_1,\ldots,h_m$ such that
\begin{align*}
\varepsilon_i=h_i(t),\qquad i=1,\ldots,m,
\end{align*}
for sufficiently small $t>0$, with $h_i(0)=0$. Differentiating~\eqref{epsilon_xi_system} at $t=0$ gives
\begin{align*}
h_i'(0)=\frac{1}{\bar v_i}>0.
\end{align*}
Hence, after restricting $t$ if necessary,
\begin{align*}
h_i(t)>0,\ h_i'(t)>0,
\quad i=1,\ldots,m.
\end{align*}
\noindent
Let
\begin{align*}
y(t):=
x(h_1(t),\ldots,h_m(t)), \ \ h(t):=\prod_{i=1}^m h_i(t).
\end{align*}
\noindent
By~\eqref{alternative_system} and~\eqref{epsilon_xi_system}, we have
\begin{align*}
\frac{\partial F}{\partial X_j}(y(t))
-
t\sum_{i=1}^m
\frac{1}{\widetilde P_i(y(t))}
\frac{\partial\widetilde P_i}{\partial X_j}(y(t))
=0,
\qquad j=1,\ldots,n,
\end{align*}
and
\begin{align*}
\prod_{i=1}^m\widetilde P_i(y(t))=h(t),
\end{align*}
which yields
\begin{align*}
\frac{\partial F}{\partial X_j}(y(t))
-
\frac{t}{h(t)}
\frac{\partial}{\partial X_j}
\left(\prod_{i=1}^m\widetilde P_i\right)(y(t))
=0,
\qquad j=1,\ldots,n.
\end{align*}
Using \eqref{polynomial_product}, this becomes
\begin{align*}
\frac{\partial F}{\partial X_j}(y(t))
-
\frac{t}{h(t)}
\frac{\partial}{\partial X_j}
\left(\prod_{i=1}^sP_i\right)(y(t))
=0,
\qquad j=1,\ldots,n.
\end{align*}
Thus, $y(t)$ is a critical point of $F$ on
\begin{align*}
\zero\left(
\prod_{i=1}^sP_i-h(t),\R^n
\right).
\end{align*}
\noindent
Further, $h(t)>0$ for sufficiently small $t>0$, and
\begin{align*}
h'(t)
=
\sum_{i=1}^m
h_i'(t)\prod_{\substack{\ell=1\\ \ell\neq i}}^m h_\ell(t)
>0.
\end{align*}
Hence, by the Semi-algebraic Inverse Function Theorem
\cite[Proposition~2.9.7]{BCR98}, $h$ has a Nash inverse for
sufficiently small positive values. Define
\begin{align*}
x_\xi:=y(h^{-1}(\xi)).
\end{align*}
It follows that $x_\xi\in\Crit(V_\xi,F)\cap\R\langle\xi\rangle_b^n$. Finally, since $\lim_{t\downarrow0}y(t)=\bar x$ and $\lim_{\xi\downarrow0}h^{-1}(\xi)=0$, we obtain $\lim_\xi(x_\xi)=\bar x$, which completes the proof.
\end{proof}
\noindent
Alternatively, we can exploit the conditions in Corollaries~\ref{bounded_critical_points} and~\ref{existence_critical_points} and Theorems~\ref{constrained_real_isolated_critical_points} and~\ref{constrained_real_non-degenerate_critical_points} to guarantee the existence of bounded isolated critical points.
\begin{corollary}\label{corollary_bounded_finiteness}
Let $F\in \R[X_1,\ldots,X_n]$, and suppose that $V$ satisfies the conditions of Corollary~\ref{bounded_critical_points} or Corollary~\ref{existence_critical_points}. Further, assume that the conditions of Theorems~\ref{constrained_real_isolated_critical_points} or~\ref{constrained_real_non-degenerate_critical_points} hold. Then $\Crit(V_\xi,F) \cap \R\la\xi\ra_b^n$ is non-empty and finite. In particular,
\begin{enumerate}[(i)]
    \item If the conditions of Corollary~\ref{bounded_critical_points} hold, then $\Crit(V_\xi,F) \neq \emptyset$ and $\Crit(V_\xi,F) \subset \R\la\xi\ra_b^n$.
    \item If the conditions of Theorem~\ref{constrained_real_non-degenerate_critical_points} hold, then all critical points in $\Crit(V_\xi,F)$ are non-degenerate.
\end{enumerate}
\end{corollary}
\begin{proof}
It follows immediately from \textcolor{black}{Corollaries}~\ref{bounded_critical_points} and~\ref{existence_critical_points} and Theorems~\ref{constrained_real_isolated_critical_points} and~\ref{constrained_real_non-degenerate_critical_points}.
\end{proof}

\subsubsection{Proof of Theorem~\ref{convergence_of_critical_points}}
We begin by establishing sufficient conditions for the existence of points in $V_{\xi}$ (Proposition~\ref{convergence_of_varieties}). We then use these results to characterize the limit of tangent spaces of $V_\xi$ (Propositions~
\ref{convergence_of_tangent_spaces}-\ref{convergence_of_tangent_spaces_extended}).

\vspace{5px}
\noindent
The following result leverages the Monotonicity Theorem (Proposition~\ref{Monoton_Thm}) and will be used in the proof of Proposition~\ref{convergence_of_varieties}.
\begin{proposition}\label{inverse_semi-alg_func}
Let $f:(0,a) \to \R$ be a semi-algebraic function such that $f(t) > 0$ and $\lim_{t \downarrow 0} (f(t)) = 0$. Then $f$ has a Nash inverse on $(0,a')$ for sufficiently small $a'>0$.    
\end{proposition}
\begin{proof}
By Proposition~\ref{Monoton_Thm} and the Tarski-Seidenberg Transfer Principle~\cite[Theorem~2.80]{BPR06}, $f$ is either constant or $\mathcal{C}^1$-smooth and strictly monotone near $0$. Since $f(t) > 0$ and $\lim_{t \downarrow 0} (f(t)) =0$, $f$ cannot be constant near $0$. \textcolor{black}{Moreover, since $f'$ is semi-algebraic, after possibly decreasing
$a'$, its zero set is empty on $(0,a')$; hence $f'(t)>0$ on
$(0,a')$}. Therefore, there exists $0 < a' \le a$ such that $f'(t) > 0$ on $(0,a')$. Now, the rest follows from the Semi-algebraic Inverse Function Theorem~\cite[Proposition~2.9.7]{BCR98}.   
\hide{
Note that the germs of semi-algebraic functions at the right of origin is isomorphic to the field of algebraic Puiseux series~\cite{BPR06}. Thus, $f$ can be described as a Puiseux series with positive order, because $\lim_{t \downarrow  0} (f(t)) = 0$. Furthermore, this Puiseux series corresponds to a branch of the graph of $f$ that is described by an affine parametrization $(t^m,\phi(t))$ for all $t \ge 0$, where $m$ is a minimal positive integer and $\phi(t)$ is a real analytic function. Since $\phi(t)$ is analytic and $\phi(t) > 0$ on $(0,\varepsilon^{1/m})$, then $\phi'(t) > 0$ on $(0,\varepsilon')$ for sufficiently small $\varepsilon'>0$%
\footnote{If we represent $\phi(t) = \sum_{i=1}^{\infty} c_i t^{\alpha+i}$ with $c_1 \neq 0$, then we must have $\alpha \ge 0$. Then $\phi(t) > 0$ for sufficiently small positive $t$ implies that $c_1 > 0$. Therefore, $\phi'(t) = \sum_{i=1}^{\infty} (\alpha+i) c_i t^{\alpha-1+i} > 0$ for sufficiently small positive $t$. }. Hence, by the Semi-algebraic Inverse Function theorem~\cite[Proposition~3.24]{BPR06}, $\phi(t)$ has a Nash inverse $\phi^{-1}$ on $(0,\varepsilon')$. Then the inverse of $f$ on $(0,\varepsilon')$ is given by $f^{-1}(x)=(\phi^{-1}(x))^m$ which is also Nash on $(0,\varepsilon')$. 
}
\end{proof}

\begin{proposition}\label{convergence_of_varieties}
Let $V$ and $V_{\xi} \subset \R\la \xi \ra^n$ be defined as in~\eqref{variety} and~\eqref{perturbed_variety}, where $V$ is non-empty. Suppose that $\Sing(\mathcal{P}) \subset \overline{V \setminus \Sing(\mathcal{P})}$. Then for every $\bar{x} \in V$ there exists $x_{\xi} \in V_{\xi}$ such that $\lim_{\xi} (x_{\xi}) = \bar{x}$.
In particular, $V_\xi$ is non-empty.   
\end{proposition}
\begin{remark}
Notice that the condition $\Sing(\mathcal{P}) \subset \overline{V \setminus \Sing(\mathcal{P})}$ enforces the existence of at least one non-singular point for $V$, which also implies that $s \le n$. 
\end{remark}
\begin{proof}
 We condition on singular and non-singular points of $V$.
\begin{enumerate}[(i)]
\item\label{smooth} Let $\bar{x} \in V$ be a non-singular point of $V$. If we assume, without loss of generality, that the leading $s$-principal submatrix of $J(\{P_1,\ldots,P_s\})$ is non-singular, then by the Semi-algebraic Implicit Function Theorem~\cite[Corollary~2.9.8]{BCR98}, there exist semi-algebraic functions $\zeta_i:U \to \R$, where $U \subset \R^s$ is a small neighborhood of $0$, such that 
\begin{align*}
(\zeta_1(\xi_1,\ldots,\xi_s),\ldots,\zeta_s(\xi_1,\ldots,\xi_s),\bar{x}_{s+1},\ldots,\bar{x}_n)
\end{align*}
for all $\xi \in U$ is a solution of the system
\begin{align*}
   P_1 - \xi_1 = 0, \ \ldots, \ P_s - \xi_s = 0, \ \ X_{s+1}-\bar{x}_{s+1} = 0,\ldots, X_n-\bar{x}_n = 0, 
\end{align*}
and $\zeta_i(0)=\bar{x}_i$ for $i=1,\ldots,s$. Thus, for every non-singular $\bar{x} \in V$ there exists $x_{\xi} \in V_{\xi}$ such that $\lim_{\xi} (x_\xi) = \bar{x}$. 

\hide{
\item\label{singular} Let $\bar{x} \in V$ be singular. By the assumption, every neighborhood $U$ of $\bar{x}$ contains a non-singular point $y \in V$, and by Part~\ref{smooth}, there exists $y' \in V_{\xi} \cap U$. All this implies that $(\bar{x},0) \in \overline{T}$, where 
\begin{align*}
    T:=\{(x,\xi) \in \R^n \times \R^s  \mid P_i(x) - \xi_i = 0, \ \xi_i > 0, \ i=1,\ldots,s\}.
\end{align*}
By the Curve Selection Lemma~\cite[Theorem~3.19]{BPR06}, there exists a semi-algebraic mapping $\eta(t):=(x(t),\xi_1(t),\ldots,\xi_s(t))$ such that $\eta((0,1]) \subset T$ and $\eta(0) = (\bar{x},0)$. By Proposition~\ref{inverse_semi-alg_func}, each $\xi_i(t)$ has a Nash inverse $\xi_i^{-1}$ on $(0,t_i')$ for sufficiently small $t_i' > 0$. Thus, $y(\xi_i) = x(\xi_i^{-1})$ is a semi-algebraic function such that $y(0) = \bar{x}$.
}

\item\label{singular} Let $\bar{x} \in V$ be singular. 
It suffices to prove that $\lim_\xi \inf_{y' \in V_\xi} ||\bar{x} - y'|| = 0$.
Suppose not.
Then there exists,
$r_0 > 0, r_0 \in \R$, such that $r_0 \le \inf_{y' \in V_\xi} ||\bar{x} - y'||$. 

\vspace{5px}
\noindent
By the assumption, every neighborhood $U$ of $\bar{x}$ contains a non-singular point of $V$.
In particular this implies that for every $r >0, r \in \R$, there exists $x' \in V \cap B(\bar{x},r)$
such that $x'$ is a non-singular point of $V$, and moreover using 
Part~\ref{smooth}, there exists $y' \in V_\xi$, such that $\lim_\xi y' = x'$. Let $r = r_0/2$ and $x',y'$ as above. Then, $||\bar{x} - x'|| < r_0/2, ||x' - y'|| <r_0/2$, and hence by triangle inequality $||\bar{x} - y'|| < r_0$. But this would imply that $\inf_{y' \in V_\xi} ||\bar{x} - y'|| < r_0 $ which is a contradiction.
\end{enumerate}
\end{proof}
\noindent

\begin{example}
 We should note that without the non-singularity condition of Proposition~\ref{convergence_of_varieties}, $V_{\xi}$ might be empty. For instance, if $P = -X^3 - X^2 \in \mathbb{R}[X]$, then $V=\zero(P,\mathbb{R})=\{-1,0\}$, where $0$ is a singular point of $V$. In this case, $V$ does not satisfy the condition of Proposition~\ref{convergence_of_varieties}, and it is easy to see that $\lim_{\xi} (y) \neq 0$ for any bounded $y \in V_{\xi}$.
\end{example}
\noindent
Now, we characterize the limit of tangent spaces of $V_{\xi} \subset \R\la \xi \ra^n$.
\hide{
\begin{proposition}\label{convergence_of_tangent_spaces}
Let $V$ and $V_{\xi}$ be defined as in~\eqref{variety} and~\eqref{perturbed_variety}, where $V$ is non-empty. Furthermore, consider $\{\xi_k\} \downarrow 0$ and $x(\xi_k) \in V_{\xi_k}$ such that $x(\xi_k) \to \bar{x} \in V$, and $\bar{x}$ is a non-singular point of $V$. Then we have
\begin{align}\label{tangent_space_limit}
T_{x(\xi_{k})} V_{\xi_k} \to T_{\bar{x}} Z,
\end{align}
where $Z$ is the stratum containing $\bar{x}$.
\end{proposition}

\begin{proof}
By Proposition~\ref{convergence_of_varieties}, the Semi-algebraic Sard Theorem~\cite[Theorem~9.6.2]{BCR98}, and the transversality at $\bar{x}$, $V_{\xi}$ is non-empty and non-singular for sufficiently small $\xi >0$, and
\begin{align}\label{tangent_transversal}
T_{x({\xi})} V_{\xi} = \bigcap_{i=1}^s T_{x({\xi})} \{P_i - \xi_i = 0\}. 
\end{align}

\vspace{5px}
\noindent
The proof of the second part is based on a result on convergence of sets in the sense of Painl\'eve-Kuratowski. Notice that  
\begin{align*}
T_{x({\xi})} V_{\xi_k}&= \{h \in \mathbb{R}^n \mid J(\{P_1,\ldots,P_s\})(x(\xi_k))h = 0\},\\
T_{\bar{x}} Z&=\{h \in \mathbb{R}^n \mid J(\{P_1,\ldots,P_s\})(\bar{x})h = 0\}.
\end{align*}
Since $J(\{P_1,\ldots,P_s\})\bar{x}$ has full row rank, then the result follows from~\cite[Theorem~4.32(b)]{Rock09}.
\end{proof}
}
\begin{definition}
Let $V$ be defined as in ~\eqref{variety}, and assume that $V$ is non-singular. We define the \textit{tangent space} of $V$ at $x$ as
\begin{align*}
T_{x} V= \{h \in \R^n \mid J(\{P_1,\ldots,P_s\})(x)h = 0, \quad \|h \| \le 1\}.
\end{align*}
\end{definition}
\begin{proposition}\label{convergence_of_tangent_spaces}
Let $V$ and $V_{\xi} \subset \R\la \xi \ra^n$ be defined as in~\eqref{variety} and~\eqref{perturbed_variety}, where $V$ is non-empty and non-singular. Furthermore, let $x_\xi \in V_{\xi}$ be bounded and $\lim_\xi (x_\xi) = \bar{x}$. Then we have
\begin{align*}
\lim_{\xi} (T_{x_\xi} V_{\xi}) = T_{\bar{x}} V.
\end{align*}
\end{proposition}

\begin{proof}
By the Semi-algebraic Sard Theorem~\cite[Theorem~9.6.2]{BCR98}, $V_{\xi}$ is non-empty and non-singular, and we have
\begin{align*}
T_{x_\xi} V_{\xi}= \{h \in \R\la \xi \ra^n \mid J(\{P_1,\ldots,P_s\})(x_\xi)h = 0, \quad \|h\|\le 1\}.
\end{align*}

\vspace{5px}
\noindent
By the definition of $\lim_{\xi}$, $\lim_\xi (T_{x_\xi} V_{\xi}) \subset T_{x} V$. Further, since $J(\{P_1,\ldots,P_s\})(\bar{x})$ has full row rank, it follows from the Implicit Function Theorem (analogous to the proof of Proposition~\ref{convergence_of_varieties}) that $T_{x} V \subset \lim_\xi( T_{x_\xi} V_{\xi})$, which completes the proof.  
\end{proof}

\noindent
Proposition~\ref{convergence_of_tangent_spaces} characterizes the limit of tangent spaces for a non-singular algebraic set. We can extend this result for the union of non-singular hypersurfaces, as follows. 
\begin{proposition}\label{convergence_of_tangent_spaces_extended}
\textcolor{black}{Let $V = \zero\big( \prod_{i=1}^s P_{i},\R^n\big)$, 
where $\mathcal{P} = \{P_1,\ldots,P_s\} \subset \R[X_1,\ldots,X_n]$ is in general position. Let $x_\xi \in V_{\xi} = \zero\big( \prod_{i=1}^s P_{i}-\xi,\R\la \xi \ra^n\big)$ be a bounded solution with $\lim_{\xi} (x_\xi) = \bar{x}$. Further, define 
\begin{align*}
I=\{i \in \{1,\ldots,s\} \mid P_i(\bar{x}) = 0\}.
\end{align*}
Then we have
\begin{align}\label{tangent_space_inclusion}
T_{\bar{x}}Z\subset \lim_{\xi}( T_{x_\xi} V_{\xi}),
\end{align}
where $Z$ is the stratum of $V$ containing $\bar{x}$, with respect to the canonical Whitney stratification of $V$.}
\end{proposition}
\begin{proof}
By the Semi-algebraic Sard Theorem~\cite[Theorem~9.6.2]{BCR98}, $V_{\xi}$ is non-empty and non-singular, and the tangent space of $V_{\xi}$ at $x_\xi$ is given by
\begin{align*}
T_{x_\xi} V_\xi = \bigg\{h \in \R\la \xi \ra^n \mid  \sum_{i=1}^s \prod_{\ell=1,\ell \neq i}^s P_{\ell}(x_\xi)\bigg(\sum_{j=1}^n    \frac{\partial P_{i}}{\partial x_j}(x_\xi) h_j\bigg) = 0, \quad \|h\|\le 1\bigg\}.
\end{align*}
If $J(\{\prod_{i=1}^s P_{i}\})(\bar{x}) \neq 0$, then the result follows from the \textcolor{black}{local argument} of Proposition~\ref{convergence_of_tangent_spaces}. We note that for every $i=1,\ldots,s$,
\begin{align*}
\lim_{\xi} \bigg(\prod_{\ell=1, \ell \neq i}^s P_{\ell}(x_\xi)\bigg) = 0.
\end{align*}
Let $a_i(\xi):=\prod_{\ell=1, \ell \neq i}^s P_{\ell}(x_\xi)$ and $\lambda_i(\xi):=a_i(\xi)/\sqrt{\sum_{k=1}^s a_k^2(\xi)}$. We also note that the vector 
\begin{align*}
\lambda(\xi) = (\lambda_1(\xi),\ldots,\lambda_s(\xi))
\end{align*}
is bounded and $\|\lambda(\xi)\|=1$. We define $\bar{\lambda}:=\lim_\xi \lambda(\xi)$, where $\|\bar{\lambda}\|  =1$. Notice that for every $j \not \in I, i \in I$, $\lim_\xi a_j(\xi)/a_i(\xi) =0$ which implies $\bar{\lambda}_j = 0$ for every $j \not \in I$.   Thus, we get
\begin{align*}
\lim_{\xi} (T_{x_\xi} V_\xi) &= \lim_{\xi} \bigg(\bigg\{h \in \R \la \xi \ra^n \mid\sum_{i=1}^s \lambda_i(\xi)  \bigg(\sum_{j=1}^n \frac{\partial P_i}{\partial x_j}(x_\xi) h_j\bigg)  = 0, \quad \|h\|\le 1\bigg\}\bigg),\\
&=\bigg\{h \in \R^n \mid\sum_{i \in I} \bar{\lambda}_i \bigg(\sum_{j=1}^n \frac{\partial P_i}{\partial x_j}(\bar{x})h_j\bigg)  = 0, \quad \|h\|\le 1\bigg\},
\end{align*}
where the second equality follows from $J(\{P_i\}_{i \in I})(\bar{x})$ being full row rank and the Semi-algebraic Implicit Function Theorem~\cite[Corollary~2.9.8]{BCR98}. Further, since $\mathcal{P}$ is in general position, we have
\textcolor{black}{
\begin{align*}
T_{\bar{x}}Z = \bigcap_{i \in I} T_{\bar{x}} \zero(P_{i},\R^n) \subset \lim_{\xi}( T_{x_\xi} V_{\xi}),
\end{align*}} 
which completes the proof.
\end{proof}
\noindent
Proposition~\ref{convergence_of_tangent_spaces_extended} implies that given $h \in T_{\bar{x}} Z$, there exist $x_\xi \in V_\xi$ and a vector $h_\xi \in T_{x_{\xi}} V_{\xi}$ infinitesimally close to $h$. This fact will be used in the proof of Theorem~\ref{convergence_of_critical_points}.

\vspace{5px}
\noindent
Now, we prove Theorem~\ref{convergence_of_critical_points}. All we need here is the inclusion~\eqref{tangent_space_inclusion}.
\begin{proof}[Proof of Theorem~\ref{convergence_of_critical_points}]
Let $Z$ be the stratum of $V$ containing $\bar{x}$. We prove by contradiction. If $\bar{x}$ is not a critical point of $F$ on $Z$, then we have $dF(\bar{x}) \mid T_{\bar{x}} Z \neq 0$. By the continuity of $dF$ and Proposition~\ref{convergence_of_tangent_spaces_extended}, all this means that $dF(x_\xi) \mid T_{x_\xi} V_{\xi} \neq 0$, which would imply that $x_\xi$ is not a critical point of $F$ on $V_{\xi}$.  
\end{proof}

\subsection{Proofs for a semi-algebraic critical path}\label{proof_of_critical_path_semi-algebraic}
We prove existence, convergence and $\mathcal{C}^{\infty}$-smoothness of critical paths for~\eqref{poly_optim}. First, we show that a critical path is semi-algebraic for sufficiently small $\mu > 0$. Further utilization of semi-algebraicity demonstrates that a critical path is a Nash mapping. To this end, we employ the existence of a \textit{Nash stratification} of a semi-algebraic set, which we introduce below.
\begin{definition}[Nash stratification~\cite{BCR98}]
Let $E \subset \R^n$ be a semi-algebraic set. A Nash stratification of $E$ is a finite partition $(E_{\alpha})_{\alpha \in A}$ of $E$, where each $E_{\alpha}$ (which we denote by a stratum) is a Nash sub-manifold of $\R^n$ (see~\cite[Definition~2.9.9]{BCR98} for details)). Moreover, the stratification satisfies the frontier condition: if $E_{\alpha} \cap \overline{E_{\beta}} \neq \emptyset$ for $\alpha \neq \beta$, then $E_{\alpha} \subset \overline{E_{\beta}}$ and $\dim(E_\alpha) < \dim(E_\beta)$.   
\end{definition}
\begin{proposition}[Proposition~9.1.8 in~\cite{BCR98}]\label{Nash_stratification_existence}
Let $E$ be a semi-algebraic subset of $\R^n$ and $(F_{\lambda})_{\lambda \in \Lambda}$ be a finite family of semi-algebraic subsets of $E$. Then there exists a Nash stratification $(E_{\alpha})_{\alpha \in A}$ for $E$ such that each $F_{\lambda}$ is the union of some of the strata $E_{\alpha}$.   
\end{proposition}
\noindent
Now, the following result on the smoothness of a critical path is in order.
\begin{proposition}\label{critical_path_semi-algebraic_proof}
 A critical path is a Nash mapping when $\mu>0$ is sufficiently small.
\end{proposition}
\begin{proof}
We use the semi-algebraic version of Definition~\ref{def:cad}. Note that by~\eqref{critical_path_algebraic} the graph of a critical path is contained in 
\begin{align*}
D:=\bigg\{(\mu,x) \in \R^{n+1} \mid \frac{\partial f}{\partial x_j}(x) \prod_{i=1}^r g_i(x) - \mu \sum_{k=1}^r \frac{\partial g_k(x)}{\partial x_j} \prod_{i \neq k} g_i(x)  = 0, \ \ j=1,\ldots,n\bigg\},     
\end{align*}  
where $D$ is a semi-algebraic subset of $\R^{n+1}$. Thus, there exists a cell decomposition $(\mathcal{D}_1,\ldots,\mathcal{D}_{n+1})$ of $\R^{n+1}$ adapted to $D$ such that for each $C \in \mathcal{D}_{n+1}$, $C \cap D$ is either equal to 
$C$ or empty. Since $x(\mu)$ is isolated for sufficiently small positive $\mu$, 
there exists a cell $C \in \mathcal{D}_{n+1}$ 
with $\dim C = 1$
such that $C$ contains $(\mu,x(\mu))$ when $\mu >0$ is sufficiently small. 
Since such a cell is a graph of a continuous semi-algebraic function of $\mu$ (by the definition of cell decomposition), this implies 
that $x(\mu)$ is continuous and semi-algebraic when $\mu>0$ is sufficiently small. 

\vspace{5px}
\noindent
To prove the smoothness, suppose that a critical path
$x(\mu):(0,a)\to\R^n$ is well-defined, where $a>0$.
By Proposition~\ref{Nash_stratification_existence}, after possibly decreasing $a$, the graph of
$x(\mu)$ is contained in a one-dimensional Nash stratum. Locally,
this stratum admits a Nash parametrization
\begin{align*}
    t\longmapsto \bigl(\phi_0(t),\phi_1(t),\ldots,\phi_n(t)\bigr).
\end{align*}
Since the stratum is contained in the graph of $x(\mu)$, the function $\phi_0$ is non-constant. By Proposition~\ref{Monoton_Thm}, after possibly restricting the parameter interval further, we may assume that $\phi_0'$ does not vanish. Hence, by the Semi-algebraic Inverse Function Theorem~\cite[Proposition~2.9.7]{BCR98},
$\phi_0$ has a Nash inverse. Therefore,
\begin{align*}
x_i(\mu)=
    \phi_i\bigl(\phi_0^{-1}(\mu)\bigr),
    \qquad i=1,\ldots,n,
\end{align*}
for all sufficiently small $\mu>0$. Thus $x(\mu)$ is a Nash mapping for all sufficiently small $\mu>0$.  
\end{proof}

\subsubsection{Existence of critical paths}
As indicated by Example~\ref{non_existence}, existence of a central path or even a critical path is not necessarily guaranteed. In fact, there may be no critical point or local minimizer for the log-barrier function. Even if they exist, they may be non-isolated (see Example~\ref{Morse_non_compact}). Proposition~\ref{existence_of_central_path} guarantees the existence of local minimizers of the log-barrier function in the presence of an isolated compact local optimal set of~\eqref{poly_optim}. However, as Remark~\ref{no_critical_path} indicates, a compact critical set of~\eqref{poly_optim} does not imply the existence of a critical path.

\vspace{5px}
\noindent
Consider again the critical points of the log-barrier function which satisfy 
\begin{align}\label{Compact_central_path_conditions}
    \frac{\partial f}{\partial x_j}(x) - \frac{\mu}{\prod_{k=1}^r g_k(x)}  \sum_{i=1}^r  \bigg(\prod_{k=1, k \neq i}^r g_k(x)\bigg) \frac{\partial g_i(x)}{\partial x_j}  &= 0, & j&=1,\ldots,n,\\
    g_i(x) &> 0, & i&=1,\ldots,r. \nonumber
\end{align}
Note that when $\prod_{i=1}^r g_i(x(\mu))$ is sufficiently small, $x(\mu)$ is a critical point of $f$ on the non-singular algebraic set (by the Semi-algebraic Sard Theorem~\cite[Theorem~9.6.2]{BCR98}) 
\begin{align}\label{prod_hyper_surface}
S_\mu:=\bigg\{x \in \R^n \mid \prod_{i=1}^r g_i(x) = \xi(\mu) \bigg\},
\end{align}
where $\xi(\mu) = \prod_{i=1}^r g_i(x(\mu))$ and $\mu/\xi(\mu)$ is the Lagrange multiplier. 
\begin{remark}\label{unbounded_critical_path}
By the Semi-algebraic Sard Theorem~\cite[Theorem~9.6.2]{BCR98}, the algebraic set defined in~\eqref{prod_hyper_surface} is non-singular when $\mu$ is sufficiently small. This condition is always satisfied for bounded critical paths, since, by~\eqref{critical_path_algebraic} and Assumption~\ref{non_trivial_problem}, the limit of a bounded critical path belongs to $S_=$. However, this does not necessarily imply that $\prod_{i=1}^r g_i(x(\mu)) \to 0$ for every critical path $x(\mu)$. For example, the minimization problem
\begin{align*}
    \inf_x\{x_1 \mid 1 - x_1 x_2 \ge 0\}
\end{align*}
has an unbounded critical path $x(\mu)=(0,-1/\mu) \in S_{>}$, where $\prod_{i=1}^r g_i(x(\mu)) = 1$ for every $\mu > 0$. We observe that $V=\zero\big(1 - X_1 X_2 - \xi(\mu),\mathbb{R}^2\big)$ is singular for any $\mu >0$, although $(0,-1/\mu)$ is a critical point of $X_1$ on $V$ with respect to its canonical Whitney stratification.
\end{remark}

\begin{proof}[Proof of Theorem~\ref{critcal_path_existence}]
Since $x_{\xi} \in S_>$, then $x_{\xi}$ corresponds to a critical point of the log-barrier function if the Lagrange multiplier $u_{\xi}$ is equal to $\mu/\xi(\mu)$, according to~\eqref{Compact_central_path_conditions}, i.e., $\xi u_{\xi}  = \mu.$ By the assumption, this equation has a positive \textcolor{black}{infinitesimal} solution in $\R\la \mu \ra$, which indicates that $V_{\mu} \cap S_>$ is non-empty.
\end{proof}
\begin{proof}[Proof of Theorem~\ref{isolated_critical_points}]
Let $G:=\prod_{i=1}^r g_i$, $\xi_\mu:=G(x_\mu)$, and
$\bar x:=\lim_\mu x_\mu$. By~\eqref{Compact_central_path_conditions},
$x_\mu$ corresponds to a critical point of $f$ on $S_{\xi_\mu}$ with
Lagrange multiplier $u_\mu=\frac{\mu}{\xi_\mu}$, so that
$\xi_\mu u_\mu=\mu$. Since $x_\mu$ is bounded,
$\lim_\mu\xi_\mu=0$. Indeed, otherwise,
\begin{align*}
df(\bar x)
=
\lim_\mu df(x_\mu)
=
\lim_\mu\frac{\mu}{\xi_\mu}dG(x_\mu)
=
0,
\end{align*}
contradicting Assumption~\ref{non_trivial_problem}.

\vspace{5px}
\noindent
By Remark~\ref{inf_approach}, let $x(\mu)$ denote the semi-algebraic germ represented by $x_\mu$, and define
\begin{align*}
\xi(\mu):=G(x(\mu)).
\end{align*}
Then
\begin{align*}
\xi(\mu)>0,
\qquad
\xi(\mu)\longrightarrow0
\quad\text{as }\mu\downarrow0.
\end{align*}
\noindent
By the proof of Theorem~\ref{constrained_real_isolated_critical_points}, there exists a finite set
$E\subset\C$ such that the set of complex projective Lagrange critical pairs of $f$ on $S_\xi$ is finite for every $\xi\notin E$. Consider
\begin{align*}
L:=
\left\{
(x,u,\xi)\in\C^n\times\C\times(\C\setminus E)
\ \middle|\
G(x)=\xi,\quad df(x)=u\,dG(x)
\right\}.
\end{align*}
The fibers of the projection
\begin{align*}
L\longrightarrow\C\setminus E,
\qquad
(x,u,\xi)\longmapsto\xi,
\end{align*}
are finite, and hence $\dim(L)\leq1$. Now consider the regular function
\begin{align*}
H:L\longrightarrow\C,
\qquad
H(x,u,\xi):=\xi u,
\end{align*}
and define
\begin{align*}
E':=\{c\in\C\mid \dim(H^{-1}(c))>0\}.
\end{align*}
The set $E'$ is finite. Indeed, if $\dim(H^{-1}(c))>0$, then
$H^{-1}(c)$ contains a one-dimensional irreducible component of $L$,
on which $H$ is constant. Since $L$ has only finitely many irreducible
components, this can occur for only finitely many $c\in\C$.

\vspace{5px}
\noindent
Since $E$ and $E'$ are finite, $\xi(\mu)>0$, and
$\xi(\mu)\to0$, for all sufficiently small $\mu>0$ we have
\begin{align*}
\xi(\mu)\notin E
\qquad\text{and}\qquad
\mu\notin E'.
\end{align*}
In particular, $H^{-1}(\mu)$ is finite. Fix such a $\mu>0$. Since
\begin{align*}
G(x(\mu))=\xi(\mu)\notin E,
\end{align*}
and $E$ is finite, there exists a neighborhood $U_\mu$ of $x(\mu)$
such that $G(U_\mu)\cap E=\emptyset$. 

Now, let
\begin{align*}
y\in V_\mu\cap S_>\cap U_\mu.
\end{align*}
By~\eqref{critical_path_algebraic},
\begin{align*}
df(y)=\frac{\mu}{G(y)}\,dG(y).
\end{align*}
Hence
\begin{align*}
\left(
y,\frac{\mu}{G(y)},G(y)
\right)\in H^{-1}(\mu).
\end{align*}
Therefore, the map
\begin{align*}
y\longmapsto
\left(
y,\frac{\mu}{G(y)},G(y)
\right)
\end{align*}
maps $V_\mu\cap S_>\cap U_\mu$ injectively into the finite set
$H^{-1}(\mu)$. Thus $V_\mu\cap S_>\cap U_\mu$ is finite, and in
particular $x(\mu)$ is an isolated point of $V_\mu\cap S_>$. This holds for every sufficiently small $\mu>0$. Hence the germ
$x(\mu)$ defines a critical path, and equivalently $x_\mu$ is a
critical path.
\end{proof}

\begin{proof}[Proof of Theorem~\ref{sufficient_conditions_existence}]
\textcolor{black}{Let
\begin{align*}
    I:=\{i\in\{1,\ldots,r\}\mid g_i(\bar x)=0\},
    \qquad
    A:=\prod_{i\notin I}g_i .
\end{align*}
After relabeling, we may assume that $I=\{1,\ldots,m\}$. Define
\begin{align*}
\widetilde g_1:=A g_1,
\qquad
\widetilde g_i:=g_i,\quad i=2,\ldots,m,
\end{align*}
and write
\begin{align*}
\widetilde g:=(\widetilde g_1,\ldots,\widetilde g_m).
\end{align*}
Notice that
\begin{align*}
\prod_{i=1}^m\widetilde g_i=\prod_{i=1}^r g_i,
\end{align*}
and, on $S_>$,
\begin{align*}
\sum_{i=1}^m\frac{d\widetilde g_i}{\widetilde g_i}
=
\sum_{i=1}^r\frac{dg_i}{g_i}.
\end{align*}
\noindent
Since $\bar x\in S$, we have $g_i(\bar x)>0$ for every $i\notin I$, and hence $A(\bar x)>0$. Therefore, the assumptions of Theorem~\ref{Morse_on_transversal} are satisfied for $F=f$ and $V=S_=$.}

\vspace{5px}
\noindent
We now use the construction in the proof of Theorem~\ref{Morse_on_transversal}. There exist Nash functions $y(t)$ and $h(t)$, for all sufficiently small $t>0$, such that
\begin{align*}
    y(t)\longrightarrow \bar x,
    \qquad
    \prod_{i=1}^r g_i(y(t))=h(t)>0,
\end{align*}
and
\begin{align*}
    df(y(t))
    -
    \frac{t}{h(t)}
    d\left(\prod_{i=1}^r g_i\right)(y(t))
    =0.
\end{align*}
Moreover, after possibly restricting $t$, we have $g_i(y(t))>0$ for every $i=1,\ldots,r$. Hence,
\begin{align*}
    df(y(t))
    -
    t\sum_{i=1}^r
    \frac{dg_i(y(t))}{g_i(y(t))}
    =0.
\end{align*}
\hide{
Thus $y(t)\in V_t\cap S_>$ for every sufficiently small $t>0$.
By Theorem~\ref{isolated_critical_points}, $y(t)$ is isolated for every sufficiently small
$t>0$, and therefore $y(t)$ is a critical path.
}
We now show that $y(t)$ is an isolated point of $V_t \cap S_{>}$ for all
sufficiently small $t>0$.
Recall from the proof of
Theorem~~\ref{Morse_on_transversal} that $(y(t),u(t),\varepsilon(t))$ was obtained as the
unique solution near $(\bar x,\bar v,0)$ of the system
\eqref{alternative_system}--\eqref{epsilon_xi_system} in the unknowns $(x,u,\varepsilon)$
depending on the parameter $t$: the Jacobian of this system with respect to
$(x,u,\varepsilon)$ at $(\bar x,\bar v,0)$ is block lower triangular, with
diagonal blocks the bordered Hessian of the Lagrangian at $(\bar x,\bar v)$
(non-singular since $\bar x$ is a non-degenerate critical point of $f$ on the
stratum through $\bar x$) and the diagonal matrix
$\mathrm{diag}(\bar v)$ (non-singular since $\bar v>0$ by the
assumption that the corresponding Lagrange multipliers are all positive). 

Hence, by the semi-algebraic implicit function theorem, 
there exist a neighborhood $U$ of $(\bar x,\bar v,0)$ and
$t_0>0$ such that for each $t\in(0,t_0)$ the system
\eqref{alternative_system}--\eqref{epsilon_xi_system} has exactly one solution
$(x,u,\varepsilon)\in U$, namely $(y(t),u(t),\varepsilon(t))$.
Shrinking $t_0$ if necessary, we may assume that
$(y(t),u(t),\varepsilon(t))\in U$ for all $t\in(0,t_0)$, where
\begin{align*}
u(t)=\frac{t}{\widetilde g(y(t))}
:=
\left(
\frac{t}{\widetilde g_1(y(t))},\ldots,
\frac{t}{\widetilde g_m(y(t))}
\right),
\end{align*}
and we let $W \subset \mathbb{R}^n$ be a
neighborhood of $\bar x$ such that
$\bigl(x,\, t/\tilde g(x),\, \tilde g(x)\bigr)\in U$ whenever
$x\in W\cap S_{>}$ and $t/\tilde g(x)$ is close to $\bar v$; such $W$ exists
because $\tilde g$ is continuous and $\tilde g(\bar x)=0$.

Now fix $t\in(0,t_0)$ and let $x'\in V_t\cap S_{>}$ be sufficiently close to
$y(t)$ (in particular $x'\in W$). By the definition of $V_t$, the point $x'$
satisfies $\tilde g(x')\, u' = t$, where 
\begin{align*}
u':=\frac{t}{\widetilde g(x')}
=
\left(
\frac{t}{\widetilde g_1(x')},\ldots,
\frac{t}{\widetilde g_m(x')}
\right),
\end{align*}
and the
Lagrange conditions defining $V_t$ are precisely \eqref{alternative_system} evaluated at
$(x',u')$; setting $\varepsilon' := \tilde g(x')$ we obtain that
$(x',u',\varepsilon')$ solves \eqref{alternative_system}--\eqref{epsilon_xi_system} with parameter
$t$. Since $u(t)=t/\tilde g(y(t))$ and $u'-u(t)\to 0$ as $x'\to y(t)$
(continuity of $\tilde g$ on $S_{>}$), we may take $x'$ close enough to $y(t)$
that $(x',u',\varepsilon')\in U$. Uniqueness of the solution in $U$ then forces
$(x',u',\varepsilon') = (y(t),u(t),\varepsilon(t))$, and in particular
$x'=y(t)$. Thus $y(t)$ is an isolated point of $V_t\cap S_{>}$ for every
$t\in(0,t_0)$, as required.

\vspace{5px}
\noindent
Finally,
the above construction shows that
the critical
path is $\mathcal C^\infty$-smooth at $t=0$. If $\R=\mathbb R$, its analyticity follows from \cite[Proposition~8.1.8]{BCR98}.
\end{proof}

\noindent
Analogous to Corollary~\ref{corollary_bounded_finiteness}, the results of Corollaries~\ref{bounded_critical_points}-\ref{existence_critical_points} and Theorems~\ref{critcal_path_existence}-\ref{isolated_critical_points} can be leveraged to guarantee the existence of a bounded critical path.
\begin{corollary}
\textcolor{black}{Suppose that $f$ and $S_=$ satisfy the conditions of Corollary~\ref{bounded_critical_points} or Corollary~\ref{existence_critical_points}.  Let $x_\xi\in\Crit(S_\xi,f)\cap \R\langle\xi\rangle_b^n$ be a bounded critical point with associated Lagrange multiplier $u_\xi$. Assume that $x_\xi\in S_>$ and that $\xi u_\xi-\mu=0$ has a positive infinitesimal solution $\xi=\xi(\mu)\in\R\langle\mu\rangle$.
Further, suppose that, for some $c\in\C$, the set of complex projective
Lagrange critical pairs of $f$ in
$\mathbb P^n(\C)\times\mathbb P^1(\C)$ on $S_c=\zero\left(\prod_{i=1}^r g_i-c,\C^n\right)$
is finite. Then there exists a bounded critical path.} 
\end{corollary}

\begin{proof}
Apply Corollaries~\ref{bounded_critical_points} and~\ref{existence_critical_points}, and Theorems~\ref{critcal_path_existence} and~\ref{isolated_critical_points} to $f$ and $S_=$.
\end{proof}
\noindent
\textcolor{black}{If $\mathcal{Q}=\{g_1,\ldots,g_r\} \subset \R[X_1,\ldots,X_n]$ in the definition of~\PO~is in general position and $f$ has a bounded set of local minimizers on $S$, then $V_{\mu} \cap S_>$ is non-empty by Proposition~\ref{existence_of_central_path}.} This leads to the following stronger result.

\hide{
\begin{corollary}\label{compactness_finiteness}
\textcolor{red}{Suppose that $\mathcal{Q}=\{g_1,\ldots,g_r\} \subset \R[X_1,\ldots,X_n]$ in the definition of~\PO~is in general position, and 
the set of local minimizers of $f$ on $S$ is non-empty and bounded. 
Further, suppose that
for some $c \in \C$, the set of all complex projective Lagrange critical pairs of $f$ in $\mathbb{P}_n(\C) \times \mathbb{P}_1(\C)$ on $S_{c} = \zero(\prod_{i=1}^r g_i - c,\C^n)$ is finite. Then a bounded central path exists.} 
\end{corollary}
}
\begin{corollary}\label{compactness_finiteness}
Suppose that Assumption~\ref{non_trivial_problem} holds, that
$\mathcal{Q}=\{g_1,\dots,g_r\}\subset \mathrm{R}[X_1,\dots,X_n]$ in the
definition of \PO\ is in general position, and that there is a local minimum
value $v^\ast$ of $f$ on $S$ such that the set
\[
  A^\ast := \{x\in S \mid f(x)=v^\ast \text{ and } x \text{ is a local minimizer of } f \text{ on } S\}
\]
is non-empty, closed and bounded, and is a union of semi-algebraically connected
components of $A:=\{x\in S \mid f(x)=v^\ast\}$. Further, suppose that for
some $c\in\C$ the set of complex projective Lagrange critical pairs of
$f$ in $\mathbb{P}^n(\C)\times\mathbb{P}^1(\C)$ on
$S_c=\mathbf{zero}(\prod_{i=1}^r g_i - c,\ \C^n)$ is finite.
Then a bounded central path exists.
\end{corollary}

\begin{proof}
\textcolor{black}{Analogous to the proof of Proposition~\ref{convergence_of_varieties}, we can show that the general position condition ensures that $S = \overline{S_>}$, and thus the conditions (a) and (c) of Proposition~\ref{existence_of_central_path} hold. Together with the closedness and boundedness assumptions on the local optimal set of $f$ on $S$, Proposition~\ref{existence_of_central_path} guarantees that for every sufficiently small $\mu > 0$ there exists a bounded local minimizer of the log-barrier function in $S_>$. By Theorem~\ref{isolated_critical_points}, these local minimizers are isolated for all sufficiently small $\mu>0$. Thus, using the same semi-algebraic argument as in the proof of Proposition~\ref{critical_path_semi-algebraic_proof}, they contain a bounded continuous germ $x(\mu)$ of local minimizers of the log-barrier function, and thus it is a bounded central path.}
\end{proof}
\subsubsection{Convergence of a critical path}\label{central_path_manifold_correspondence}
\vspace{5px}
\noindent
We recall from Problem~\eqref{unconstrained_cusp} that a bounded critical path (or even a central path) does not necessarily converge to a KKT point (A \PO~may have no KKT point). However, using the implication~\eqref{Compact_central_path_conditions}-\eqref{prod_hyper_surface} and Theorem~\ref{isolated_critical_points}, we can characterize the limit point of a critical path in terms of critical points of $f$ on $S_=$.
\begin{proof}[Proof of Theorem~\ref{limiting_behavior_of_critical_path}]
\textcolor{black}{Recall $\xi(\mu)=\prod_{i=1}^r g_i(x(\mu))$.
As in the proof of Theorem~\ref{isolated_critical_points}, $\xi(\mu)>0$ and $\xi(\mu)\to0$ as $\mu\downarrow0$. By Proposition~\ref{inverse_semi-alg_func},
after possibly restricting to sufficiently small $\mu>0$, $\xi(\mu)$ has a Nash inverse $\mu=\mu(\xi)$. Hence $x(\mu(\xi))$ is a bounded critical point of $f$ on
\begin{align*}
    \zero\left(
        \prod_{i=1}^r g_i-\xi,\R\langle\xi\rangle^n
    \right),
\end{align*}
and Theorem~\ref{convergence_of_critical_points} implies that $\bar{x}=\lim_{\mu\downarrow0}x(\mu)$ is a critical point of $f$ on $S_=$ with respect to its canonical Whitney stratification.}

\vspace{5px}
\noindent
The proof technique for the second part is analogous to~\cite{BM22,BM24}. \textcolor{black}{Notice that $x_{\mu} \in S_>$ and $S_>$ is open in $\R^n$, which imply that $x_\mu$ is locally isolated in $V_\mu$, and thus a singleton connected component of $V_\mu$. To semi-algebraically describe $x_\mu$}, we adopt the Parameterized Bounded Algebraic Sampling~\cite[Algorithm~12.18]{BPR06} to compute points that meet every semi-algebraically connected component of $V_{\mu}$. As a result, we obtain a $(n+2)$-tuple of polynomials $(f',g'_0,\ldots,g'_{n}) \subset \R[\mu,T]^{n+2}$ and a Thom encoding $\sigma$ (see Section~\ref{sign_cond_Thom_encod}) such that for all sufficiently small $\mu >0$ there exists a real root $t_{\sigma}$ of $f'(\mu,T) = 0$ with Thom encoding $\sigma$ such that
\begin{align}\label{parameterized_bounded_sampling}
    x_i(\mu) = \frac{g'_i(\mu,T)}{g'_0(\mu,T)}, \quad i=1,\ldots,n,
\end{align}
where $\deg(f'),\deg(g'_i) =(rd)^{O(n)}$. Now, by applying
Theorem~\ref{14:the:tqe} (Quantifier Elimination) to the quantified formula obtained from~\eqref{parameterized_bounded_sampling}, \textcolor{black}{we get a semi-algebraic description of the graph of the $i^{\mathrm{th}}$ coordinate of $x(\mu)$. More precisely, we obtain a set of polynomials $\{P_1,\ldots,P_n\} \subset \R[\mu,X_i]$ of degree $(rd)^{O(n)}$, such that $P_i(\mu,x_i(\mu)) = 0$ for sufficiently small $\mu > 0$}. Finally, we apply Proposition~\ref{Newton_Puiseux_Thm} to each $P_i$, which implies that $x_i(\mu)-\bar{x}_i$ can be described by a Puiseux series (with coefficients in $\R$) in $\mu$ of order $1/\lambda_i$, where $\lambda_i=(rd)^{O(n)}$. \textcolor{black}{More precisely, for each $i=1,\ldots,n$, we have
\begin{align*}
    |x_i(\mu)-\bar x_i|=O(\mu^{1/\lambda_i}).
\end{align*}
Let $\lambda:=\max_{1\leq i\leq n}\lambda_i$. Then $\gamma=(rd)^{O(n)}$, and
\begin{align*}
    \|x(\mu)-\bar x\|=O(\mu^{1/\gamma}),
\end{align*}
which proves the second part.} 
\end{proof}
\noindent
Theorem~\ref{limiting_behavior_of_critical_path} rests on the canonical Whitney stratification of $S_=$, and assumes that $\mathcal{Q}$ is in general position. However, regardless of the stratification of $S_=$ and existence of KKT points, we can describe the limit of a (possibly unbounded) critical path in terms of projective KKT points of \PO~(see Section~\ref{sec:KKT_PO}).

\vspace{5px}
\noindent
Consider the first-order optimality conditions~\eqref{central_path_original_form} and define $u_i:=\mu / g_i$ for $i=1,\ldots,r$, which gives rise to 
\begin{equation}
\begin{aligned}\label{primal_dual_central_path_form}
\frac{\partial f}{\partial x_j} -  \sum_{i=1}^r u_i \frac{\partial g_i}{\partial x_j} &= 0, \qquad j=1,\ldots,n,\\
    u_i g_i &= \mu.
\end{aligned}
\end{equation}
If we bi-homogenize the polynomial equations in~\eqref{primal_dual_central_path_form} we get
\begin{equation}\label{homogenized_primal_dual_central_path_form}
\begin{aligned}
    u_0 F_j-  \sum_{i=1}^r u_i  G_{ij} &= 0, \qquad j=1,\ldots,n,\\
    u_i g^H_i -\mu u_0 x_0^{\alpha_i}&= 0, \quad i=1,\ldots,r,
\end{aligned}
\end{equation}
where $\alpha_i = \deg(g_i^H)$, $F_j,G_{1j},\ldots,G_{rj} \in \R[X_0,\ldots,X_n]$ and $g_i^H \in \R[X_0,\ldots,X_n]$ are homogeneous polynomials. Now, the following proposition is in order.
\begin{proposition}\label{convergence_to_projective_KKT}
If \PO~has a critical path, then~\eqref{homogenized_primal_dual_central_path_form} has a projective solution $(x'_{\mu},u'_{\mu}) \in \mathbb{P}_n(\R\la \mu \ra) \times \mathbb{P}_r(\R\la \mu \ra)$ and $\lim_{\mu}((x'_{\mu},u'_{\mu}))$ is a projective KKT point of \PO.
\end{proposition}

\begin{proof}
Let $x_\mu$ be a critical path and $(u_\mu)_i = \mu/g_i(x_\mu)$ for $i=1,\ldots,r$. Then it is easy to see that 
\begin{align*}
\Big(\big(1:(x_\mu)_1:\ldots:(x_\mu)_n\big),\big(1:(u_\mu)_1:\ldots:(u_\mu)_r\big)\Big)
\end{align*} 
is a solution of~\eqref{homogenized_primal_dual_central_path_form}. If we multiply 
\begin{align*}
    \big(1:(x_\mu)_1:\ldots:(x_\mu)_n\big)
\end{align*}
by $\mu^{\alpha}$ where $\alpha = \max\{-\min_{i\in\{1,\ldots,n\}}\{\order((x_\mu)_i)\},0\}$ and multiply
\begin{align*}
\big(1:(u_\mu)_1:\ldots:(u_\mu)_r\big)
\end{align*}
by $\mu^{\beta}$ where $\beta =\max\{-\min_{i\in\{1,\ldots,r\}}\{\order((u_\mu)_i)\},0\}$, we obtain a solution   
\begin{align*}
\Big(\big((x'_\mu)_0:(x'_\mu)_1:\ldots:(x'_\mu)_n\big),\big((u'_\mu)_0:(u'_\mu)_1:\ldots:(u'_\mu)_r\big)\Big)
\end{align*}
where $(x'_\mu)_i$ for $i=0,\ldots,n$ and $(u'_\mu)_i$ for $i=0,\ldots,r$ are bounded. By the definition of $\lim_{\mu}$, $\lim_{\mu} ((x'_\mu,u'_\mu))$ is a projective solution of~\eqref{Projective_KKT_conditions_PO}.
\end{proof}
\begin{remark}
The analysis based on projective KKT points in Proposition~\ref{convergence_to_projective_KKT} extends the classical KKT-based framework in \NO~ (e.g., in~\cite{WO2002} or~\cite{FGW02}), which relies on the boundedness of both the critical path and its associated Lagrange multipliers.
\end{remark}

\subsubsection{Smoothness of critical paths at $\mu = 0$}
If any of the conditions of Theorem~\ref{sufficient_conditions_existence} fails, a critical path (when it exists) is not guaranteed to be $\mathcal{C}^{\infty}$-smooth at $\mu =0$. Nevertheless, it is still possible to recover the smoothness using a reparametrization. 

\begin{proof}[Proof of Theorem~\ref{analytic_reparam_semi-algebraic}]
Consider the Puiseux expansion of $x_i(\mu)$ in the proof of Theorem~\ref{limiting_behavior_of_critical_path}, which has a ramification index $(rd)^{O(n)}$. \textcolor{black}{Letting $\rho$ be the least common multiple of all ramification indices over all coordinates ensures the analyticity at $\mu = 0$, and it also yields the bound
\begin{align*}
    \rho \le \prod_{i=1}^n q_i  = (rd)^{O(n^2)}.
\end{align*} }       
\end{proof}

\subsection{Proofs for a definable critical path}\label{proof_of_o-minimal}
We extend the existence, convergence, and smoothness of critical paths to \NO~problems 
\begin{align}\label{definable_optim}
\inf_x\{f(x) \mid g_i(x) \
\ge 0, \ \ i=1,\ldots,r\},
\end{align}
\noindent
where $f,g_i$ are definable functions in an o-minimal structure $\mathcal{S}(\mathbb{R})$. As the definition of a critical path (Definition~\ref{critical_path_def}) relies on differentials of $f,g_i$, we assume that $f$ and each $g_i$ are $\mathcal{C}^k$-smooth for some $k \in \mathbb{Z}_+$.

\vspace{5px}
\noindent
First, we prove that in this setting, a critical path is a definable function in $\mathcal{S}(\mathbb{R})$. 
\begin{proposition}\label{critical_path_definable_proof}
Let $\mathcal{S}(\mathbb{R})$ be an o-minimal structure, and suppose that $f,g_i \in \mathcal{S}(\mathbb{R})$ are $\mathcal{C}^1$-smooth definable functions. Then a critical path is a $\mathcal{C}^k$-smooth definable function in $\mathcal{S}(\mathbb{R})$ for 
every $k \in \mathbb{Z}_+$. If $\mathcal{S}(\mathbb{R}) = \mathbb{R}_{\mathrm{an,exp}}$, then the critical path is analytic. 
\end{proposition}
\begin{proof}
The proof for the first part is analogous to Proposition~\ref{critical_path_semi-algebraic_proof}, when ``semi-algebraic" is replaced by ``definable". Since $x(\mu)$ is a definable function, there exists a 
cell decomposition of $\mathbb{R}$, with is a
$\mathcal{C}^k$-cell decomposition for 
every positive integer $k$, such that the restriction of $x(\mu)$ to each interval is  $\mathcal{C}^k$-smooth~\cite[Theorem~6.7]{Michel2}, implying that $x(\mu)$ is $\mathcal{C}^k$-smooth for all sufficiently small $\mu>0$.

\vspace{5px}
\noindent
The proof of analyticity for $\mathbb{R}_{\mathrm{an,exp}}$ follows from~\cite[Theorem~8.8]{Dries7}. 
\end{proof}

\subsubsection{Existence of a critical path}
The analogs of Theorems~\ref{Morse_on_smooth_hypersurface}-\ref{Morse_on_transversal} are still valid using the Definable Implicit Function Theorem~\cite[Page~113]{Dries2}. 

\begin{proof}[Proof of Theorem~\ref{sufficient_conditions_existence_definable}]
Replace ``semi-algebraic" by ``definable" in the proof of Theorem~\ref{sufficient_conditions_existence} and then apply the Definable Implicit Function Theorem.
\end{proof}

\subsubsection{Convergence of a critical path}
Analogous to the semi-algebraic case, if a critical path is bounded near $\mu = 0$, then it converges. 
\begin{proposition}\label{convergence_definable}
Let $\mathcal{S}(\mathbb{R})$ be an o-minimal structure, and suppose that $f,g_i \in \mathcal{S}(\mathbb{R})$ are $\mathcal{C}^1$-smooth definable functions. Then a critical path $x(\mu)$, uniformly bounded near $\mu = 0$, converges to some $\hat{x} \in S_= \cap S$. 
\end{proposition}

\begin{proof}
By Proposition~\ref{Monoton_Thm}, there exists a sufficiently small $\bar{\mu}>0$ such that each coordinate of $x(\mu)$ is either \textcolor{black}{constant, or continuous and strictly monotone} on $(0,\bar{\mu})$. Since each coordinate of $x(\mu)$ is bounded, then they have limit on interval $(0,\bar{\mu})$. \textcolor{black}{The rest of the proof follows from the fact that $S$ is closed and Assumption~\ref{non_trivial_problem}}.  
\end{proof}

\vspace{5px}
\noindent
Now, we show that the analog of Theorem~\ref{limiting_behavior_of_critical_path} holds in the o-minimal setting, under the assumption that $f,g_i$ are $\mathcal{C}^1$-smooth definable functions. By the Definable Sard Theorem~\cite[Theorem~2.7]{Wilkie2}, 
\begin{align}\label{definable_non-singular}
\bigg\{x \in \mathbb{R}^n \mid \prod_{i=1}^r g_i(x) = \xi(\mu) \bigg\},
\end{align}
is a definable $\mathcal{C}^1$- manifold when $\mu>0$ is sufficiently small (recall that $\xi(\mu) \downarrow 0$ as $\mu \downarrow 0$). Therefore, $x(\mu)$ can be considered as a critical point of $f$ on ~\eqref{definable_non-singular} for all sufficiently small $\mu > 0$. First, we prove an o-minimal version of Proposition~\ref{convergence_of_tangent_spaces_extended}. To this end, we define the convergence of tangent spaces using the usual topology on Grassmannian $\mathrm{Gr}_{n-1}(\mathbb{R}^n)$ of linear subspaces of $\mathbb{R}^n$.

\begin{definition}[Convergence of tangent spaces]
Given the tangent space $T_{x(\xi)} V_{\xi}$, we define
\begin{align*}
 \lim_{\xi \to 0} (T_{x(\xi)} V_{\xi}) = T,   
\end{align*}
when for any sequence $\{\xi_k\} \to 0$, there exists an orthonormal basis $\{v_1^k,\ldots,v^k_{n-1}\}$ of $T_{x(\xi_k)} V_{\xi_k}$ such that $v_i^k \to v_i$ for each $i$, and $\{v_1,\ldots,v_{n-1}\}$ is an orthonormal basis of $T$.       
\end{definition}

\begin{proposition}
\label{convergence_of_tangent_spaces_definable}
\textcolor{black}{Let $\mathcal S(\mathbb{R})$ be an o-minimal structure, and let $V=\zero\left(\prod_{i=1}^s f_i,\mathbb{R}^n\right)$, where $f_1,\ldots,f_s$ are $\mathcal{C}^1$-smooth definable functions in $\mathcal S(\mathbb{R})$ and $\{f_1,\ldots,f_s\}$ is in general position.
Let $x:(0,a)\to\mathbb{R}^n$ be a bounded definable function such that $x(\xi)\in V_\xi:=
\zero\left(\prod_{i=1}^s f_i-\xi,\mathbb{R}^n\right)$ for every $\xi\in(0,a)$ and $\lim_{\xi\downarrow0}x(\xi)=\bar x$. Further, define
\[
I:=\{i\in\{1,\ldots,s\}\mid f_i(\bar x)=0\}.
\]
If $Z$ is the stratum of $V$ containing $\bar x$, with respect
to the canonical Whitney stratification of $V$, then
\begin{align*}
T_{\bar x}Z
\subset
\lim_{\xi\downarrow0}T_{x(\xi)}V_\xi.
\end{align*}
}
\end{proposition}

\begin{proof}
\textcolor{black}{The proof is analogous to Proposition~\ref{convergence_of_tangent_spaces_extended}. By the Definable Sard Theorem, after possibly restricting the interval $(0,a)$, $V_\xi$ is a $\mathcal{C}^1$-smooth hypersurface for every $\xi\in(0,a)$. Hence,
\begin{align*}
T_{x(\xi)}V_\xi
=
\left\{
h\in\mathbb{R}^n\ \middle|\
\sum_{i=1}^s
\prod_{\substack{\ell=1\\ \ell\neq i}}^s
f_\ell(x(\xi))\bigg(\sum_{j=1}^n    \frac{\partial f_{i}}{\partial x_j}(x(\xi))h_j\bigg)=0
\right\}.
\end{align*}
\noindent
For $i=1,\ldots,s$, define $a_i(\xi)
:=
\prod_{\substack{\ell=1\\ \ell\neq i}}^s
f_\ell(x(\xi))
$
and $\lambda_i(\xi)
:=a_i(\xi)/\sqrt{\sum_{k=1}^s a_k(\xi)^2}$. Since $\prod_{i=1}^s f_i(x(\xi))=\xi>0$, all $f_i(x(\xi))$ are non-zero, and hence the above quantities are well-defined. Further, $\|\lambda(\xi)\|=1$, where $\lambda(\xi):=
(\lambda_1(\xi),\ldots,\lambda_s(\xi))$. Since $\lambda$ is bounded and definable, after possibly
restricting $(0,a)$ further, the limit $\bar\lambda
:=
\lim_{\xi\downarrow0}\lambda(\xi)$ exists, and $\|\bar\lambda\|=1$.} 

\vspace{5px}
\noindent
Notice that for every $i\in I$ and $j\notin I$, we obtain $\lim_{\xi\downarrow0} a_j(\xi)/a_i(\xi)=0$, which implies $\bar\lambda_j=0$ for $j\not \in I$. Therefore,
\begin{align*}
T_{x(\xi)}V_\xi
=
\left\{
h\in\mathbb{R}^n\ \middle|\
\sum_{i=1}^s
\lambda_i(\xi)\bigg(\sum_{j=1}^n    \frac{\partial f_{i}}{\partial x_j}(x(\xi))h_j\bigg)=0
\right\},
\end{align*}
and, by continuity of the differentials,
\begin{align*}
\lim_{\xi\downarrow0}T_{x(\xi)}V_\xi
=
\left\{
h\in\mathbb{R}^n\ \middle|\
\sum_{i\in I}
\bar\lambda_i\bigg(\sum_{j=1}^n    \frac{\partial f_{i}}{\partial x_j}(\bar{x})h_j\bigg)=0
\right\},
\end{align*}
which follows from $J(\{f_i\}_{i \in I})(\bar{x})$ being full row rank and the Definable Implicit Function Theorem. \textcolor{black}{Notice that the limiting co-vector $\sum_{i \in I} \bar{\lambda}_i df_i(\bar{x})$ is non-zero, due to $\|\bar{\lambda}\|=1$ and the independence of the active differentials}.
Finally, by the description of the canonical Whitney
stratification under the general position assumption,
\begin{align*}
T_{\bar x}Z
=\bigcap_{i \in I} T_{\bar{x}} \zero(f_{i},\mathbb{R}^n).
\end{align*}
Thus, for every $h\in T_{\bar x}Z$,
\begin{align*}
\sum_{i\in I}
\bar\lambda_i\bigg(\sum_{j=1}^n    \frac{\partial f_{i}}{\partial x_j}(\bar{x})h_j\bigg)=0,
\end{align*}
which implies $T_{\bar x}Z
\subset
\lim_{\xi\downarrow0}T_{x(\xi)}V_\xi$, as required.
\end{proof}
\noindent
Now, we prove Theorem~\ref{limiting_behavior_of_definable_critical_path}. For the quantitative part of the theorem, we need to assume that $\mathcal{S}(\mathbb{R})$ is a polynomially bounded o-minimal structure, and $f,g_i$ are definable functions in $\mathcal{S}(\mathbb{R})$. Further, we will need a result on the H\"older continuity of definable functions in $\mathcal{S}(\mathbb{R})$.

\begin{proposition}[H\"older inequality~\cite{Dries2}]\label{Holder_ineq}
Let $\mathcal{S}(\mathbb{R})$ be a polynomially bounded o-minimal structure, $A \subset \mathbb{R}$ be a compact set in $\mathcal{S}(\mathbb{R})$, and let $f:A \to \mathbb{R}$ be a continuous definable function in $\mathcal{S}(\mathbb{R})$. Then there exist $C,r > 0$ such that 
\begin{align*}
    |f(x) - f(y)| \le C |x-y|^r, \quad x,y \in A.
\end{align*}
\end{proposition}

\begin{proof}[Proof of Theorem~\ref{limiting_behavior_of_definable_critical_path}]
\textcolor{black}{Let $\xi(\mu):=\prod_{i=1}^r g_i(x(\mu))$. As in the proof of Proposition~\ref{convergence_definable}, we have $\xi(\mu)>0$ and $\xi(\mu)\to0$ as $\mu\downarrow0$. By the Monotonicity Theorem, after possibly restricting to
sufficiently small $\mu>0$, the definable function $\xi(\mu)$ is strictly monotone and hence has a definable inverse $\mu=\mu(\xi)$. Therefore, $x(\mu(\xi))$ is a bounded definable critical point of $f$ on
\begin{align*}
    \left\{x\in\mathbb{R}^n\ \middle|\ 
    \prod_{i=1}^r g_i(x)=\xi\right\}.
\end{align*}
Now replace Proposition~\ref{convergence_of_tangent_spaces_extended} by Proposition~\ref{convergence_of_tangent_spaces_definable} in the proof of
Theorem~\ref{convergence_of_critical_points} and apply the result to $f$ and $S_=$.}

\vspace{5px}
\noindent
For the quantitative part, we apply the H\"older inequality in Proposition~\ref{Holder_ineq} to the definable function 
\begin{align*}
 d(\mu):[0,\bar{\mu}] \to \mathbb{R}:\mu \mapsto \|x(\mu)-\bar{x}\|^2,   
\end{align*}
where $d(0) = 0$, and $\bar{\mu} > 0$ is small enough to ensure continuity of $d$. Then, by Proposition~\ref{Holder_ineq}, for all $\mu \in [0,\bar{\mu}]$ there exist $C,N' > 0$ such that 
\begin{align*}
d(\mu) \le C \mu^{N'}.
\end{align*}
Taking
\begin{align*}
    N:=\min\left\{\frac{N'}{2},1\right\},
\end{align*}
we obtain
\begin{align*}
    \|x(\mu)-\bar x\|=O(\mu^N),
    \qquad N\in(0,1].
\end{align*}
\noindent
\textcolor{black}{When $\mathcal{S}(\mathbb{R})=\mathbb{R}_{\mathrm{an}}$, Proposition~\ref{Puiseux_subanalytic} applied to $d$ shows that its leading exponent is rational. Thus the above
argument yields $N\in\Q\cap(0,1]$ (and if $d$ is eventually zero, we may take $N=1$).}
\end{proof}

\subsubsection{Smoothness of a critical path at the limit point}
Finally, we establish the analog of Theorem~\ref{analytic_reparam_semi-algebraic} for~\eqref{definable_optim}. To that end, we need to assume that $\mathcal{S}(\mathbb{R})$ is a polynomially bounded o-minimal structure, and $f,g_i$ are definable functions in $\mathcal{S}(\mathbb{R})$. Further, we will need a growth dichotomy result for definable functions in $\mathcal{S}(\mathbb{R})$ due to Miller~\cite[Page~258]{Miller94(b)}.

\begin{proposition}[Growth dichotomy~\cite{Miller94(b)}]\label{growth_dichotomy}
Let $\mathcal{S}(\mathbb{R})$ be a polynomially bounded o-minimal structure, $f:\mathbb{R} \to \mathbb{R}$ be a definable function in $\mathcal{S}(\mathbb{R})$, and $f$ be not ultimately identically zero. Then there exists $r \in \mathbb{R}$ such that \textcolor{black}{the power function $x \mapsto x^r:(0,\infty) \to \mathbb{R}$ is definable in $\mathcal{S}(\mathbb{R})$ and} $\lim_{x \to \infty} (f(x)/x^r)$ exists, and it is non-zero. 
\end{proposition}
\noindent
We also use the following Puiseux-type expansion from~\cite{P84} for the globally sub-analytic functions in $\mathbb{R}_{\mathrm{an}}$  (see also~\cite[Lemma~2.6]{Kur88}).
\begin{proposition}[Lemma~2.6 in~\cite{Kur88}]\label{Puiseux_subanalytic}
Let $f:[0,\delta) \to \mathbb{R}$ be in $\mathbb{R}_{\mathrm{an}}$. Then there exist $\varepsilon > 0$, $q \in \mathbb{Z}_+$, and a real analytic function $h(x) = \sum_{i=0}^{\infty} \alpha_i x^i$ in a neighborhood of $0$ such that $f(x) = h(x^{1/q})$ for all $x \in [0,\varepsilon)$.
\end{proposition}
\begin{remark}
A Puiseux type expansion also exists for definable functions in $\mathbb{R}_{\mathrm{an}}^{\mathbb{R}}$~\cite[Proposition~4.5]{Miller94}: there exist $\ell \in \mathbb{N}$, a convergent power series $F \in \mathbb{R}\{X_1,\ldots,X_{\ell}\}$ with $F(\mathbf{0}) \neq 0$ and $r_0,\ldots,r_{\ell} \in \mathbb{R}$ with $r_1,\ldots,r_{\ell} >0$ such that $f(x) = x^{r_0} F(x^{r_1},\ldots,x^{r_{\ell}})$ for all sufficiently small positive $x$. In this case, however, $r_0,\ldots,r_{\ell}$ do not need to be rational. 
\end{remark}
\vspace{5px}
\noindent
Now, we use Propositions~\ref{growth_dichotomy}-\ref{Puiseux_subanalytic} to prove an o-minimal analog of Theorem~\ref{analytic_reparam_semi-algebraic}.

\begin{proof}[Proof of Theorem~\ref{thm:analyticity_definable_func}]
Since the map $x\mapsto 1/x$ is definable in
$\mathcal{S}(\mathbb{R})$, Proposition~\ref{growth_dichotomy} also implies that if $h:(0,\varepsilon)\to\mathbb{R}$ is an eventually non-zero definable function, then there exist $c\neq 0$ and $r\in\mathbb{R}$ such that $x\mapsto x^r$ is definable and
\begin{align*}
h(x)=cx^r+o(x^r)
\qquad\text{as }x\downarrow0.
\end{align*}
\noindent
 \textcolor{black}{Let $\bar{x} = \lim_{\mu \downarrow 0} x(\mu)$. Then applying the above result to $y_i(\mu):=x_i(\mu)-\bar{x}_i$,  for every coordinate $i$ which is not eventually constant there exist $c_i\in\mathbb{R}\setminus \{0\}$ and $r_i > 0$ such that   
\begin{align}\label{critical_path_asymptotic_behavior}
    y_i(\mu) = c_i\mu^{r_i} + o(\mu^{r_i}),
\end{align}
or equivalently, $y_i(\mu)/c_i\mu^{r_i} \to 1$. We also note from Proposition~\ref{growth_dichotomy} that the power function $\mu\mapsto \mu^{r_i}$ is definable.}

\vspace{5px}
\noindent
\textcolor{black}{Choose $\rho\in\Z_+$ such that
\begin{align*}
\rho r_i>k
\end{align*}
for every $i$ for which $y_i$ is not eventually zero; for example,
it suffices to take
\begin{align*}
\rho>
    \frac{k}{
    \min\{r_i:y_i\text{ is not eventually zero}\}}.
\end{align*}
For each such coordinate, define $z_i(t):=y_i(t^\rho)$. Then $z_i$ is definable and,~by \eqref{critical_path_asymptotic_behavior},
\begin{align*}
z_i(t)=c_i t^{\rho r_i}
    +o(t^{\rho r_i}).
\end{align*}}

\vspace{5px}
\noindent
\textcolor{black}{Since the germs of univariate definable functions in an o-minimal expansion of the real field form a Hardy field~\cite[Section~3]{Miller12} (see also~\cite{R73}), the l'H\^{o}pital's rule can be applied in
both directions for such germs. Consequently, for every integer $1\leq j\leq k$, we have
$j<\rho r_i$ and
\begin{align*}
\frac{z_i^{(j)}(t)}{c_i(\rho r_i)(\rho r_i-1)\cdots
        (\rho r_i-j+1)t^{\rho r_i-j}
    }
    \longrightarrow 1
    \qquad\text{as }t\downarrow0.
\end{align*}
Therefore,
\begin{align*}
z_i^{(j)}(t)=
    c_i(\rho r_i)(\rho r_i-1)\cdots
    (\rho r_i-j+1)t^{\rho r_i-j}
    +o(t^{\rho r_i-j}).
\end{align*}
Since $\rho r_i-j>0$ for $1\leq j\leq k$, it follows that $\lim_{t\downarrow0}z_i^{(j)}(t)=0$ for every $j=1,\ldots,k$. Thus, after defining $z_i(0)=0$, the function $z_i$ extends to
a $\mathcal{C}^k$-smooth function at $t=0$. The coordinates which are
eventually constant clearly present no difficulty. Hence, after
defining $x(0)=\bar x$, the reparametrized path $t\longmapsto x(t^\rho)$
extends to a $\mathcal{C}^k$-smooth function at $t=0$.}

\vspace{5px}
\noindent
\textcolor{black}{For the second part, suppose that $f,g_i\in\mathbb{R}_{\mathrm{an}}$.
Applying Proposition~\ref{Puiseux_subanalytic} to each coordinate $x_i$, there exist
$q_i\in\Z_+$ and a real analytic function $h_i$ in a neighborhood
of $0$ such that $x_i(\mu)=h_i(\mu^{1/q_i})$ 
for all sufficiently small $\mu\geq0$. Let
$\rho$ be the least common multiple of $q_1,\ldots,q_n$. Then
\begin{align*}
x_i(t^\rho)=h_i(t^{\rho/q_i}),
    \qquad i=1,\ldots,n.    
\end{align*}
\noindent
Since $\rho/q_i\in\Z_+$, every coordinate $x_i(t^\rho)$ is
analytic at $t=0$. Hence $x(t^\rho)$ is analytic at $t=0$.}
\end{proof}

\begin{remark}\label{Extension_of_ Peterzil}
\textcolor{black}{The analytic conclusion of Theorem~\ref{thm:analyticity_definable_func} also applies when $f,g_1,\ldots,g_r$ are real globally analytic functions. Indeed, since $x(\mu)$ is bounded, there exist $\mu_0>0$ and bounded boxes $B\subset B'\subset\mathbb{R}^n$ such that
\begin{align*}
    x(\mu)\in B,\qquad 0<\mu<\mu_0,
\end{align*}
and $\overline{B}\subset B'$. Since $f,g_1,\ldots,g_r$ are real
globally analytic, their restrictions to $B'$, after an affine rescaling, are restricted analytic functions and hence are
definable in $\mathbb{R}_{\mathrm{an}}$. Consequently, the set of $(\mu,x)\in(0,\mu_0)\times B$ satisfying
the first-order conditions \eqref{critical_path_algebraic} is
definable in $\mathbb{R}_{\mathrm{an}}$. Since $x(\mu)$ is isolated
in the corresponding fiber for every sufficiently small $\mu>0$,
the same cell-decomposition argument as in the proof of Proposition~\ref{critical_path_definable_proof} shows, after possibly
decreasing $\mu_0$, that $x(\mu)$ is definable in
$\mathbb{R}_{\mathrm{an}}$. In particular, by the Monotonicity
Theorem and boundedness,
\begin{align*}
    \bar x:=\lim_{\mu\downarrow0}x(\mu)
\end{align*}
exists.}

\textcolor{black}{
Therefore, the argument in the second part of
Theorem~\ref{thm:analyticity_definable_func} applies and yields
$\rho\in\mathbb{Z}_+$ such that $x(\mu^\rho)$ is analytic at
$\mu=0$. Moreover, this implies that $x(\mu)$ itself is analytic
for all sufficiently small $\mu>0$, since
\begin{align*}
x(\mu)=\bigl[x(t^\rho)\bigr]_{t=\mu^{1/\rho}},
\end{align*}
and $\mu\mapsto\mu^{1/\rho}$ is analytic on $(0,\mu_0)$.
Thus, in the globally analytic setting, the conclusion is stronger than the $\mathcal{C}^k$-smoothness obtained in
Proposition~\ref{critical_path_definable_proof}. It is also consistent with~\cite[Remark~1]{GP02}, where the central path of a convex \SDO~problem with analytic data is shown to be definable in
$\mathbb{R}_{\mathrm{an}}$.}
\end{remark}

\section{Concluding remarks}
In this paper, we studied critical points of a polynomial $F \in \R[X_1,\ldots,X_n]$
on the algebraic set $V_{\xi} = \zero(\{P_1-\xi_1,\ldots,P_s-\xi_s\},\R\la \xi \ra^n)$, where $\{P_1,\ldots,P_s\} \subset \R[X_1,\ldots,X_n]$ is a finite set of polynomials. We proved different sets of conditions - based on homogenization $P_i^H$ of $P_i$ and generic properties of the projective zeros of $P_i^H$ - that guarantee existence (Corollaries~\ref{bounded_critical_points}-\ref{existence_critical_points} and Theorems~\ref{Morse_on_smooth_hypersurface}-\ref{Morse_on_transversal}), boundedness (Theorems~\ref{bounded_fibers},~\ref{bounded_connected_component},~\ref{Morse_on_smooth_hypersurface}-\ref{Morse_on_transversal}), finiteness (Theorems~\ref{constrained_real_isolated_critical_points}-\ref{constrained_real_non-degenerate_critical_points}), and non-degeneracy (Theorems~\ref{constrained_real_non-degenerate_critical_points}-\ref{Morse_on_smooth_hypersurface}) of critical points. Furthermore, we characterized (Theorem~\ref{convergence_of_critical_points}) the limit of critical points in terms of critical points of $F$ on $V = \zero(\mathcal{P},\R^n)$ with respect to its canonical Whitney stratification.  

\vspace{5px}
We applied our theoretical results to the log-barrier function and critical paths of~\PO, as a special case of the problem considered in the first part. This led to new conditions for the existence and convergence  (Theorems~\ref{critcal_path_existence}-\ref{sufficient_conditions_existence}) of critical paths. Additionally, we characterized the limit of a bounded critical path (Theorem~\ref{limiting_behavior_of_critical_path}), and using the Quantifier Elimination and the Newton-Puiseux theorems, we quantified the convergence rate of critical paths. The Newton-Puiseux Theorem also yields a reparametrization $\mu \mapsto \mu^{\rho}$, for some positive integer $\rho$ (Theorem~\ref{analytic_reparam_semi-algebraic}), under which a critical path is $\mathcal{C}^{\infty}$-smooth at $\mu = 0$. 

\vspace{5px}
Finally, we established conditions for the existence and convergence (Theorem~\ref{sufficient_conditions_existence_definable}) of critical paths of \NO~problems involving definable sets and functions in an o-minimal structure, preserving the tameness properties of semi-algebraic structures. Analogous to~\PO, we characterized (Theorem~\ref{limiting_behavior_of_definable_critical_path}) the limit of a bounded critical path, and using the H\"older inequality, we quantified its convergence rate. Further, as an abstraction of the notion of reparametrization for a critical path, we proved (Theorem~\ref{thm:analyticity_definable_func})  that when the o-minimal structure is polynomially bounded, a bounded critical path can be reparametrized to establish $\mathcal{C}^k$-smoothness at $\mu = 0$ for any order $k > 0$. As a result of the Puiseux-type expansions for globally sub-analytic functions, we obtain a stronger result for $\mathbb{R}_{\mathrm{an}}$: a critical path admits an analytic reparametrization $\mu \mapsto \mu^{\rho}$ for some positive integer $\rho$.  

\vspace{5px}
\textcolor{black}{We end this section with some discussions and a few open problems.}

\subsection{Beyond right-hand-side perturbations}
\label{sec:beyond-rhs-perturbations}
\textcolor{black}{While Theorems~\ref{sufficient_conditions_existence_definable}-\ref{thm:analyticity_definable_func} extend some of our results to
definable sets and functions, they retain the same basic perturbation
scheme, namely, perturbations of the right-hand sides of the defining
constraints. More generally, Problems~\ref{existence_critical_path}-\ref{boundedness_critical_path} may be formulated for critical points of
\(F\) on arbitrary \(\mathcal{P}\)-semi-algebraic sets, where
\begin{align*}
 \mathcal{P}\subset \R\langle\xi_1,\ldots,\xi_s\rangle_b
    [X_1,\ldots,X_n].   
\end{align*}
\noindent
This formulation permits a more general algebraic dependence of the
defining polynomials on the perturbation parameters than the
right-hand-side perturbations considered in this paper. Nevertheless,
our results do not extend to this setting automatically, since several
arguments use the particular structure of the perturbed equations and
the lexicographic hierarchy
\begin{align*}
    0<\xi_s\ll\cdots\ll\xi_1\ll 1.
\end{align*}}
\begin{example}
The conclusion of Theorem~\ref{bounded_fibers} does not extend automatically to arbitrary perturbations with coefficients in
$\R\langle\xi\rangle_b$. For example, consider $P = X$ and $P_\xi=X-\xi X^2 - \xi \in \mathbb{R}\la \xi \ra_b[X]$. Although $\lim_\xi P_\xi = P$ and $P$ has no zero at infinity, $P_\xi$ has an unbounded zero over $\mathbb{R}$, failing the extension of Theorem~\ref{bounded_fibers} to arbitrary infinitesimal deformations.
\end{example}
\noindent
A further extension obtained by replacing the iterated real closed field
$\R\langle\xi_1,\ldots,\xi_s\rangle$ with the ring of arbitrary
multi-variate convergent power series, or with its fraction field, would
require substantially different arguments. The ring of convergent power
series is not a real closed field, and neither it nor its fraction field
has all of the closure and transfer properties used in our proofs.
Moreover, a multi-variate analytic germ does not, by itself, prescribe a
relative order of vanishing among the perturbation parameters and hence
does not encode the above lexicographic hierarchy. Different analytic
arcs approaching the origin may therefore select different limiting
geometric configurations or different branches of critical points.

\vspace{5px}
\noindent
For a single perturbation parameter over $\R$, this difficulty
is less pronounced: algebraic Puiseux series represent convergent
semi-algebraic germs on a sufficiently small interval. In contrast, an
extension to arbitrary multi-variate analytic perturbations would require
analytic-geometric substitutes for the Tarski--Seidenberg Transfer
Principle and for the effective algebraic tools used in this paper,
together with a uniform treatment of the possible relative rates at
which the perturbation parameters tend to zero.

\subsection{Existence and convergence in the presence of singularities}
In addition to their applications to critical paths, Corollaries~\ref{bounded_critical_points}-\ref{existence_critical_points} and Theorems~\ref{bounded_fibers}-\ref{convergence_of_critical_points} answer key questions in computational optimization and perturbation analysis of equality constrained \PO~problems. For instance, when a $g_i$ has a singular zero or $\mathcal{Q}$ is not in general position,~\PO~may have no KKT solution. In such cases, one may want to slightly perturb the original problem to
\begin{align}\label{perturbed_PO}
\inf_x\{f(x) \mid g_i(x) = \xi_i, \quad i=1,\ldots,r\},
\end{align}
with $\xi$ being a sufficiently small positive value, to restore KKT points or quadratic convergence of Newton's method. However, it is important to understand when the perturbed problem~\eqref{perturbed_PO} has a critical point, whether the critical points converge, and how to characterize the limit point. Toward this end, establishing weaker existence conditions (than those given in Corollaries~\ref{bounded_critical_points}-\ref{existence_critical_points} and Theorems~\ref{Morse_on_smooth_hypersurface}-\ref{Morse_on_transversal}) for the critical points would be highly desired. Moreover, it is worthwhile to investigate the behavior of critical points in the presence of singularities. In Theorem~\ref{convergence_of_critical_points} (and its o-minimal version), we proved that when $\mathcal{P}$ is in general position, the limits of critical points of $F$ on $V_{\xi}$ are critical points of $F$ on $V$ with respect to its canonical Whitney stratification. However, additional complications arise when $\mathcal{P}$ is not in general position, mostly due to a more complicated stratification of $V$, which renders the proof technique of Theorem~\ref{convergence_of_critical_points} inapplicable.

\subsection{Boundedness and finiteness of critical points for the o-minimal case}
We derived conditions for the existence and convergence of definable critical paths, and we characterized the limit point, using the o-minimal analogs of Theorems~\ref{Morse_on_transversal} and~\ref{convergence_of_critical_points}. However, o-minimal counterparts of Corollaries~\ref{bounded_critical_points}-\ref{existence_critical_points} and Theorems~\ref{bounded_fibers}-\ref{constrained_real_non-degenerate_critical_points} are not currently available, limiting the direct applicability of our approach in certain settings. Given an o-minimal structure $\mathcal{S}(\mathbb{R})$ (possibly polynomially bounded), it is of interest to establish conditions that guarantee existence, finiteness, and boundedness of critical points of a definable $f \in \mathcal{S}(\mathbb{R})$ on the definable set $\zero(\{g_1-\xi_1,\ldots,g_r-\xi_r\},\mathbb{R}^n)$, where $g_i \in \mathcal{S}(\mathbb{R})$. Furthermore, an o-minimal analog of the extension of Theorem~\ref{convergence_of_critical_points} to settings involving singularities is a compelling avenue for further investigation.
\subsection{Strict complementarity of projective KKT points}
Convergence of a critical path to a projective KKT point, as established in Proposition~\ref{convergence_to_projective_KKT}, raises a natural question about the $\mathcal{C}^{\infty}$-smoothness of the critical path at $\mu = 0$. It is well-known that the existence of a strictly complementary optimal solution is both necessary and sufficient for the analyticity of the central path of \SDO~at $\mu = 0$~\cite{GS98,H02}. \textcolor{black}{It would be interesting to investigate
whether the notion of strict complementarity for projective KKT points
introduced in Section~\ref{sec:KKT_PO} yields at least a
sufficient condition for the $\mathcal C^\infty$-smoothness of a critical
path at $\mu=0$. This conjecture is consistent with Example~\ref{non_analytic_at_0}, in which
the central path is not $\mathcal C^\infty$-smooth at $\mu=0$ and its
limiting projective KKT point fails the strict complementarity condition.}

\section*{Acknowledgments} 
\noindent
We thank the anonymous referees for their detailed comments which helped us to improve the paper. 
The authors were supported by the NSF grant CCF-2128702 while working on the initial version of this paper.

\bibliographystyle{abbrv}
\bibliography{mybibfile}
\end{document}